\documentclass[10pt, a4paper]{article}
\usepackage{titling}
\title{Finite element approximation of \protect\\ an anisotropic porous medium equation \protect\\ with fractional pressure}
\newcommand{\titlePDF}{Finite Element Scheme for the Fractional Porous Medium Equation
with Fractional Pressure}
\newcommand{\authorPDF}{Carrillo, Fronzoni, Suli.}
\newcommand{\subjectPDF}{}
\newcommand{\keywordsPDF}{}

 \usepackage{authblk}
\usepackage{verbatim}

\author[1]{ Stefano Fronzoni}

\affil[1]{
	Mathematical Institute, University of Oxford
}
\affil[ ]{
	OX2 6GG Oxford, United Kingdom
}
\affil[ ]{\textit{ stefano.fronzoni@maths.ox.ac.uk,}}

\usepackage[left=0.7in, right=0.7in, top=1.1in, bottom=1.1in]{geometry}

\makeatletter
\let\newtitle\@title
\let\newauthor\@author
\let\newdate\@date
\makeatother
\usepackage{fancyhdr}
\usepackage[
	hypertexnames=false,
	pdftex,
	pdftitle={\titlePDF. \authorPDF},
	pdfauthor={\authorPDF},
	pdfsubject={\subjectPDF},
	pdfkeywords={\keywordsPDF},
]{hyperref}
\hypersetup{
	colorlinks=true,
	linkcolor=blue,
	citecolor=blue,
	urlcolor=blue,
	linktocpage
}

\usepackage[utf8]{inputenc}
\usepackage[T1]{fontenc}
\usepackage{amsmath} 
\usepackage{amsthm}
\usepackage{amsfonts}
\usepackage{amssymb}
\usepackage{graphicx}
\usepackage{mathtools}
\usepackage{enumitem}
\usepackage{ifthen}
\usepackage{dsfont}
\usepackage{float}
\usepackage{lettrine}
\usepackage{lipsum}
\usepackage{subcaption}
\usepackage{appendix}

\usepackage[UKenglish]{babel}

\usepackage[dvipsnames]{xcolor}

\definecolor{ppGreen}{HTML}{008000}
\definecolor{ppBlue}{HTML}{0000FF}
\definecolor{ppRed}{HTML}{FF0000}
\definecolor{ppPurple}{HTML}{800080}
\definecolor{lightblue}{rgb}{0.145,0.6666,1}
\definecolor{grey52}{RGB}{52,52,52}
\definecolor{color1}{RGB}{0,62,116}
\definecolor{color2}{RGB}{152,152,152}
\definecolor{color3}{RGB}{52,52,52}
\definecolor{color4}{RGB}{100,100,100}

\definecolor{imperialnavy}{RGB}{0,33,71}
\definecolor{imperialblue}{RGB}{0,62,116}
\definecolor{imperialgrey}{RGB}{235,238,238}
\definecolor{imperialcoolgrey}{RGB}{157,157,157}

\newcounter{review}

\newcommand{\ntcreview}[3]{\refstepcounter{review}{\color{#2}{\textbf{[#1]}: #3}}}

\newcommand{\creview}[3]{\ntcreview{#1}{#2}{#3}\addcontentsline{tor}{subsection}{\thereview~\textbf{[#1]}:~#3
	}}

\newcommand{\review}[2]{\creview{#1}{blue}{#2}}

\makeatletter

\newcommand\listreviewname{List of Reviews}
\newcommand\listofreviews{\section*{\listreviewname}\@starttoc{tor}}
\makeatother

\usepackage[all]{nowidow}

\newcommand{\subjectclassification}[1]{

	{\small\textbf{\textit{AMS Subject Classification --- }} #1}

}
\newcommand{\keywords}[1]{

	{\small\textbf{\textit{Keywords --- }} #1}

}

\usepackage{tikz}

\usepackage{pgfplots}
\pgfplotsset{compat=1.16}

\usepackage{float}
\usepackage{subcaption}
\usepackage[section]{placeins}
\usepackage{rotating}

\usepackage{array}
\newcolumntype{L}[1]{>{\raggedright\let\newline\\\arraybackslash\hspace{0pt}}m{#1}}
\newcolumntype{C}[1]{>{\centering\let\newline\\\arraybackslash\hspace{0pt}}m{#1}}
\newcolumntype{R}[1]{>{\raggedleft\let\newline\\\arraybackslash\hspace{0pt}}m{#1}}
\usepackage{booktabs}
\usepackage{tabularx}
\usepackage{multirow}

\usepackage{titling}

\usepackage[algoruled]{algorithm2e}

\usepackage[capitalise]{cleveref}

 \usepackage{accents}

\newcommand\term\emph

\numberwithin{equation}{section}

\usepackage{autobreak}

\makeatletter
\def\@maketitle{\newpage
	\begin{center}\let \footnote \thanks
		{\LARGE\bfseries \@title \par}\vskip 2.5em{\large
				\lineskip .5em\begin{tabular}[t]{c}\@author
				\end{tabular}\par}\vskip 1em{\large \@date}\end{center}\par
	\vskip 1.5em}
\makeatother



\newlist{keylist}{enumerate}{1}
\setlist[keylist]{label=\textbf{[\thekeylisti]}, ref=\textbf{\thekeylisti}} 

\usepackage{amsthm}
\theoremstyle{plain}
\newtheorem{theo}{Theorem}[section]
\newtheorem{lemma}[theo]{Lemma}
\newtheorem{prop}[theo]{Proposition}
\newtheorem{cor}[theo]{Corollary}

\theoremstyle{remark}
\newtheorem{remark}[theo]{\bf Remark}

\usepackage{esint}

\def\XXint#1#2#3{{\setbox0=\hbox{$#1{#2#3}{\int}$ }
			\vcenter{\hbox{$#2#3$ }}\kern-.6\wd0}}

\DeclarePairedDelimiter{\norm}{\|}{\|}

\DeclarePairedDelimiter{\inn}{\langle}{\rangle}

\usepackage{suffix}

\newcommand{\inner}[2]{\inn{#1,#2}}
\WithSuffix\newcommand\inner*[2]{\inn*{#1,#2}}

\DeclarePairedDelimiter{\positive}{(}{)^{+}}
\DeclarePairedDelimiter{\negative}{(}{)^{-}}

\newcommand\pos\positive
\renewcommand\neg\negative

\WithSuffix\newcommand\pos*{\positive*}
\WithSuffix\newcommand\neg*{\negative*}

\newcommand{\R}{{\mathbb{R}}}

\newcommand{\pnorm}[2]{\norm{#2}_{\L{#1}}}
\WithSuffix\newcommand\pnorm*[2]{\norm*{#2}_{\L{#1}}}

\newcommand{\psnorm}[3]{\norm{#3}_{\L{#1}(#2)}}
\WithSuffix\newcommand\psnorm*[3]{\norm*{#3}_{\L{#1}(#2)}}

\newcommand{\pnormp}[2]{\pnorm{#1}{#2}^{#1}}
\WithSuffix\newcommand\pnormp*[2]{\pnorm*{#1}{#2}^{#1}}

\newcommand{\psnormp}[3]{\psnorm{#1}{#2}{#3}^{#1}}
\WithSuffix\newcommand\psnormp*[3]{\psnorm*{#1}{#2}{#3}^{#1}}

\newcommand\svec\vec

\renewcommand{\vec}{\mathbf}
\renewcommand{\svec}{\boldsymbol}

\newcommand{\conv}{\ast}
\renewcommand{\star}{\conv}

\renewcommand{\d}{\mathrm{d}}
\newcommand{\dd}{\mathop{}\!\d}

\newcommand{\dt}{\dd t}

\newcommand{\dx}{\dd x}
\newcommand{\dy}{\dd y}

\newcommand{\ppr}{(r)}

\newcommand{\Wr}{^{W,\,\ppr}}

\newlength{\dhatheight}

\ifthenelse{\isundefined{\Wr}}{
	\newcommand{\Wr}{^{W,\,\ppr}}
}{
	\renewcommand{\Wr}{^{W,\,\ppr}}
}

\newcommand{\s}{s}

\newcommand{\rhodL}{\rho_{\delta, L}}

\newcommand{\rhodLh}{\rho_{h, \delta, L}}
\newcommand{\cdLh}{c_{h, \delta, L}}
\newcommand{\rhoL}{\rho_{L}}
\newcommand{\rhoLDt}{\rho^{\Delta t}_{L}}
\newcommand{\rhoLDtp}{\rho^{\Delta t, +}_{L}}
\newcommand{\rhoLDtm}{\rho^{\Delta t, -}_{L}}
\newcommand{\rhoLDtpm}{\rho^{\Delta t(, \pm)}_{L}}
\newcommand{\rhoDt}{\rho^{\Delta t}}
\newcommand{\rhoDtp}{\rho^{\Delta t, +}}
\newcommand{\rhoDtm}{\rho^{\Delta t, -}}
\newcommand{\rhoDtpm}{\rho^{\Delta t(, \pm)}}

\newcommand{\dtau}{\mathrm{d}\tau}

\newcommand*\circled[1]{\tikz[baseline=(char.base)]{
            \node[shape=circle,draw,inner sep=0.5pt] (char) {#1};}}

\newcommand{\LL}{\mathcal{L}}

\newcommand{\LN}{\mathcal{L}_{\mathrm{N}}}
\newcommand{\LNh}{\mathcal{L}_{\mathrm{N}, h}}
\newcommand{\ee}{\mathrm{e}}

\newcommand{\Hsd}{\mathcal{H}^s_\diamond}

\newcommand{\vrm}{\textrm{v}} \usepackage{todonotes}

\newif\ifskiptable

\pgfplotsset{colormap={hsv}{
			hsb(0.00cm)=(0.00,0,0.95);
			hsb(0.05cm)=(0.05,1,1);
			hsb(0.10cm)=(0.10,1,1);
			hsb(0.15cm)=(0.15,1,1);
			hsb(0.20cm)=(0.20,1,1);
			hsb(0.25cm)=(0.25,1,1);
			hsb(0.30cm)=(0.30,1,1);
			hsb(0.35cm)=(0.35,1,1);
			hsb(0.40cm)=(0.40,1,1);
			hsb(0.45cm)=(0.45,1,1);
			hsb(0.50cm)=(0.50,1,1);
			hsb(0.55cm)=(0.55,1,1);
			hsb(0.60cm)=(0.60,1,1);
			hsb(0.65cm)=(0.65,1,1);
			hsb(0.70cm)=(0.70,1,1);
			hsb(0.75cm)=(0.75,1,1);
			hsb(0.80cm)=(0.80,1,1);
			hsb(0.85cm)=(0.85,1,1);
			hsb(0.90cm)=(0.90,1,1);
			hsb(0.95cm)=(0.95,1,1);
			hsb(1.00cm)=(1.00,1,1);
		}
}

\pgfplotsset{colormap={hsvSoft}{
			hsb(0.00cm)=(0.00,0,0.95);
			hsb(0.05cm)=(0.05,1,1);
			hsb(0.10cm)=(0.10,1,1);
			hsb(0.15cm)=(0.15,1,1);
			hsb(0.20cm)=(0.20,1,1);
			hsb(0.25cm)=(0.25,1,1);
			hsb(0.30cm)=(0.30,1,1);
			hsb(0.35cm)=(0.35,1,1);
			hsb(0.40cm)=(0.40,1,1);
			hsb(0.45cm)=(0.45,1,1);
			hsb(0.50cm)=(0.50,1,1);
			hsb(0.55cm)=(0.55,1,1);
			hsb(0.60cm)=(0.60,1,1);
			hsb(0.65cm)=(0.65,1,1);
			hsb(0.70cm)=(0.70,1,1);
			hsb(0.75cm)=(0.75,1,1);
			hsb(0.80cm)=(0.80,1,1);
			hsb(0.85cm)=(0.85,1,1);
			hsb(0.90cm)=(0.90,1,1);
			hsb(0.95cm)=(0.95,1,1);
			hsb(1.00cm)=(0.00,0,0.95);
		}
}

\pgfplotsset{colormap={viridisFull}{
			rgb=(0.26700401, 0.00487433, 0.32941519)
			rgb=(0.26851048, 0.00960483, 0.33542652)
			rgb=(0.26994384, 0.01462494, 0.34137895)
			rgb=(0.27130489, 0.01994186, 0.34726862)
			rgb=(0.27259384, 0.02556309, 0.35309303)
			rgb=(0.27380934, 0.03149748, 0.35885256)
			rgb=(0.27495242, 0.03775181, 0.36454323)
			rgb=(0.27602238, 0.04416723, 0.37016418)
			rgb=(0.2770184 , 0.05034437, 0.37571452)
			rgb=(0.27794143, 0.05632444, 0.38119074)
			rgb=(0.27879067, 0.06214536, 0.38659204)
			rgb=(0.2795655 , 0.06783587, 0.39191723)
			rgb=(0.28026658, 0.07341724, 0.39716349)
			rgb=(0.28089358, 0.07890703, 0.40232944)
			rgb=(0.28144581, 0.0843197 , 0.40741404)
			rgb=(0.28192358, 0.08966622, 0.41241521)
			rgb=(0.28232739, 0.09495545, 0.41733086)
			rgb=(0.28265633, 0.10019576, 0.42216032)
			rgb=(0.28291049, 0.10539345, 0.42690202)
			rgb=(0.28309095, 0.11055307, 0.43155375)
			rgb=(0.28319704, 0.11567966, 0.43611482)
			rgb=(0.28322882, 0.12077701, 0.44058404)
			rgb=(0.28318684, 0.12584799, 0.44496 )
			rgb=(0.283072 , 0.13089477, 0.44924127)
			rgb=(0.28288389, 0.13592005, 0.45342734)
			rgb=(0.28262297, 0.14092556, 0.45751726)
			rgb=(0.28229037, 0.14591233, 0.46150995)
			rgb=(0.28188676, 0.15088147, 0.46540474)
			rgb=(0.28141228, 0.15583425, 0.46920128)
			rgb=(0.28086773, 0.16077132, 0.47289909)
			rgb=(0.28025468, 0.16569272, 0.47649762)
			rgb=(0.27957399, 0.17059884, 0.47999675)
			rgb=(0.27882618, 0.1754902 , 0.48339654)
			rgb=(0.27801236, 0.18036684, 0.48669702)
			rgb=(0.27713437, 0.18522836, 0.48989831)
			rgb=(0.27619376, 0.19007447, 0.49300074)
			rgb=(0.27519116, 0.1949054 , 0.49600488)
			rgb=(0.27412802, 0.19972086, 0.49891131)
			rgb=(0.27300596, 0.20452049, 0.50172076)
			rgb=(0.27182812, 0.20930306, 0.50443413)
			rgb=(0.27059473, 0.21406899, 0.50705243)
			rgb=(0.26930756, 0.21881782, 0.50957678)
			rgb=(0.26796846, 0.22354911, 0.5120084 )
			rgb=(0.26657984, 0.2282621 , 0.5143487 )
			rgb=(0.2651445 , 0.23295593, 0.5165993 )
			rgb=(0.2636632 , 0.23763078, 0.51876163)
			rgb=(0.26213801, 0.24228619, 0.52083736)
			rgb=(0.26057103, 0.2469217 , 0.52282822)
			rgb=(0.25896451, 0.25153685, 0.52473609)
			rgb=(0.25732244, 0.2561304 , 0.52656332)
			rgb=(0.25564519, 0.26070284, 0.52831152)
			rgb=(0.25393498, 0.26525384, 0.52998273)
			rgb=(0.25219404, 0.26978306, 0.53157905)
			rgb=(0.25042462, 0.27429024, 0.53310261)
			rgb=(0.24862899, 0.27877509, 0.53455561)
			rgb=(0.2468114 , 0.28323662, 0.53594093)
			rgb=(0.24497208, 0.28767547, 0.53726018)
			rgb=(0.24311324, 0.29209154, 0.53851561)
			rgb=(0.24123708, 0.29648471, 0.53970946)
			rgb=(0.23934575, 0.30085494, 0.54084398)
			rgb=(0.23744138, 0.30520222, 0.5419214 )
			rgb=(0.23552606, 0.30952657, 0.54294396)
			rgb=(0.23360277, 0.31382773, 0.54391424)
			rgb=(0.2316735 , 0.3181058 , 0.54483444)
			rgb=(0.22973926, 0.32236127, 0.54570633)
			rgb=(0.22780192, 0.32659432, 0.546532 )
			rgb=(0.2258633 , 0.33080515, 0.54731353)
			rgb=(0.22392515, 0.334994 , 0.54805291)
			rgb=(0.22198915, 0.33916114, 0.54875211)
			rgb=(0.22005691, 0.34330688, 0.54941304)
			rgb=(0.21812995, 0.34743154, 0.55003755)
			rgb=(0.21620971, 0.35153548, 0.55062743)
			rgb=(0.21429757, 0.35561907, 0.5511844 )
			rgb=(0.21239477, 0.35968273, 0.55171011)
			rgb=(0.2105031 , 0.36372671, 0.55220646)
			rgb=(0.20862342, 0.36775151, 0.55267486)
			rgb=(0.20675628, 0.37175775, 0.55311653)
			rgb=(0.20490257, 0.37574589, 0.55353282)
			rgb=(0.20306309, 0.37971644, 0.55392505)
			rgb=(0.20123854, 0.38366989, 0.55429441)
			rgb=(0.1994295 , 0.38760678, 0.55464205)
			rgb=(0.1976365 , 0.39152762, 0.55496905)
			rgb=(0.19585993, 0.39543297, 0.55527637)
			rgb=(0.19410009, 0.39932336, 0.55556494)
			rgb=(0.19235719, 0.40319934, 0.55583559)
			rgb=(0.19063135, 0.40706148, 0.55608907)
			rgb=(0.18892259, 0.41091033, 0.55632606)
			rgb=(0.18723083, 0.41474645, 0.55654717)
			rgb=(0.18555593, 0.4185704 , 0.55675292)
			rgb=(0.18389763, 0.42238275, 0.55694377)
			rgb=(0.18225561, 0.42618405, 0.5571201 )
			rgb=(0.18062949, 0.42997486, 0.55728221)
			rgb=(0.17901879, 0.43375572, 0.55743035)
			rgb=(0.17742298, 0.4375272 , 0.55756466)
			rgb=(0.17584148, 0.44128981, 0.55768526)
			rgb=(0.17427363, 0.4450441 , 0.55779216)
			rgb=(0.17271876, 0.4487906 , 0.55788532)
			rgb=(0.17117615, 0.4525298 , 0.55796464)
			rgb=(0.16964573, 0.45626209, 0.55803034)
			rgb=(0.16812641, 0.45998802, 0.55808199)
			rgb=(0.1666171 , 0.46370813, 0.55811913)
			rgb=(0.16511703, 0.4674229 , 0.55814141)
			rgb=(0.16362543, 0.47113278, 0.55814842)
			rgb=(0.16214155, 0.47483821, 0.55813967)
			rgb=(0.16066467, 0.47853961, 0.55811466)
			rgb=(0.15919413, 0.4822374 , 0.5580728 )
			rgb=(0.15772933, 0.48593197, 0.55801347)
			rgb=(0.15626973, 0.4896237 , 0.557936 )
			rgb=(0.15481488, 0.49331293, 0.55783967)
			rgb=(0.15336445, 0.49700003, 0.55772371)
			rgb=(0.1519182 , 0.50068529, 0.55758733)
			rgb=(0.15047605, 0.50436904, 0.55742968)
			rgb=(0.14903918, 0.50805136, 0.5572505 )
			rgb=(0.14760731, 0.51173263, 0.55704861)
			rgb=(0.14618026, 0.51541316, 0.55682271)
			rgb=(0.14475863, 0.51909319, 0.55657181)
			rgb=(0.14334327, 0.52277292, 0.55629491)
			rgb=(0.14193527, 0.52645254, 0.55599097)
			rgb=(0.14053599, 0.53013219, 0.55565893)
			rgb=(0.13914708, 0.53381201, 0.55529773)
			rgb=(0.13777048, 0.53749213, 0.55490625)
			rgb=(0.1364085 , 0.54117264, 0.55448339)
			rgb=(0.13506561, 0.54485335, 0.55402906)
			rgb=(0.13374299, 0.54853458, 0.55354108)
			rgb=(0.13244401, 0.55221637, 0.55301828)
			rgb=(0.13117249, 0.55589872, 0.55245948)
			rgb=(0.1299327 , 0.55958162, 0.55186354)
			rgb=(0.12872938, 0.56326503, 0.55122927)
			rgb=(0.12756771, 0.56694891, 0.55055551)
			rgb=(0.12645338, 0.57063316, 0.5498411 )
			rgb=(0.12539383, 0.57431754, 0.54908564)
			rgb=(0.12439474, 0.57800205, 0.5482874 )
			rgb=(0.12346281, 0.58168661, 0.54744498)
			rgb=(0.12260562, 0.58537105, 0.54655722)
			rgb=(0.12183122, 0.58905521, 0.54562298)
			rgb=(0.12114807, 0.59273889, 0.54464114)
			rgb=(0.12056501, 0.59642187, 0.54361058)
			rgb=(0.12009154, 0.60010387, 0.54253043)
			rgb=(0.11973756, 0.60378459, 0.54139999)
			rgb=(0.11951163, 0.60746388, 0.54021751)
			rgb=(0.11942341, 0.61114146, 0.53898192)
			rgb=(0.11948255, 0.61481702, 0.53769219)
			rgb=(0.11969858, 0.61849025, 0.53634733)
			rgb=(0.12008079, 0.62216081, 0.53494633)
			rgb=(0.12063824, 0.62582833, 0.53348834)
			rgb=(0.12137972, 0.62949242, 0.53197275)
			rgb=(0.12231244, 0.63315277, 0.53039808)
			rgb=(0.12344358, 0.63680899, 0.52876343)
			rgb=(0.12477953, 0.64046069, 0.52706792)
			rgb=(0.12632581, 0.64410744, 0.52531069)
			rgb=(0.12808703, 0.64774881, 0.52349092)
			rgb=(0.13006688, 0.65138436, 0.52160791)
			rgb=(0.13226797, 0.65501363, 0.51966086)
			rgb=(0.13469183, 0.65863619, 0.5176488 )
			rgb=(0.13733921, 0.66225157, 0.51557101)
			rgb=(0.14020991, 0.66585927, 0.5134268 )
			rgb=(0.14330291, 0.66945881, 0.51121549)
			rgb=(0.1466164 , 0.67304968, 0.50893644)
			rgb=(0.15014782, 0.67663139, 0.5065889 )
			rgb=(0.15389405, 0.68020343, 0.50417217)
			rgb=(0.15785146, 0.68376525, 0.50168574)
			rgb=(0.16201598, 0.68731632, 0.49912906)
			rgb=(0.1663832 , 0.69085611, 0.49650163)
			rgb=(0.1709484 , 0.69438405, 0.49380294)
			rgb=(0.17570671, 0.6978996 , 0.49103252)
			rgb=(0.18065314, 0.70140222, 0.48818938)
			rgb=(0.18578266, 0.70489133, 0.48527326)
			rgb=(0.19109018, 0.70836635, 0.48228395)
			rgb=(0.19657063, 0.71182668, 0.47922108)
			rgb=(0.20221902, 0.71527175, 0.47608431)
			rgb=(0.20803045, 0.71870095, 0.4728733 )
			rgb=(0.21400015, 0.72211371, 0.46958774)
			rgb=(0.22012381, 0.72550945, 0.46622638)
			rgb=(0.2263969 , 0.72888753, 0.46278934)
			rgb=(0.23281498, 0.73224735, 0.45927675)
			rgb=(0.2393739 , 0.73558828, 0.45568838)
			rgb=(0.24606968, 0.73890972, 0.45202405)
			rgb=(0.25289851, 0.74221104, 0.44828355)
			rgb=(0.25985676, 0.74549162, 0.44446673)
			rgb=(0.26694127, 0.74875084, 0.44057284)
			rgb=(0.27414922, 0.75198807, 0.4366009 )
			rgb=(0.28147681, 0.75520266, 0.43255207)
			rgb=(0.28892102, 0.75839399, 0.42842626)
			rgb=(0.29647899, 0.76156142, 0.42422341)
			rgb=(0.30414796, 0.76470433, 0.41994346)
			rgb=(0.31192534, 0.76782207, 0.41558638)
			rgb=(0.3198086 , 0.77091403, 0.41115215)
			rgb=(0.3277958 , 0.77397953, 0.40664011)
			rgb=(0.33588539, 0.7770179 , 0.40204917)
			rgb=(0.34407411, 0.78002855, 0.39738103)
			rgb=(0.35235985, 0.78301086, 0.39263579)
			rgb=(0.36074053, 0.78596419, 0.38781353)
			rgb=(0.3692142 , 0.78888793, 0.38291438)
			rgb=(0.37777892, 0.79178146, 0.3779385 )
			rgb=(0.38643282, 0.79464415, 0.37288606)
			rgb=(0.39517408, 0.79747541, 0.36775726)
			rgb=(0.40400101, 0.80027461, 0.36255223)
			rgb=(0.4129135 , 0.80304099, 0.35726893)
			rgb=(0.42190813, 0.80577412, 0.35191009)
			rgb=(0.43098317, 0.80847343, 0.34647607)
			rgb=(0.44013691, 0.81113836, 0.3409673 )
			rgb=(0.44936763, 0.81376835, 0.33538426)
			rgb=(0.45867362, 0.81636288, 0.32972749)
			rgb=(0.46805314, 0.81892143, 0.32399761)
			rgb=(0.47750446, 0.82144351, 0.31819529)
			rgb=(0.4870258 , 0.82392862, 0.31232133)
			rgb=(0.49661536, 0.82637633, 0.30637661)
			rgb=(0.5062713 , 0.82878621, 0.30036211)
			rgb=(0.51599182, 0.83115784, 0.29427888)
			rgb=(0.52577622, 0.83349064, 0.2881265 )
			rgb=(0.5356211 , 0.83578452, 0.28190832)
			rgb=(0.5455244 , 0.83803918, 0.27562602)
			rgb=(0.55548397, 0.84025437, 0.26928147)
			rgb=(0.5654976 , 0.8424299 , 0.26287683)
			rgb=(0.57556297, 0.84456561, 0.25641457)
			rgb=(0.58567772, 0.84666139, 0.24989748)
			rgb=(0.59583934, 0.84871722, 0.24332878)
			rgb=(0.60604528, 0.8507331 , 0.23671214)
			rgb=(0.61629283, 0.85270912, 0.23005179)
			rgb=(0.62657923, 0.85464543, 0.22335258)
			rgb=(0.63690157, 0.85654226, 0.21662012)
			rgb=(0.64725685, 0.85839991, 0.20986086)
			rgb=(0.65764197, 0.86021878, 0.20308229)
			rgb=(0.66805369, 0.86199932, 0.19629307)
			rgb=(0.67848868, 0.86374211, 0.18950326)
			rgb=(0.68894351, 0.86544779, 0.18272455)
			rgb=(0.69941463, 0.86711711, 0.17597055)
			rgb=(0.70989842, 0.86875092, 0.16925712)
			rgb=(0.72039115, 0.87035015, 0.16260273)
			rgb=(0.73088902, 0.87191584, 0.15602894)
			rgb=(0.74138803, 0.87344918, 0.14956101)
			rgb=(0.75188414, 0.87495143, 0.14322828)
			rgb=(0.76237342, 0.87642392, 0.13706449)
			rgb=(0.77285183, 0.87786808, 0.13110864)
			rgb=(0.78331535, 0.87928545, 0.12540538)
			rgb=(0.79375994, 0.88067763, 0.12000532)
			rgb=(0.80418159, 0.88204632, 0.11496505)
			rgb=(0.81457634, 0.88339329, 0.11034678)
			rgb=(0.82494028, 0.88472036, 0.10621724)
			rgb=(0.83526959, 0.88602943, 0.1026459 )
			rgb=(0.84556056, 0.88732243, 0.09970219)
			rgb=(0.8558096 , 0.88860134, 0.09745186)
			rgb=(0.86601325, 0.88986815, 0.09595277)
			rgb=(0.87616824, 0.89112487, 0.09525046)
			rgb=(0.88627146, 0.89237353, 0.09537439)
			rgb=(0.89632002, 0.89361614, 0.09633538)
			rgb=(0.90631121, 0.89485467, 0.09812496)
			rgb=(0.91624212, 0.89609127, 0.1007168 )
			rgb=(0.92610579, 0.89732977, 0.10407067)
			rgb=(0.93590444, 0.8985704 , 0.10813094)
			rgb=(0.94563626, 0.899815 , 0.11283773)
			rgb=(0.95529972, 0.90106534, 0.11812832)
			rgb=(0.96489353, 0.90232311, 0.12394051)
			rgb=(0.97441665, 0.90358991, 0.13021494)
			rgb=(0.98386829, 0.90486726, 0.13689671)
			rgb=(0.99324789, 0.90615657, 0.1439362 )
		}
}

\pgfplotsset{colormap={viridisSoft}{
			rgb255=(242, 242, 242);
rgb=(0.28026,0.1657,0.4765);
			rgb=(0.26366,0.23763,0.51877);
			rgb=(0.23744,0.3052,0.54192);
			rgb=(0.20862,0.36775,0.55267);
			rgb=(0.18225,0.42618,0.55711);
			rgb=(0.1592,0.48224,0.55807);
			rgb=(0.13777,0.53749,0.5549);
			rgb=(0.12115,0.59274,0.54465);
			rgb=(0.12808,0.64775,0.5235);
			rgb=(0.18065,0.7014,0.48819);
			rgb=(0.27415,0.75198,0.4366);
			rgb=(0.39517,0.79747,0.36775);
			rgb=(0.53561,0.83578,0.2819);
			rgb=(0.68895,0.86545,0.18272);
			rgb=(0.84557,0.88733,0.0997);
			rgb=(0.99324,0.90616,0.14394)
		}
}

\pgfplotsset{colormap={cellRed}{
			rgb255=(242.0,242.0,242.0);
			rgb255=(241.63157894736844,234.47368421052633,234.47368421052633);
			rgb255=(241.26315789473685,226.94736842105266,226.94736842105266);
			rgb255=(240.89473684210526,219.42105263157893,219.42105263157893);
			rgb255=(240.5263157894737,211.89473684210526,211.89473684210526);
			rgb255=(240.1578947368421,204.3684210526316,204.3684210526316);
			rgb255=(239.78947368421052,196.84210526315792,196.84210526315792);
			rgb255=(239.42105263157896,189.31578947368422,189.31578947368422);
			rgb255=(239.05263157894737,181.78947368421052,181.78947368421052);
			rgb255=(238.6842105263158,174.26315789473688,174.26315789473688);
			rgb255=(238.31578947368422,166.73684210526315,166.73684210526315);
			rgb255=(237.94736842105263,159.21052631578948,159.21052631578948);
			rgb255=(237.57894736842104,151.68421052631578,151.68421052631578);
			rgb255=(237.21052631578948,144.1578947368421,144.1578947368421);
			rgb255=(236.84210526315792,136.63157894736844,136.63157894736844);
			rgb255=(236.47368421052633,129.10526315789474,129.10526315789474);
			rgb255=(236.10526315789474,121.57894736842107,121.57894736842107);
			rgb255=(235.73684210526318,114.05263157894737,114.05263157894737);
			rgb255=(235.3684210526316,106.52631578947368,106.52631578947368);
			rgb255=(235.0,99.0,99.0);
		}
}

\pgfplotsset{colormap={cellGreen}{
			rgb255=(242.0,242.0,242.0);
			rgb255=(236.21052631578948,239.5263157894737,234.26315789473685);
			rgb255=(230.42105263157896,237.05263157894737,226.5263157894737);
			rgb255=(224.6315789473684,234.57894736842104,218.78947368421052);
			rgb255=(218.8421052631579,232.10526315789474,211.05263157894737);
			rgb255=(213.05263157894737,229.63157894736844,203.31578947368422);
			rgb255=(207.26315789473685,227.1578947368421,195.57894736842107);
			rgb255=(201.4736842105263,224.68421052631578,187.8421052631579);
			rgb255=(195.68421052631578,222.21052631578948,180.10526315789474);
			rgb255=(189.8947368421053,219.73684210526318,172.36842105263162);
			rgb255=(184.10526315789474,217.26315789473682,164.63157894736844);
			rgb255=(178.31578947368422,214.78947368421052,156.89473684210526);
			rgb255=(172.5263157894737,212.31578947368422,149.1578947368421);
			rgb255=(166.73684210526318,209.84210526315792,141.42105263157896);
			rgb255=(160.94736842105263,207.3684210526316,133.6842105263158);
			rgb255=(155.1578947368421,204.89473684210526,125.94736842105263);
			rgb255=(149.3684210526316,202.42105263157893,118.21052631578948);
			rgb255=(143.57894736842104,199.94736842105266,110.47368421052632);
			rgb255=(137.78947368421052,197.47368421052633,102.73684210526316);
			rgb255=(132.0,195.0,95.0);
		}
}

\pgfplotsset{colormap={cellRedSquared}{
			rgb255=(242.0,242.0,242.0);
			rgb255=(241.28254847645428,227.34349030470915,227.34349030470915);
			rgb255=(240.60387811634348,213.47922437673128,213.47922437673128);
			rgb255=(239.9639889196676,200.40720221606648,200.40720221606648);
			rgb255=(239.36288088642658,188.1274238227147,188.1274238227147);
			rgb255=(238.8005540166205,176.63988919667594,176.63988919667594);
			rgb255=(238.2770083102493,165.94459833795014,165.94459833795014);
			rgb255=(237.79224376731304,156.04155124653738,156.04155124653738);
			rgb255=(237.34626038781164,146.93074792243766,146.93074792243766);
			rgb255=(236.93905817174516,138.61218836565098,138.61218836565098);
			rgb255=(236.57063711911357,131.0858725761773,131.0858725761773);
			rgb255=(236.2409972299169,124.35180055401662,124.35180055401662);
			rgb255=(235.95013850415512,118.40997229916897,118.40997229916897);
			rgb255=(235.69806094182823,113.26038781163435,113.26038781163435);
			rgb255=(235.4847645429363,108.90304709141274,108.90304709141274);
			rgb255=(235.3102493074792,105.33795013850416,105.33795013850416);
			rgb255=(235.17451523545705,102.56509695290858,102.56509695290858);
			rgb255=(235.0775623268698,100.58448753462605,100.58448753462605);
			rgb255=(235.01939058171746,99.3961218836565,99.3961218836565);
			rgb255=(235.0,99.0,99.0);
		}
}
\pgfplotsset{colormap={cellGreenSquared}{
			rgb255=(242.0,242.0,242.0);
			rgb255=(230.7257617728532,237.18282548476455,226.93351800554018);
			rgb255=(220.06094182825484,232.62603878116343,212.6814404432133);
			rgb255=(210.00554016620498,228.32963988919667,199.2437673130194);
			rgb255=(200.5595567867036,224.29362880886427,186.62049861495845);
			rgb255=(191.7229916897507,220.5180055401662,174.8116343490305);
			rgb255=(183.49584487534625,217.0027700831025,163.81717451523545);
			rgb255=(175.87811634349032,213.74792243767314,153.63711911357342);
			rgb255=(168.86980609418282,210.75346260387812,144.27146814404432);
			rgb255=(162.47091412742384,208.01939058171746,135.72022160664818);
			rgb255=(156.68144044321332,205.54570637119116,127.98337950138506);
			rgb255=(151.50138504155126,203.33240997229916,121.06094182825484);
			rgb255=(146.9307479224377,201.37950138504155,114.95290858725764);
			rgb255=(142.96952908587255,199.68698060941827,109.65927977839334);
			rgb255=(139.61772853185596,198.25484764542935,105.18005540166205);
			rgb255=(136.8753462603878,197.0831024930748,101.51523545706371);
			rgb255=(134.74238227146813,196.17174515235456,98.66481994459834);
			rgb255=(133.21883656509695,195.5207756232687,96.62880886426592);
			rgb255=(132.30470914127426,195.13019390581718,95.40720221606648);
			rgb255=(132.0,195.0,95.0);
		}
} 
\usepgfplotslibrary{patchplots}

\pgfplotsset{every axis/.append style={
			grid=both,
			grid style={white, line width=.1pt},
			major grid style={white, line width=1.5pt},
			axis background/.style={fill=gray!10},
			axis line style={draw=none},
			tick style={draw=none},
			xlabel = $x$,
line width=1pt,
legend style={
					line width = 1pt,
					draw=none,
					/tikz/every even column/.append style={column sep=0.5cm}
				},
		}}

\usepgfplotslibrary{groupplots}

\usepackage{calc}

\definecolor{gg0}{HTML}{E24A33}
\definecolor{gg1}{HTML}{348ABD}
\definecolor{gg2}{HTML}{988ED5}
\definecolor{gg3}{HTML}{777777}
\definecolor{gg4}{HTML}{FBC15E}
\definecolor{gg5}{HTML}{8EBA42}
\definecolor{gg6}{HTML}{FFB5B8}

\pgfplotsset{
	/pgfplots/colormap={bright}{rgb255=(0,0,0) rgb255=(78,3,100) rgb255=(2,74,255)
			rgb255=(255,21,181) rgb255=(255,113,26) rgb255=(147,213,114) rgb255=(230,255,0)
			rgb255=(255,255,255)}
}

\newcommand{\addappendix}{\section*{\appendixname}\addcontentsline{toc}{section}{\appendixname}\counterwithin*{figure}{section}\stepcounter{section}\renewcommand{\thesection}{A}\renewcommand{\thefigure}{\thesection.\arabic{figure}}}

\definecolor{brandeisblue}{rgb}{0.0, 0.44, 1.0}
\definecolor{lincolngreen}{rgb}{0.11, 0.35, 0.02}
\definecolor{indiagreen}{rgb}{0.07, 0.53, 0.03}
\definecolor{venetianred}{rgb}{0.78, 0.03, 0.08}
\definecolor{darkorange}{rgb}{1.0, 0.55, 0.0}
\definecolor{burntorange}{rgb}{0.8, 0.33, 0.0}
\definecolor{flame}{rgb}{0.89, 0.35, 0.13}
\definecolor{non-photoblue}{rgb}{0.64, 0.87, 0.93}

\usepackage{tikz}
\usepackage{tkz-base}
\usetikzlibrary{calc}
\usepackage{tkz-euclide}
\usepackage{xcolor}

\renewcommand{\review}[2]{}
\renewcommand{\creview}[3]{}
\renewcommand{\ntcreview}[3]{}

\renewcommand{\tableofcontents}{}
\renewcommand{\listofreviews}{}

\definecolor{revisionColourOne}{RGB}{180,0,0}
\definecolor{revisionColourTwo}{RGB}{0,0,180}

\usepackage{fullpage}

\usepackage{setspace}
\begin{document}
\begin{singlespace}\maketitle\end{singlespace}
\begin{abstract}
We study a nonlocal diffusion equation of porous medium type featuring a generalised fractional pressure with spatial anisotropy.
We construct a finite element method for the numerical solution of the equation on a bounded open Lipschitz polytopal domain $\Omega \subset \R^{d}$, where $d = 2$ or $3$. The pressure in the model is defined as the solution of fractional elliptic problem involving the fractional power of a second order differential operator, in terms of its spectral definition.
Under suitable assumptions on the fractional order and the coefficients of the operator, we rigorously prove convergence of the numerical scheme. The analysis is carried out in two stages: first passing to the limit in the spatial discretization, and then in the time step, ultimately showing that a subsequence of the sequence of finite element approximations defined by the proposed numerical method converges to a bounded and nonnegative weak solution of the initial-boundary-value problem under consideration. 
Finally, we present numerical experiments in two dimensions illustrating the computational aspects of the method and highlighting the interplay between nonlocal effects and spatial anisotropy under different configurations. We also show numerically the failure of the comparison principle and exponential decay of the numerical solution to a steady state.

\end{abstract}
 \subjectclassification{\subjectPDF 35K55; 35R11; 65N30.}

\section{Introduction}\label{Sect:1}

The aim of this work is to present a convergent finite element scheme for a nonlocal diffusion equation of porous medium type with a generalised fractional pressure, capable of accounting for spatial anisotropy. We present a rigorous proof of convergence for the resulting finite element approximations and we support the analysis with numerical simulations of different configurations in two dimensions. 

Let us denote by $\rho$ a certain density distribution, function of space and time, $\rho = \rho(x, t)$.
The problem that we will treat has the form 
    \begin{equation} \label{ExtendedProb intro}
\left \{
\begin{aligned}
&\frac{\partial \rho}{\partial t} = \Delta \rho - \nabla \cdot (\rho A(x)\nabla c), \\
& - \LL^{\s} c = \rho,
\end{aligned}
\right.
\end{equation}
where $0<s<1$ is the \textit{fractional order} and $\LL$ is a second order elliptic operator of the form $\LL u = \textrm{div}(A(x) \nabla u) + Q(x) u$ for $A(x) \in \mathbb{R}^{d \times d}$ and $Q(x) \in \mathbb{R}$. The problem \eqref{ExtendedProb intro} is a parabolic regularization of the porous medium equation
    \begin{equation} \label{ExtendedProb intro noParab Regularization}
\frac{\partial \rho}{\partial t} = - \nabla \cdot (\rho A(x)\nabla c), \qquad - \LL^{\s} c = \rho,
\end{equation}
The model is formulated in order to generalise and introduce spatial anisotropy to the porous medium equation with a fractional pressure 
\begin{equation} \label{FracPor FracLap}
\frac{\partial \rho}{\partial t} = - \nabla \cdot (\rho \nabla c), \qquad - (-\Delta)^{\s} c = \rho,
\end{equation}
where the pressure is given by the inverse fractional Laplacian of the density. 

In \cite{caffarelli2010nonlinear,caffarelli2013regularity,chen2022analysis}, under suitable assumptions, the authors proved there exists a weak solution to \eqref{FracPor FracLap} on the whole of $\mathbb{R}^{d}$. For a parabolic regularization of \eqref{FracPor FracLap} the existence of a smooth solution on $\mathbb{R}^{d}$ has also been proved, cf. \cite{choi2021classical},  and the asymptotic behaviour of such model has been studied in \cite{caffarelli2010asymptotic, carrillo2015exponential}.
We refer to \cite{Vazquez2017} for a broader discussion on diffusion models of this type. 

Usually nonlocal models involving the fractional Laplacian are treated on the whole space $\mathbb{R}^{d}$, as the previously cited references do. From a numerical perspective we often need to confine the values of the density to a bounded domain; this necessity comes with several challenges, the first of which is the use of a proper definition of a nonlocal operator on $\mathbb{R}^d$. A finite element scheme for a parabolic regularization of problem \eqref{FracPor FracLap} on a bounded domain, with no-flux boundary conditions, has been recently presented in \cite{Fronzoni2025}, using the spectral definition of the fractional Laplacian. Our aim is to apply a similar strategy for proving convergence of a finite element scheme to the anisotropic model \eqref{ExtendedProb intro}.

In \cite{Li2018} an anisotropic porous medium equation (case $s=0$) is developed from Darcy’s law in anisotropic porous media, where the matrix $A(x)$ is the permeability matrix, representing heterogeneous media; finite element methods were used to compute numerical solutions. Existence of weak solutions for an anisotropic porous medium equation has been proved in \cite{Zhi2022}.
Anisotropic models of chemotaxis have also been studied through the years and we cite \cite{Liu2025} as a recent example, where matrix-valued sensitivities are introduced in a Keller-Segel system. 

Our equation \eqref{ExtendedProb intro} extends these models to the case of a general fractional operator that features also a potential term $Q(x)$ that can be used for confinement of the density. The heterogeneity in the media interplays with the presence of nonlocality, provided by the fractional order operator $\LL$.

As we mentioned, our future analysis will be made on a bounded domain $\Omega$. In doing so, we shall confine ourselves to the physically relevant cases of $d=2$ and $d=3$ space dimensions. We will pair equation \eqref{ExtendedProb intro} and the operator $\LL$ with some boundary conditions on the boundary $\partial \Omega$.

We now make a standard set of assumptions for the operator $\LL$: having  
\begin{equation} \label{mytypeofoperator}
    \LL u = - \text{div}(A(x) \nabla u) + Q(x) u,
\end{equation}
we assume that  $Q:\Omega \to \mathbb{R}$ is a nonnegative and bounded function and $A(x) \in \mathbb{R}^{d \times d}$  is a symmetric matrix with coefficients $A_{ij}(x) = A_{ji}(x)$ bounded, measurable and satisfying uniform ellipticity, i.e. $\Lambda_1 |\vrm|^2 \leq A(x) \vrm \cdot \vrm \leq \Lambda_2 |\vrm|^2$, for all $\vrm \in \mathbb{R}^d$ and almost every $x \in \Omega$, for some constants $0 < \Lambda_1 \leq \Lambda_2$.

Our goal is to formulate a fully discrete numerical scheme based on spatial finite element discretization and implicit Euler time-stepping, similar to the one presented in \cite{Fronzoni2025} for the standard porous medium equation with a fractional pressure (case $\LL=-\Delta$). Under some restrictive hypotheses on the fractional order $s$ and with some care for the general form of the operator $\LL$ we will show that the numerical scheme can reproduce desirable properties of the model also in this case. Our analysis will take advantage of the spectral definition that we will adopt for the fractional power of the operator and we will benefit from the use of its representation in terms of the heat kernel associated to the operator.
After proving the convergence of the scheme, we will present several numerical experiments, showcasing the versatility of this nonlocal model and its capacity to incorporate nonlocality with spatial heterogeneity. 

The analysis of convergence of the scheme will be made in two consecutive steps: we will first pass to the limit in the spacial discretization parameter and then, 
at the temporally semidiscrete-level, we  shall pass to the limit with the time step, while keeping the essential estimates intact in the limit, to establish the existence of a weak solution to our problem. The reasons why the two limits are not taken simultaneously will be clear in the following sections. 

The paper is structured as follows. In Section \ref{sec:2} we first introduce the relevant functional spaces and the spectral definition of the fractional operator $\LL^s$ and we then formulate the weak form of the initial-boundary-value problem. In Section \ref{FEM:section} we present the fully discrete scheme for the equation and we prove convergence for the spatial discretization parameter tends to zero. In Section \ref{sec:4} we take the limit as the time step vanishes. Finally in Section \ref{sec:5} we comment on the scientific computing aspects of the method and we present several simulations, for different configurations of $A(x)$ and $Q(x)$, commenting on the interplay between nonlocality and spatial anisotropy.

\section{Preliminaries} \label{sec:2}

We specialize our problem \eqref{ExtendedProb intro}, endowing it with appropriate initial and boundary conditions and we provide the definition that we will use for the fractional power of the operator $\LL$. Before doing this, we present some of the properties that hold for a differential operator $\LL$ with our set of hypotheses. These features will be essential in the subsequent analysis, as well as in the spectral definition of $\LL^s$. 

As we said in the introduction, we will be working on a bounded domain and we will set a Neumann boundary condition for the density $\rho$; this will be accompanied by a boundary condition related to the differential operator $\LL$.  
 We denote by $n$ the outward unit normal to the boundary of the domain $\partial \Omega$ and we denote by $\partial_n$ the (weak) directional derivative with respect to $n$. Then, let us denote by $\partial_{n_A}$ the directional derivative $\partial_{n_A} u = A(x) \nabla u \cdot n$, let $\partial_{n_A}(\cdot) =0$ be a homogeneous Neumann boundary condition  and let us denote by $\LN$ the differential operator $\LL$ paired with it.
If we define the bilinear form 
\begin{equation} \label{Bilinear From Operator LN}
    a_{\LN}(w, v) \coloneqq \int_\Omega A(x) \nabla w \cdot \nabla v \dx + \int_\Omega Q(x) wv \dx,
\end{equation}
the pair $(\mathrm{Dom}(\LN), \LN)$ is defined
as follows: 
\begin{align*}
 \mathrm{Dom}(\LN) \coloneqq \bigg\{ u \in &H^1(\Omega) \,:\, H^1(\Omega)\ni \,v \mapsto \! a_{\LN}(u, v) \\
 &\mbox{is a continuous linear funcional on } L^2(\Omega)\bigg \},\\
 \int_\Omega (\LN u)\, v \dx &\coloneqq a_{\LN}(u, v), \quad u \in \mathrm{Dom}(\LN),\, \, v \in H^{1}(\Omega).
\end{align*}

Let $\{ \ee^{-t \LN} \}_{t\geq 0}$ be the heat semigroup associated to $\LN$ and $W_t(x,y)$ the corresponding heat kernel: $\ee^{-t\LN} u(x)$ solves the generalised heat equation
\begin{equation} \label{GeneralHE}
\left \{
\begin{aligned}
&\frac{\partial w}{\partial t} = \LL w && \textrm{in } \Omega \times (0, \infty),\\
 &w(x,0) = u(x) && \textrm{in } \Omega, \\
 &\partial_{n_A}w = 0 && \textrm{on } \partial\Omega \times (0, \infty).
\end{aligned}
\right.
\end{equation}

For an operator of the form \eqref{mytypeofoperator}, under our assumptions on $A(x)$ and $Q(x)$, the following Gaussian upper bound for the heat kernel holds:  
\begin{equation} \label{Gaussianupperbound}
    W_t(x,y) \leq  c_1 \frac{\ee^{-|x-y|^2/(c_2 t)}}{t^{\frac{d}{2}}}, \quad x, y \in \Omega, \quad t\in(0, \infty),
\end{equation}
for some positive constants $c_1, c_2$. The reasons for which \eqref{Gaussianupperbound} holds are discussed in \cite[Chapter 7]{ElMaati2005}: \eqref{Gaussianupperbound} is true for the operator $-\text{div}(A(x) \nabla u)$ (see \cite[Theorem 7]{Lumer1998}) and then for $\LN$ by \cite[Theorem 2.24]{ElMaati2005}. The upper bound \eqref{Gaussianupperbound} on the heat kernel immediately implies that the operator $\LN$ is $L^{\infty}$-\textit{contractive}, meaning that  
\begin{equation} \label{Linfty contractivity}
    \| e^{-t \LN} f \|_{L^\infty(\Omega)} \leq \| f\|_{L^\infty(\Omega)}, \quad \textrm{ for every} \quad f \in L^{\infty}(\Omega).
\end{equation}
We also notice that 
\begin{equation*}
    \int_\Omega W_t(x,y) \dy = \ee^{-t \LN}1(x),
\end{equation*}
where $1(x)$ stands for the function that is identically equal to 1 for $x \in \Omega$, and hence the heat kernel has in general mass smaller than or equal to 1, i.e.
\begin{equation} \label{mass heat kernel}
     \int_\Omega W_t(x,y) \dy \leq 1.
\end{equation}

\begin{remark}[Absence of potential term and Neumann boundary conditions]
We mention that if the lower order term in \eqref{mytypeofoperator} vanishes, that is $Q \equiv 0$ on $\Omega$, then one has at first that the heat semigroup satisfies the conservation property, that is
\begin{equation*}
    \ee^{-t \LN} 1 = 1, \quad \textrm{for all } t \geq 0,  
\end{equation*}
which implies that the kernel has unit mass, 
\begin{equation} \label{unit mass kernel}
    \int_\Omega W_t(x,y) \dy = 1, \quad \textrm{for all } x \in \Omega
\end{equation}
(the proof of this fact is due to \cite[Theorem 4.17]{ElMaati2005} and the fact that we have $L^\infty$-contractivity). 
If the operator $\LL u = - \textrm{div}(A(x) \nabla u)$ is paired with a Neumann boundary condition, the heat kernel has also a Gaussian lower bound
\begin{equation} \label{upper and lower Gaussian estimate}
{c}_{1} \frac{\mathrm{e}^{-|x-y|^{2}/(c_{3}t)}}{t^{\frac{d}{2}}} \leq W_{t}(x, y) \leq c_{2} \frac{{\mathrm e}^{-|x-y|^{2}/(c_{4}t)}}{t^{\frac{d}{2}}}, \quad x, y \in \Omega, \quad t > 0, 
\end{equation}
for positive constants $c_1, c_2, c_3, c_4$ (see \cite[Section 7]{caffarelli2016fractional} and references therein). 
\end{remark}

We remark that, under the assumptions that we set for the operator, $\LN$ is nonnegative and self-adjoint and there exists an orthonormal basis of $L^2(\Omega)$ consisting of eigenfunctions $\varphi_k$, $k\geq 1$, of $\LN$, that correspond to eigenvalues $0 \leq \lambda_0 < \lambda_1 < \lambda_2 < \cdots \nearrow \infty $, namely 
\begin{equation}\label{eigen}
    \left \{ \begin{aligned}
        & \LL \psi_k = \lambda_k \varphi_k &&\quad \textrm{in } \Omega,  \\
        &A(x) \nabla \varphi_k \cdot n = 0 &&\quad\textrm{on } \partial \Omega. 
    \end{aligned} \right.
\end{equation}

We are ready to present the spectral definition of the operator $\LN^s$ for $0<s<1$ on the bounded domain $\Omega$. We define
\begin{equation} \label{FracOpeigen}
    \LN^s u \coloneqq \sum_{k=0}^\infty \lambda_k^s u_k \varphi_k, \quad u_k \coloneqq \int_\Omega u(x) \varphi_k(x) \dx.
\end{equation}

For the definition \eqref{FracOpeigen} of the operator $\LN^s$ it holds true an additional representation formula, in terms of the heat semigroup $\{ e^{-t \LN} \}_{t\geq 0}$: 
\begin{equation} \label{semigroupformula}
    \LN^s u(x) = \frac{1}{\Gamma(-s)} \int_0^\infty (\ee^{-t \LN}u(x) - u(x)) \frac{1}{t^{1+s}} \dt;
\end{equation}
this formula will be useful in our analysis. 

Let us use the notation $\rho^{\ast}$ for the projection of $\rho$ into the space of functions with zero integral average, namely 
\begin{equation} \label{ProjZeroAv} \rho^{\ast} := \rho - \overline{\rho}, \quad \text{with } \overline{\rho} := \frac{1}{|\Omega|} \int_{\Omega} \rho \dx. \end{equation}
Let us denote by $L^2_\ast(\Omega)$ the linear subspace of $L^2(\Omega)$with zero integral average, namely 
\begin{equation*}
    L^{2}_{\ast}(\Omega) \coloneqq \left\{ v \in L^{2}(\Omega) \text{ such that } \int_{\Omega}v \dx = 0 \right\}.
\end{equation*}

We notice now that if $\LN u = 0$, then one has 
\begin{equation*}
   0 =  \int_\Omega A(x) \nabla u \cdot \nabla u \dx + \int_\Omega Q(x) u^2 \geq \Lambda_1 |\nabla u|^2,
\end{equation*}
which implies that $\nabla u = 0$ a.e. in $\Omega$, meaning that, since $\Omega$ is a domain, if  $\LN u = 0$ then $u$ is constant a.e. in $\Omega$. However, clearly if $u$ is constant a.e. in $\Omega$ then $\LN u = 0$ if and only if $\int_\Omega Q(x) \dx = 0$, which implies, by nonnegativity of $Q$, that $Q(x) = 0$ a.e. in $\Omega$.
In summary, we have two cases:
\begin{enumerate}[label=(K\arabic*), start=0]
\item there exists a subset of $\Omega$ with positive measure such that $Q(x) \neq 0 $ on that subset, and thus $\textrm{Ker}(\LN)$ is trivial; \label{(K0)}
    \item  $Q(x) = 0$ a.e. in $\Omega$, and thus $\textrm{Ker}(\LN)$ is the set of constant functions. \label{(K1)}
\end{enumerate}

Let us denote by $L^2_\diamond(\Omega)$ the space $L^2(\Omega)$ quotiented with the kernel of the operator $\LN$, namely
\begin{equation*}
    L^2_\diamond(\Omega) = L^2(\Omega)/ \textrm{Ker}(\LN);
\end{equation*}
the space $L^2_\diamond(\Omega)$ is nothing but $L^2_\ast(\Omega)$ in case \ref{(K1)} and $L^2(\Omega)$ itself in case \ref{(K0)}.
For the following we also define the projection $\Pi_\diamond: L^2(\Omega) \to L^2_\diamond(\Omega)$ as 
\begin{equation} \label{general projection}
    v^\diamond =\Pi_\diamond v \coloneqq \left\{
    \begin{array}{ll} 
    \sum_{k=0}^{\infty} v_k \varphi_k & \textrm{if \ref{(K0)}},  \\ \\
    \sum_{k=1}^{\infty} v_k \varphi_k & \textrm{if \ref{(K1)}} . \end{array} \right.
\end{equation}
Clearly $\| \Pi_\diamond v \|_{L^2(\Omega)} \leq \| v\|_{L^2(\Omega)}$. Notice that $ \Pi_\diamond v = v^\ast$  in case \ref{(K1)} and $ \Pi_\diamond v = v$ in case \ref{(K0)}. 

Let us also denote by $H^1_\diamond(\Omega)$ the set of $H^1(\Omega)$ functions that belong to $L^2_\diamond(\Omega)$, i.e. $H^1_\diamond \coloneqq  H^1(\Omega) \cap L^2_\diamond(\Omega) $

In order to ensure existence and uniqueness of solutions to elliptic problems of the type $\LN^s u = f$, we define $\mathcal{H}^s_{\diamond}(\Omega) \coloneqq \textup{Dom}(\LN^s)$ as the Hilbert space of functions, quotiented with the kernel of $\LN$ such that \eqref{FracOpeigen} belongs to $L^2_\diamond(\Omega)$, namely 
\begin{equation} \label{general Hs space}
    \Hsd(\Omega) = \left\{ u(\cdot) = \sum_{k=1}^{\infty} u_{k} \varphi_{k}(\cdot) \in L^{2}_\diamond(\Omega): \|u\|^{2}_{\Hsd(\Omega)} \coloneqq \sum_{k=0}^{\infty} \lambda_{k}^{s} u_{k}^{2} < \infty \right\}. 
\end{equation}
One has that 
\begin{equation*}
    \| u \|_{\Hsd(\Omega)} = \| \LN^{s/2} u \|_{L^2(\Omega)} \quad \forall u \in \Hsd(\Omega)
\end{equation*}
and we notice that for $s=1$ 
\begin{align*}
    \| u \|_{\mathcal{H}^1_\diamond(\Omega)}^2 = (\LL u, u) = \int_\Omega A(x) \nabla u \cdot \nabla u \dx + \int_\Omega Q(x) u^2 \dx
\end{align*}
and by uniform ellipticity and the fact that $Q$ is nonnegative and bounded one has the equivalence between $\| \cdot \|_{\mathcal{H}_\diamond^1(\Omega)}$ and $|\cdot |_{H^1(\Omega)}$ on the space $\Hsd(\Omega)$. This also implies that if $v\in H^1(\Omega)$ then $\Pi_\diamond v \in H^1(\Omega) \cap L^2_\diamond(\Omega)$ by the definition \eqref{general projection}. We denote by $\mathcal{H}^{-s}_\diamond(\Omega)$ the dual space of $\Hsd(\Omega)$.

In order to proceed further we introduce an assumption on the elliptic regularity of the operator $\LN$. We shall assume from now on that for the operator $\LN$ elliptic regularity of order $\alpha \in (0,1]$ holds. Specifically, if $T_{\LN}: \mathcal{H}^{-1}_\diamond(\Omega) \to \mathcal{H}^1_\diamond(\Omega)$ is the operator defined for $F \in \mathcal{H}^{-1}_\diamond(\Omega)$ as the solution to
\begin{equation*}
a_{\LN}(T_{\LN}F, v) = (F, v), \quad \textrm{for all } v \in \mathcal{H}^1_\diamond(\Omega),
\end{equation*}
(which exists and it is unique by the Lax-Milgram theorem and the assumptions of uniform ellipticity and boundedness that we have already set for $\LL$), we assume that, for $r \in (0, \alpha]$, $T_{\LN}$ is a bounded map of $\mathcal{H}^{-1+r}(\Omega)$ into $\mathcal{H}^{1+r}(\Omega)$. Elliptic regularity will be required to obtain convergence rates for the finite element approximation of fractional powers of the general operator $\LN$. We remark that, for example, elliptic regularity of order $\alpha=1$ will hold if $\LN$ has sufficiently smooth coefficients and if $\Omega$ is convex, or if $\LN$ has sufficiently smooth coefficients and the boundary of the domain is $C^{1,1}$; however the latter is not the case if we work, as we will, with polygonal and polyhedral domains.

We are now ready to specialize the problem \eqref{ExtendedProb intro} and pair it with some boundary conditions and an initial condition. The problem that we consider is the following: 
\begin{equation} \label{general MainFullProb}
\left \{
\begin{aligned}
&\frac{\partial \rho}{\partial t} = \Delta  \rho - \nabla \cdot (\rho A(x) \nabla c) & \textrm{in } & \Omega \times (0, \infty) , \\
& - \LN^{\s} c =  \rho^\diamond & \textrm{in } & \Omega \times (0, \infty), \\
& \partial_n \rho = 0,  \partial_{n_A} c = 0 & \textrm{on~\!} & \partial \Omega \times (0, \infty), \\
&\rho(x,0) = \rho_0(x)\quad \mbox{for all $x \in \Omega$}, &
\end{aligned}
\right.
\end{equation}
where $\rho_0 \in L^\infty(\Omega)$ is a given nonnegative initial datum, with normalized mass $\frac{1}{|\Omega|} \int_{\Omega} \rho_{0} \dx = 1$, $\rho_\diamond$ is defined through \eqref{general projection}.
 We restrict the range of the fractional order $s$ to $(1/2, 1)$; this assumption will be crucial for the convergence of the whole finite element approximation for this generalised evolutionary problem.

We now express the problem \eqref{general MainFullProb} in its weak formulation. Let $V\coloneqq H^{1}(\Omega)$; the weak formulation of the initial-boundary-value problem \eqref{general MainFullProb} is thus as follows:
\begin{subequations} \label{general BasicWeakFormMain}
\begin{gather}
\text{Find } \rho \in L^2(0,T;V)  \text{ with $\frac{\partial \rho}{\partial t} \in L^\infty(0,T;V')$ such that} \nonumber \\
\Big\langle \frac{\partial \rho}{\partial t}, \phi \Big\rangle = - \int_{\Omega} \nabla \rho \cdot \nabla \phi \dx + \int_{\Omega} \rho  A(x) \nabla c \cdot \nabla \phi \dx \quad \text{for all } \phi \in V \text{ and a.e. $t \in (0,T]$}, \label{general BasicWeakForm}
\end{gather}
subject to the initial condition $\rho(x, 0) = \rho_{0}(x)$, where $\rho_{0} \in L^{\infty}(\Omega)$ and $\rho_{0}(x) \geq 0$ for a.e. $x \in \Omega$ and where
\begin{equation} \label{general MyFracPois1}
    -\LN^{\s} c = \rho^\diamond \textrm{ in } \Omega, \quad \partial_{n_A} c = 0  \textrm{ on~\!}  \partial \Omega, \quad s\in(1/2, 1). 
\end{equation}
\end{subequations}
 Notice that $c$ is well defined, as the solution to the fractional equation (\ref{general MyFracPois1}) exists and it is unique by Theorem \ref{general FracPoiExUniq}; we will therefore write $c = -\LN^{-\s}\rho$. When we speak of a weak solution to the problem (\ref{general BasicWeakFormMain}) we mean a couple $\rho$ and $c$ such that (\ref{general BasicWeakForm}) and (\ref{general MyFracPois1}) hold.

In order to build our numerical scheme for equation \eqref{general MainFullProb}, we introduce for $0<\delta < 1< L < \infty$ the positive cut-off functions $\beta_{\delta}^{L}, \beta^L \in C^{0,1}(\mathbb{R})$
\begin{equation*} 
\beta_{\delta}^{L}(s) := \left \{ \begin{array}{ll} \delta &  \mbox{if } s \leq \delta, \\ s &  \mbox{if } \delta < s < L, \\ L & \mbox{if }  s \geq L, \end{array} \right.  \qquad \beta^{L}(s) := \left \{ \begin{array}{ll} s &  \mbox{if } \delta < s < L, \\ L & \mbox{if }  s \geq L, \end{array} \right.
\end{equation*}
prior to discretizing the problem. 

We therefore introduce the following `truncated' weak formulation:
\begin{subequations} \label{RegWeakFormMain}
\begin{gather}
\text{Find } \rhodL  \in L^2(0,T;V) \text{ with $\frac{\partial \rho_{\delta,L}}{\partial t} \in L^2(0,T;V')$ such that } \nonumber \\
    \Big\langle \frac{ \partial \rhodL}{\partial t}, \phi  \Big\rangle = -\int_{\Omega} \nabla \rhodL \cdot \nabla \phi  \dx + \int_{\Omega} \beta_{\delta}^{L}(\rhodL) A(x)  \nabla c_{\delta, L} \cdot \nabla \phi  \dx \quad \textrm{for all } \phi \in V \textrm{ and a.e. $t \in (0,T]$}, \label{RegWeakForm}
\end{gather}
subject to the initial condition $\rho_{\delta, L}(\cdot, 0) = \rho_{0}$, where 
\begin{equation}\label{RegWeakFormb}
 -\LN^{\s} c = (\beta^L(\rho))^\diamond \textrm{ in } \Omega, \quad s\in(1/2, 1) 
\end{equation}
\end{subequations}
and, as before, $V = H^1(\Omega)$. 

Let us also define the nonnegative convex function $G \in C([0,\infty))$ by
\begin{align}\label{eq:G}
G(s) := s (\log s - 1) + 1\quad \mbox{for $s>0$}\quad \mbox{and}\quad  G(0):=1.
\end{align}
We note that $G(1)=0$ and $\lim_{s \rightarrow +\infty} G(s)/s = + \infty$. 
We introduce a regularized version $G^L_\delta \in C^{2,1}(\mathbb{R})$ of the function $G(\cdot)$, defined, for $0 < \delta < 1<L < \infty$, by
\begin{equation*} G_{\delta}^{L}(s) := \left \{ \begin{array}{ll} \frac{s^{2} - \delta^{2}}{2 \delta} + (\log \delta - 1)s + 1 & \mbox{if }  s \leq \delta, \\
G(s) & \mbox{if }  \delta < s < L, \\
\frac{s^{2} - L^{2}}{2L} + (\log L - 1) s + 1 & \mbox{if }  s \geq L.
\end{array} \right.
\end{equation*}
Clearly, 
\[ (G_{\delta}^{L})'(s) = \left \{ \begin{array}{ll} \frac{s}{\delta} + \log \delta - 1 & \mbox{if }  s \leq \delta, \\ \log s & \mbox{if }  \delta < s < L, \\ \frac{s}{L} + \log L - 1 & \mbox{if }  s \geq L, \end{array} \right. \qquad (G_{\delta}^{L})''(s) = \left \{ \begin{array}{ll}  \frac{1}{\delta} & \mbox{if }   s \leq \delta, \\ \frac{1}{s} & \mbox{if }  \delta < s < L, \\ \frac{1}{L} & \mbox{if }  s \geq L ,\end{array} \right. \]
and therefore 
\[ \beta_{\delta}^{L}(s) (G_{\delta}^{L})''(s) = 1 \quad \textrm{for all } s \in \mathbb{R}.\]
In addition, for any sufficiently smooth function $\phi$, the following equality holds (see \cite{barrett2012finite}):
\begin{equation} \label{PropBas1} \beta_{\delta}^{L}(\phi) \nabla [(G_{\delta}^{L})'(\phi)] = \nabla \phi. \end{equation}
Moreover, we have that
\begin{equation} \label{ImpProp1}
\min \{ G_{\delta}^{L}(s), s(G_{\delta}^{L})'(s) \} \geq \left \{ \begin{array}{ll} \frac{s^{2}}{2 \delta} &\textrm{if } s \leq 0, \\ \frac{s^{2}}{4L} - C(L) & \textrm{if } s \geq 0, \end{array} \right.
\end{equation}
where $C(L)$ is a positive constant depending only on $L$. Finally we note that $(G_{\delta}^{L})''(s) \geq \frac{1}{L}$ for all $s \in \mathbb{R}$.

The two cut-off parameters, $\delta$ and $L$, will be present in the statement and the analysis of the fully discrete approximation of our model in Section \ref{FEM:section}. The cut-off functions will be used to have control of the discretization of the equation and obtain in the limit nonnegativity of the numerical approximations. More precisely, the strategy of our proof of convergence is the following: we will take the limit as the lower cut-off parameter $\delta$ and the spatial discretization parameter tend to zero, and subsequently we will make the semidiscrete-in-time approximation independent of the upper cut-off parameter $L$, by establishing a bound on the $L^{\infty}(\Omega)$ norm of the solutions to a semidiscrete-in-time problem, and finally we will take the limit as the time step decreases towards zero.

Thus we will show that, as the spatial and temporal discretization parameters tend to zero, a subsequence of numerical approximations converges to a bounded and nonnegative weak solution of our original problem in a suitable sense. The arguments that follow can be therefore viewed as a constructive proof of the existence of a nonnegative, weak solution to the initial-boundary-value problem for the anisotropic porous medium equation under consideration.

 We proceed now with the analysis of a finite element scheme for the problem \eqref{general BasicWeakFormMain}.

\section{Finite Element Approximation}
\label{FEM:section}

We assume that $\Omega$ is a bounded open polygonal domain in $\mathbb{R}^{2}$ or a bounded open Lipschitz polyhedral domain in 
$\mathbb{R}^{3}$. With $h>0$ being our spatial discretization parameter, let $\mathcal{T}_{h}= \{ K_{n} \}_{n=1}^{M_{h}}$ be a quasi-uniform and shape-regular family of triangulations of $\overline\Omega$, where $K_{n}$, $n=1,\ldots, M_h$, are closed simplices with mutually disjoint interiors, such that $\overline{\Omega} = \cup_{n=1}^{M_{h}} K_{n}$; assume further that the triangulation is weakly acute. Let $V_{h}$ be the linear space of continuous piecewise affine functions defined on this triangulation, i.e., 
\begin{equation*}
    V_{h} := \big\{ v_{h} \in C(\overline{\Omega}) \textrm{ such that } v_{h}\big|_{K} \in \mathbb{P}^{1} \textrm{ for all } K \in \mathcal{T}_{h} \big\},
\end{equation*}
and let $N_{h}:=\mbox{dim}(V_h)$. Moreover, let $\pi_{h}: C^{0}(\overline{\Omega}) \to V_{h}$ be the interpolation operator into the linear space $V_{h}$ based on nodal evaluation, i.e., such that, for each $\phi \in C(\overline\Omega)$,
\begin{equation*}
    \pi_{h} \phi(P_{j}) = \phi(P_{j}), \quad j = 1, \dots, N_{h},
\end{equation*} 
where $\{ P_{j} \}_{ j=1}^{N_{h}}$ are the nodes (vertices) of $\mathcal{T}_{h}$ contained in $\overline \Omega$. We note that, for each $K \in \mathcal{T}_{h}$ with vertices $\{ P_{j} \}_{j=0}^{d}$, by virtue of Jensen's inequality we have
\begin{equation} \label{PropFE2}
 (\pi_{h}(\phi)(x))^{2} = \bigg( \sum_{j = 0}^{d} \phi(P_{j}) \phi_{j}(x) \bigg)^{2} \leq \sum_{j = 0}^{d} (\phi(P_{j}))^{2} \phi_{j}(x) = \pi_{h}(\phi^{2})(x) 
 \end{equation}
for all $x \in K$ and all $\phi \in C(\overline{K})$.

We need to construct a discrete version of the operator $\LN^s$ and, in order to do this, we can consider the bilinear form $a_{\LN}(\cdot, \cdot)$ that we defined in \eqref{Bilinear From Operator LN},
which we recall to be symmetric, bounded and coercive on $H^1(\Omega) \cap L^2_\diamond (\Omega)$. We can hence restrict $a_{\LN}(\cdot, \cdot)$ to $(V_h \cap L^2_\diamond) \times (V_h \cap L^2_\diamond)$. Thus the bilinear form 
 gives rise in $V_h \cap L^2_\diamond(\Omega)$ to a discrete orthonormal basis of eigenfunctions that we denote by $\{ \psi_k^h \}_{k=1}^{N_h}$ with corresponding positive eigenvalues $0< \lambda_1^h < \cdots < \lambda_{N_h}^h$ such that 
 \begin{equation}
     a_{\LN}(\psi_k^h, v_h) = \lambda_k^h (\psi_k^h, v_h) \quad \textrm{for all } v_h \in V_h \cap L^2_\diamond(\Omega), \quad k=1, \ldots, N_h .
 \end{equation}

 We can then define a finite-dimensional counterpart of the fractional power of $\LN$: for $u_h \in V_h \cap L^2_\diamond(\Omega)$ and $s\in [0,1]$, we define
 \begin{equation} \label{Discrete FracPowerOp}
    \LNh^s u_h \coloneq \sum_{k=1}^{N_h} (\lambda_k^h)^s u_k^h \psi_k^h \quad \textrm{with } u_k^h \coloneqq \int_\Omega u_h \psi_k^h \dx, \quad k=1, \ldots, N_h.
 \end{equation}

 We shall require in the analysis of our fully discrete numerical method a discrete version of the property (\ref{PropBas1}). In order to ensure that this holds, we define a diagonal matrix $\Theta_{\delta}^{L}(\phi_{h}) \in \mathbb{R}^{d \times d}$ in the following way: for an element $K \in \mathcal{T}_h$, let $\{ P_{i} \}_{i = 0}^{d}$ be the vertices of the simplex $K$; then, for $j = 1, \dots, d$ and $x \in K^\circ$ (the interior of $K$), we define
\begin{subequations} \label{ThetaMatrix}
\begin{equation}
\widetilde{\Theta}_{\delta}^{L}(\phi_{h})_{jj}(x) := \left \{ \begin{aligned}& \frac{\phi_{h}(P_{j}) - \phi_{h}(P_{0})}{(G_{\delta}^{L})'(\phi_{h}(P_{j})) - (G_{\delta}^{L})'(\phi_{h}(P_{0}))} & \textrm{if } \phi_{h}(P_{j}) \neq \phi_{h}(P_{0}), \\
&\frac{1}{(G_{\delta}^{L})''(\phi_{h}(P_{j}))} = \beta_{\delta}^{L}(\phi_{h}(P_{j})) & \textrm{if } \phi_{h}(P_{j}) = \phi_{0}(P_{j}). \end{aligned} \right.
\end{equation}

The matrix $\widetilde{\Theta}_{\delta}^{L}$ is designed to approximate $\beta_{\delta}^{L} I$, where $\beta_{\delta}^{L}$ is the positive cut-off function defined above, and its elements are simply difference quotients approximating the inverse of the derivative of $(G_{\delta}^{L})'$, which is $(\beta_{\delta}^{L})^{-1}$. 
Now,  let $\widehat{K}$ be the reference simplex and $\widehat{x} \in \widehat{K} \mapsto P_{0} + B_{K} \widehat{x}=x \in K$ the affine function that maps $\widehat{K}$ onto $K$, and for $x \in K^\circ$ define
\begin{equation}
\Theta_{\delta}^{L}(\phi_{h})(x) := (B_{K}^{\mathrm{T}})^{-1} \,\widetilde{\Theta}_{\delta}^{L}(\phi_{h})(x) \,B_{K}^{\mathrm{T}}.
\end{equation}
\end{subequations}
With this definition in place we have that 
\begin{equation} \label{PropFE1} \Theta_{\delta}^{L}(\phi_{h})(x)\, \nabla [\pi_{h}(G_{\delta}^{L})'(\phi_{h})(x)] = \nabla \phi_{h} (x) \quad \textrm{for all } x \in K^\circ. \end{equation}

In order to clarify the sense in which the matrix function $\Theta_{\delta}^{L}(\cdot)$ approximates the cut-off function $\beta_{\delta}^{L}(\cdot) I$ as $\delta$ and $h$ tend to zero, we have the following lemma (the proof can be found in \cite{barrett2011finite}).
\begin{lemma} \label{LemmaMatrix}
    For each $K \in \mathcal{T}_{h}$ and for all $\phi_{h} \in V_{h}$ we have that 
    \begin{equation*}
        \int_{K} |\Theta_{\delta}^{L}(\phi_{h}) - \beta^{L}_{\delta}(\phi_{h}) I |^{2} \dx \leq C \bigg(\delta^{2} + h^{2} \int_{K} |\nabla \phi_{h}|^{2} \dx + \int_{K} \pi_{h}([\phi_{h}]_{-}^{2}) \dx \bigg).
    \end{equation*}
\end{lemma}
The following two lemmas will be useful in the convergence analysis of the fully discrete scheme. The proof of them can be also found in \cite{barrett2011finite}.

\begin{lemma} \label{ThetaContLemma}
    Let $\| \cdot \|$ denote the spectral norm on $\mathbb{R}^{d \times d}$. For any $\delta \in (0,1)$ and $L>1$ the function $\Theta_{\delta}^{L}: V_{h} \to \mathbb{R}^{d \times d}$ is continuous, and it satisfies 
    \begin{equation*}
        \delta \,  \xi^{\mathrm{T}} \xi \leq \xi^{\mathrm{T}} \Theta_{\delta}^{L}(\phi_{h}) \xi \leq L \, \xi^{\mathrm{T}} \xi \quad \text{for all } \xi \in \mathbb{R}^{d},\, \phi_{h} \in V_{h}.    
    \end{equation*}
    Furthermore, for all $\phi_{h}^{(1)}, \phi_{h}^{(2)} \in V_{h}$ and all $K \in \mathcal{T}_{h}$, one has that 
        \begin{equation*}
            \big\| \big(\Theta_{\delta}^{L}\big(\phi_{h}^{(1)}\big) - \Theta_{\delta}^{L}\big(\phi_{h}^{(2)}\big)\big)(x) \big\| \leq \frac{L}{\delta} \max_{j=1, \dots, d} (|\phi_{h}^{(1)}(P_{j}) - \phi_{h}^{(2)}(P_{j})| + |\phi_{h}^{(1)}(P_{0}) - \phi_{h}^{(2)}(P_{0})|)
        \end{equation*}
        for all $x \in K^\circ$.
\end{lemma}

\begin{cor} \label{CorLip}
Let $g$ be defined and strictly monotonically increasing on $\mathbb{R}$, such that $g^{-1}$, the inverse function of $g$, is Lipschitz continuous on $\mathbb{R}$, with Lipschitz constant $(g^{-1})_{\mathrm{Lip}}$; then, for all $K \in \mathcal{T}_{h}$ and for all $\phi_{h} \in V_{h}$,
\begin{equation*}
    \int_{K} |\nabla \phi_{h}|^{2} \dx \leq (g^{-1})_{\mathrm{Lip}} \int_{K} \nabla \phi_{h} \cdot \nabla \pi_{h}(g(\phi_{h})) \dx.
\end{equation*}
\end{cor}

Given the initial datum $\rho_{0} \in L^\infty(\Omega)$ such that $\rho_0 \geq 0$ and $\frac{1}{|\Omega|} \int_{\Omega} \rho_{0} \dx = 1$, we first choose $\rho^0_h \in V_{h}$ as the unique solution of the following problem:
\begin{equation} \label{InValDiscProb}
\int_{\Omega} \pi_h(\rho^0_h \phi_{h}) \dx+ \Delta t \int_{\Omega} \nabla \rho^0_h \cdot \nabla \phi_{h} \dx= \int_{\Omega} \rho_{0} \phi_{h} \dx \quad \text{for all } \phi_{h} \in V_{h}.
\end{equation}
In \cite{Fronzoni2025} it is shown the existence of a unique solution $\rho^0_h \in V_h$, together with properties 
\begin{align}\label{inibound}
0 \leq {\rho}^0_h(x) \leq \|\rho_0\|_{L^\infty(\Omega)}\quad \mbox{for all $x \in \overline\Omega$}\quad \mbox{and}\quad \frac{1}{|\Omega|}\int_\Omega \rho_h^0(x) \dx = 1.
\end{align}

 Suppose that $N \in \mathbb{N}_{\geq 2}$ and $\Delta t := T/N$. The fully-discrete scheme for the generalised problem \eqref{general BasicWeakFormMain} is the following: 
\begin{subequations} \label{general FullDiscWeakFormMain}
\begin{gather}
\textrm{ Let $\rho^0_{h,\delta,L}:=\rho^0_h \in V_h$.}
\textrm{ For } n = 1, \dots, N, \textrm{ given } \rhodLh^{n-1} \in V_{h}, \textrm{ find } \rhodLh^{n} \in V_
{h} \textrm{ such that } \nonumber \\
\label{general FullDiscWeakForm} \int_{\Omega} \pi_h\bigg(\frac{\rhodLh^{n} - \rhodLh^{n-1}}{\Delta t} \phi_{h} \bigg)\dx = -  \int_{\Omega}  \nabla \rhodLh^{n} \cdot \nabla \phi_{h} \dx + \int_{\Omega} \Theta_{\delta}^{L}(\rhodLh^{n}) A(x)  \nabla \cdLh^{n} \cdot \nabla \phi_{h} \dx \quad \\ \textrm{for all } \phi_{h} \in V_{h} \nonumber,
\end{gather}
 where $\cdLh^{n} \in V_h \cap L^2_{\diamond}(\Omega)$ satisfies
\begin{equation} \label{general SpaceDiscFracPois}
- (-\LNh)^{\s} \cdLh^{n} = (\beta^L( \rhodLh^{n}))^{\diamond}, \quad s \in (1/2, 1).
\end{equation}
\end{subequations}

\begin{remark}
    The use of the cut-off function $\beta^L(\cdot)$ in \eqref{general SpaceDiscFracPois} will help us in controlling the $L^2(\Omega)$ norm of $\cdLh^{n}$. In fact, notice that this implies
    \begin{align*}
    \| \cdLh^{n} \|^2_{L^2(\Omega)} &= \sum_{k=1}^{N_h} (\lambda_k^h)^{-2s} ((\beta^L(\rhodLh^n))^\diamond, \psi_k^h)^2 \leq (\lambda_1^h)^{-2s} \sum_{k=1}^{N_h}  ((\beta^L(\rhodLh^n))^\diamond, \psi_k^h)^2 \\
    &= (\lambda_1^h)^{-2s} \| (\beta^L(\rhodLh^n))^\diamond \|_{L^2(\Omega)}^2 \leq (\lambda_1^h)^{-2s} L^2 |\Omega| .
    \end{align*}
    For $h \to 0$  we have that $\lambda_1^h \to \lambda_1$, therefore for $h$ small enough we can control $\lambda_1^h \leq c \lambda_1$ for a positive constant $c$  (say for example $\lambda_1^h \leq  \frac{3}{2} \lambda_1$). Thus, for $h$ small enough, there exists a constant $C(L, \Omega)>0$, depending only on $L$ and $\Omega$ such that
    \begin{equation} \label{eq:uniform bound on L2 norm of chdL}
        \| \cdLh^{n} \|_{L^2(\Omega)} \leq C(L, \Omega).
    \end{equation}
    
\end{remark}

The work in this section consists in showing that, for fixed $L$ and $\Delta t$, we can take the limit as $h, \delta \to 0_+$ in \eqref{general FullDiscWeakFormMain} and find a family of limit functions $\rho_{L}^{n}$ $n=1, \dots, N$, that satisfies a limit problem. We report below the semidiscrete-in-time problem, that we will reach in the limit $h, \delta \to 0_+$.

We assign to $\rho_{0}$, for a fixed value of $\Delta t>0$, a certain smoothed initial datum $\rho^{0} = \rho^0({\Delta t}) \in V=H^{1}(\Omega)$, that is the solution to the following problem: 
\begin{equation} \label{InValSemiDiscProb}
    \int_{\Omega} \rho^{0} \phi \, \dx + \Delta t \int_{\Omega} \nabla \rho^{0} \cdot \nabla \phi \dx= \int_{\Omega} \rho_{0} \phi \, \dx \quad \text{for all } \phi \in V.
\end{equation}
Existence and uniqueness of a solution $\rho^0= \rho^0(\Delta t) \in V=H^1(\Omega)$ is proved in \cite{Fronzoni2025}, where the authors also claim that  the following four assertions hold true:
\begin{equation*}
    \rho^{0} \to \rho_{0} \quad \text{weakly in } L^{2}(\Omega) \quad \text{as} \quad  \Delta t \to 0_{+};
\end{equation*} 
thanks to the assumed nonnegativity of $\rho_0$ also $\rho^0 \geq 0$ a.e.~on $\Omega$; the function $\rho^0$ satisfies
\begin{equation} \label{IneqPropIn}
    \int_{\Omega} G(\rho^{0}) \, \dx + 4 \Delta t \int_{\Omega} \Big|\nabla \sqrt{\rho^{0}}\Big|^{2} \, \dx \leq \int_{\Omega} G(\rho_{0})\dx,
\end{equation}
where $G\in C([0,\infty))$ is the nonnegative convex function defined in \eqref{eq:G}; and 
$\rho^{0} \in L^{\infty}(\Omega) \cap V$.

The semidiscrete-in-time approximation of our problem is then formulated as follows:
\begin{subequations} \label{general WeakFormDiscTimeMain}
 \begin{gather}
\textrm{Let $\rho^0_L:= \rho^0 = \rho^0(\Delta t)$. For } n = 1, \dots, N, \textrm{ given } \rho_{L}^{n-1} \in V, \textrm{ find } \rho_{L}^{n} \in V \textrm{ such that } \nonumber \\
 \int_{\Omega} \frac{\rho_{L}^{n} - \rho_{L}^{n-1}}{\Delta t} \phi \dx = - \int_{\Omega} \nabla \rho_{L}^{n} \cdot \nabla \phi \dx+ \int_{\Omega} \beta^{L}(\rho_{L}^{n}) A(x) \nabla c_{L}^{n} \cdot \nabla \phi \dx \quad \textrm{for all } \phi \in V, \label{general WeakFormDisctime}
\end{gather}
subject to the initial condition $\rhoL^{0} = \rho^{0}$, where
\begin{equation} \label{general FinElemFracLap} - \LN^{\s} c_{L}^{n} = (\beta^L(\rho_{L}^{n}))^{\diamond} \quad \mbox{in $\Omega$}, \quad s \in (1/2, 1). \end{equation}
\end{subequations}

\begin{lemma} \label{general Fully disc uniform estimates}
    For $n=1, \dots, N$ a solution $\rhodLh^n \in V_h$ of \eqref{general FullDiscWeakForm} satisfies the following bound:
    \begin{align} 
        \int_{\Omega} \pi_{h} (G_{\delta}^{L}(\rhodLh^{n})) \dx + \Delta t \int_{\Omega}  \nabla \rhodLh^{n} \cdot \nabla \pi_{h}((G_{\delta}^{L})'(\rhodLh^{n}))) \dx \leq & \int_{\Omega} \pi_{h} (G_{\delta}^{L}(\rhodLh^{n-1})) \dx  \nonumber \\
        & \label{general DiscBound 1} +  \Delta t \int_\Omega A(x) \nabla \cdLh^n \cdot \nabla \rhodLh^n. 
    \end{align} 
    Thus, for $h$ small enough, the following bound holds:
    \begin{equation} \label{general DiscBound 2}
    \underset{n = 1, \dots, N}{\max} \int_{\Omega} (\rhodLh^{n})^{2} \dx + \frac{1}{\delta}\underset{n = 1, \dots, N}{\max} \int_{\Omega} \pi_{h}( [\rhodLh^{n}]_{-}^{2}) \dx + \Delta t \sum_{n =1}^{N} \int_{\Omega} |\nabla \rhodLh^{n}|^{2} \dx \leq C(L).
    \end{equation}

\end{lemma}
\begin{proof}
   By taking $\phi_{h} = \pi_{h}((G_{\delta}^{L})'(\rho_{h}^{n}))$ as test function in (\ref{general FullDiscWeakForm}), we have that
\begin{align*}
\int_{\Omega} \pi_h\bigg(\frac{\rhodLh^{n} - \rhodLh^{n-1}}{\Delta t} \pi_{h}((G_{\delta}^{L})'(\rhodLh^{n})) \bigg)\dx = &- \int_{\Omega}  \nabla \rhodLh^{n} \cdot \nabla \pi_{h}((G_{\delta}^{L})'(\rhodLh^{n})) \dx \\
& + \int_{\Omega} \Theta_{\delta}^{L}(\rhodLh^{n}) A(x) \nabla \cdLh^{n} \cdot \nabla \pi_{h}((G_{\delta}^{L})'(\rhodLh^{n})) \dx.
\end{align*}
Hence, by (\ref{PropFE1}), it follows that
\begin{align*} 
\int_{\Omega}\! \pi_h \bigg(\!\frac{\rhodLh^{n} - \rhodLh^{n-1}}{\Delta t} \pi_{h}((G_{\delta}^{L})'(\rhodLh^{n})) \!\bigg)\!\dx = &- \int_{\Omega} \nabla \rhodLh^{n} \cdot \nabla \pi_{h}((G_{\delta}^{L})'(\rhodLh^{n}))\! \dx\, \\
&+ \int_{\Omega}  A(x) \nabla \cdLh^{n} \cdot \nabla \rhodLh^{n}\! \dx.
\end{align*}
Next, we bound the left-hand side of this equality from below by using the convexity of $G_{\delta}^{L}$, which implies that
\begin{align*}
   \pi_h \bigg( (\rhodLh^{n} - \rhodLh^{n-1})\,\pi_{h}((G_{\delta}^{L})'(\rhodLh^{n})) \bigg) &=\pi_h \bigg( (\rhodLh^{n} - \rhodLh^{n-1})\,(G_{\delta}^{L})'(\rhodLh^{n}) \bigg)\\&\hspace{-7mm}\geq \pi_{h}\bigg(G_{\delta}^{L}(\rhodLh^{n}) - G_{\delta}^{L}(\rhodLh^{n-1}) \bigg)= \pi_{h}(G_{\delta}^{L}(\rhodLh^{n})) - \pi_{h} (G_{\delta}^{L}(\rhodLh^{n-1})),
\end{align*}
and therefore 
\begin{align} \frac{1}{\Delta t}  \int_{\Omega} (\pi_{h}(G_{\delta}^{L}(\rhodLh)) - \pi_{h} (G_{\delta}^{L}(\rhodLh^{n-1}))) \dx &+ \int_{\Omega} \nabla \rhodLh^{n} \cdot \nabla \pi_{h}((G_{\delta}^{L})'(\rhodLh^{n})) \dx \nonumber \\
& \label{general DiscBound1}  \leq \int_{\Omega} A(x) \nabla \cdLh^{n} \cdot \nabla \rhodLh^{n} \dx. \end{align}
This inequality completes the proof of the first bound \eqref{general DiscBound 1}. 

In this generalised model we have that  the right-hand side of inequality \eqref{general DiscBound1} doesn't fully contribute ``constructively'' to the desired bound. In fact, if one had $\LL=-\Delta$ it would be simple to show that $\int_{\Omega}  \nabla \cdLh^{n} \cdot \nabla \rhodLh^{n} \dx \leq 0$ (see \cite{Fronzoni2025}).   Using Young's inequality and the uniform ellipticity assumption of $\LN$ (recall that $\Lambda_1 |\mathrm{v}|^2 \leq A(x) \mathrm{v} \cdot \mathrm{v} \leq \Lambda_2 |\mathrm{v}|^2$, for all $\mathrm{v} \in \mathbb{R}^d$ and almost every $x \in \Omega$, for some constants $0 < \Lambda_1 \leq \Lambda_2$), we have that, for $\varepsilon>0$,
\begin{align} \label{young inequality grad rho grad c}
   \int_{\Omega} A(x)  \nabla c_{h, \delta, L}^{n} \cdot \nabla \rho_{h, \delta, L}^{n} \dx &\leq \frac{1}{\varepsilon} C \int_{\Omega} |\nabla c_{h, \delta, L}^{n}|^2 \dx + \varepsilon C\int_\Omega |\nabla \rhodLh^n|^2 \dx,
\end{align}
where $C$ stands for a real positive constant, independet of $h$.

We can now us the definition of the bilinear form \eqref{Bilinear From Operator LN}, together with the assumption on the positivity and  uniform ellipticity of the operator $\LN$ to have that 
\begin{align*}
  \int_{\Omega} |\nabla c_{h, \delta, L}^{n}|^2 \dx &\leq C \int_\Omega A(x) \nabla  c_{h, \delta, L}^{n} \cdot \nabla c_{h, \delta, L}^{n} \leq  C \left( \int_\Omega A(x) \nabla  c_{h, \delta, L}^{n} \cdot \nabla c_{h, \delta, L}^{n} \dx + \int_\Omega Q(x) (c_{h, \delta, L}^{n})^2 \dx \right) \\
  &= C a_{\LN}(c_{h, \delta, L}^{n}, c_{h, \delta, L}^{n}) = C \sum_{k=1}^{N_h} \lambda_k^h (c_{h, \delta, L}^{n})_k^2,
\end{align*}
where $(c_{h, \delta, L}^{n})_k = (c_{h, \delta, L}^{n}, \psi_k^h)$.
Therefore, with $(\beta^L(\rho_{h, \delta, L}^{n})^\diamond)_k = (\beta^L(\rho_{h, \delta, L}^{n})^\diamond, \psi_k^h)$, we have 
\begin{align*}
    \int_{\Omega} |\nabla c_{h, \delta, L}^{n}|^2 \dx &\leq C (\LL_{\mathrm{N}, h} c_{h, \delta, L}^{n}, c_{h, \delta, L}^{n}) = C (\LL_{\mathrm{N}, h}^{1-s} (\beta^L(\rhodLh^{n}))^{\diamond}, \LL_{\mathrm{N}, h}^{-s} (\beta^L (\rhodLh^{n}))^{\diamond}) \\
    &= C (\LL_{\mathrm{N}, h}^{1-2s} (\beta^L(\rhodLh^{n}))^{\diamond},  (\beta^L (\rhodLh^{n}))^{\diamond})  \\
    &= C \sum_{k=1}^{N_h} (\lambda_k^h)^{1-2s}(\beta^L(\rho_{h, \delta, L}^{n})^\diamond)_k^2 \leq C (\lambda_1^h)^{1-2s}\sum_{k=1}^{N_h} (\beta^L(\rho_{h, \delta, L}^{n})^\diamond)_k^2\\
    &= C (\lambda_1^h)^{1-2s} \| (\beta^L( \rhodLh^{n}))^{\diamond} \|_{L^2(\Omega)}^2,
\end{align*}
with $C>0$ a real constant, independent of $h$.
By the fact that $s\in (1/2,1)$, $\lambda_1^h \to \lambda_1$ as $h \to 0$ and by the definition of the cut-off function $\beta^L(\cdot)$, one has that, for $h$ small enough there exists a constant, only dependent on $L$ and $\Omega$, that we denote simply by $C(L)$ (absorbing $|\Omega|$ into $C(L)$ and writing $C(L)$ instead of $C(L)|\Omega|$), such that 
\begin{equation} \label{uniform ineq gradient c general}
    \int_{\Omega} |\nabla c_{h, \delta, L}^{n}|^2 \dx \leq C(L) .
\end{equation}

Using Corollary \ref{CorLip} with $\phi_{h} = \rhodLh^{n}$ and $g = (G_{\delta}^{L})'$, noting that $g$ is strictly monotonically increasing and the inverse $g^{-1}$ is Lipschitz continuous with Lipschitz constant $L$, we then have that
\begin{equation} \label{gen bound in L on grad rho}
\frac{1}{L}  \int_{\Omega} |\nabla \rhodLh^{n}|^{2} \dx \leq \int_{\Omega} \nabla \rhodLh^{n} \cdot \nabla \pi_{h}((G_{\delta}^{L})'(\rhodLh^{n})) \dx. 
\end{equation}
Therefore, if we pick $\varepsilon$ in \eqref{young inequality grad rho grad c} such that $\varepsilon C < L^{-1}$ (say $\varepsilon = (2LC)^{-1}$ for example) and we combine \eqref{general DiscBound1}, \eqref{young inequality grad rho grad c}, \eqref{uniform ineq gradient c general}, \eqref{gen bound in L on grad rho} we can sum the bounds (\ref{general DiscBound 1}) over $n = 1, \dots, m$ to deduce that, for each $m\in \{1,\ldots,N\}$,
\begin{align*}
     \int_{\Omega}\pi_{h}(G_{\delta}^{L}(\rhodLh^{m})) \dx + \Delta t \left( \frac{1}{L} - \varepsilon C \right) \sum_{n = 1}^{m} \int_{\Omega} |\nabla \rhodLh^{n}|^{2} \dx &\leq m \Delta t C(L) +  \int_{\Omega}\pi_{h}(G_{\delta}^{L}(\rho^0_h)) \dx\\
     &\leq T C(L) +  \int_{\Omega}\pi_{h}(G_{\delta}^{L}(\rho^0_h)) \dx \\
     & = C(L) +  \int_{\Omega}\pi_{h}(G_{\delta}^{L}(\rho^0_h)) \dx,
\end{align*}
where in the last inequality we have simply absorbed $T$ (the final time $T=N \Delta t $) in the contant $C(L)$ as $T$ is a datum of the problem. 
Let $\Omega_+:= \{ x \in \overline{\Omega}\,:\, \rho^n_{h,\delta,L}\geq 0\}$ and let $\Omega_{-}:= \overline\Omega \setminus \Omega_+$.
We shall further bound the first term on the left-hand side of this inequality from below by noting that, thanks to (\ref{ImpProp1}) and (\ref{PropFE2}), we have 
\begin{align*}
  \int_{\Omega}\pi_{h}(G_{\delta}^{L}(\rhodLh^{n})) \dx &= \int_{\Omega_{+}}    \pi_{h}(G_{\delta}^{L}(\rhodLh^{n})) \dx + \int_{\Omega_{-}} \pi_{h}(G_{\delta}^{L}(\rhodLh^{n})) \dx \\
  &\geq \frac{1}{4L} \int_{\Omega_{+}} \pi_{h}((\rhodLh^{n})^{2})\dx - C(L)|\Omega_+| + \frac{1}{2\delta}\int_{\Omega_{-}} \pi_{h}((\rhodLh^{n})^{2}) \dx \\
  &= \frac{1}{4L} \int_{\Omega} \pi_{h}((\rhodLh^{n})^{2})\dx - C(L)|\Omega_+| + \frac{1}{2\delta}\int_{\Omega_{-}} \pi_{h}((\rhodLh^{n})^{2}) \dx \\
  & \quad - \frac{1}{4L} \int_{\Omega_{-}} \pi_{h}((\rhodLh^{n})^{2})\dx\\
  &\geq \frac{1}{4L} \int_{\Omega} (\pi_{h}\rhodLh^{n})^{2} \dx + \frac{1}{4\delta} \int_{\Omega} \pi_{h}([\rhodLh^{n}]_{-}^{2}) \dx - C(L)|\Omega| \\
  &= \frac{1}{4L} \int_{\Omega} (\rhodLh^{n})^{2} \dx + \frac{1}{4\delta} \int_{\Omega} \pi_{h}([\rhodLh^{n}]_{-}^{2}) \dx - C(L)|\Omega|, 
\end{align*}
where in the transition to the fourth line we have used that $\delta <L$ and that $[\rhodLh^{n}]_{-}^{2}\equiv 0$ on $\Omega_+$. It remains to note that, because for all $\delta$ and $L$ such that $0<\delta<1<L$ we have that
\[ 0 \leq G^L_\delta(s) \leq \max\bigg\{1 - \frac{\delta}{2}, \frac{s^2-L^2}{2L} + (\log L - 1)s + 1\bigg\}, \quad \mbox{for all $s \geq 0$},\]
it follows from \eqref{inibound} that $0 \leq G^L_\delta(\rho^0_{h,\delta, L}) \leq C$, where $C=C(L)$ is a positive constant, independent of $h$ and $\delta$, whereby the same is true of  $\int_\Omega \pi_h(G^L_\delta(\rho^0_{h,\delta, L})) \dx$. As $\Omega$ is a fixed domain, we shall absorb $|\Omega|$ into $C(L)$ and write $C(L)$ instead of $C(L)|\Omega|$.
It follows from \eqref{inibound} that $0 \leq G^L_\delta(\rho^0_{h,\delta, L}) \leq C$, where $C=C(L)$ is a positive constant, independent of $h$ and $\delta$, whereby the same is true of  $\int_\Omega \pi_h(G^L_\delta(\rho^0_{h,\delta, L})) \dx$. From this,  the inequality (\ref{general DiscBound 2}) directly follows.

\end{proof}

We shall now prove an existence result that will ensure that a solution to the fully discrete problem \eqref{general FullDiscWeakForm} at time step $t_n$ exists, if we are provided with a solution at the previous time step $t_{n-1}$. We will prove the result using Brouwer's fixed point theorem. 

\begin{lemma} \label{general DiscExistLemma1}
For any $\Delta t > 0$, for $h$ small enough, given $\rhodLh^{n-1} \in V_{h}$, there exists at least one solution $(\rhodLh^{n}, c_{h, \delta, L}^{n}) \in V_{h} \times (V_{h} \cap L^{2}_{\diamond}(\Omega))$ to \eqref{general FullDiscWeakForm}.
\begin{proof}
 We begin by equipping $V_h$ with the inner product $((\cdot,\cdot))$, defined by
\[ ((\psi_h, \varphi_h)):= \int_\Omega \pi_h (\psi_h(x) \phi_h(x)) \dx,\quad \phi_h, \varphi_h \in V_h.\]
Next, we define the function $\mathcal{H}: V_{h} \to V_{h}$ such that, for any $\rho_{h} \in V_{h}$,
\[ ((\mathcal{H}(\rho_{h}), \phi_{h})) := \int_{\Omega} \pi_h\bigg(\frac{\rho_{h} - \rhodLh^{n-1}}{\Delta t} \phi_{h} \bigg)\dx + \int_{\Omega} \nabla \rho_{h} \cdot \nabla \phi_{h}\dx - \int_{\Omega} \Theta_{\delta}^{L}(\rho_{h}) A(x) \nabla c_{h} \cdot \nabla \phi_{h} \dx \quad \textrm{for all } \phi_{h} \in V_{h},  \] 
where $c_h \in V_h \cap L^2_\diamond(\Omega)$ satisfies $-\LL_{\mathrm{N},h}^{\s} c_{h} = (\beta^L( \rhodLh^{n}))^{\diamond}$. The construction of $\Theta_{\delta}^{L}$, together with  Lemma \ref{ThetaContLemma}, ensures that $\mathcal{H}$ is a continuous mapping. If a solution $\rhodLh^{n}$ to (\ref{general FullDiscWeakForm}) exists then it is a zero of $\mathcal{H}$, namely
\[ ((\mathcal{H}(\rhodLh^{n}), \phi_{h})) = 0 \quad \textrm{for all } \phi_{h} \in V_{h}. \] 

We shall therefore prove that $\mathcal{H}$ has a zero. 
For contradiction, let us assume that $\mathcal{H}$ has no zero for any $\gamma \in \mathbb{R}_{>0}$ in the ball 
$B_{\gamma} = \{ \psi_{h} \in V_{h} \textrm{ such that } \|\psi_{h}\|_{L^{2}(\Omega)} \leq \gamma \}$. We define the function $\mathcal{E}_{\gamma}: B_{\gamma} \to B_{\gamma}$ by
\[ \mathcal{E}_{\gamma}(\psi_{h}) = - \gamma \frac{\mathcal{H}(\psi_{h})}{\| \mathcal{H}(\psi_{h})\|_{L^{2}(\Omega)}}. \]

By Brouwer's fixed point theorem $\mathcal{E}_{\gamma}$ has at least one fixed point $\rho_{h}^{\gamma}$ in $B_{\gamma}$, and this means that $\|\rho_{h}^{\gamma}\|_{L^{2}(\Omega)} = \|\mathcal{E}_{\gamma}(\rho_{h}^{\gamma})\|_{L^{2}(\Omega)} = \gamma$. Let $c_{h}^{\gamma}$ be the solution of the equation $-\LL_{\mathrm{N},h}^{\s} c_{h} = (\beta^L( \rho_h^\gamma))^{\diamond}$ in $\Omega$. Choosing $\phi_{h} = \pi_{h}((G_{\delta}^{L})'(\rho_{h}^{\gamma}))$ as test function, by (\ref{ImpProp1}) and inequality \eqref{PropFE2}, together with the trivial equality $\pi_h \rho^\gamma_h = \rho^\gamma_h$, we have that
\begin{align*}
 ((\mathcal{H}(\rho_{h}^{\gamma}), \pi_{h}((G_{\delta}^{L})'(\rho_{h}^{\gamma})))) 
  &\leq -  \frac{\|\mathcal{H}(\rho_{h}^{\gamma})\|_{L^{2}(\Omega)}}{\gamma} \bigg( \int_{\Omega} \bigg( \frac{(\rho_{h}^{\gamma})^{2}}{4L} - C(L) \bigg) \dx \bigg) \\ &=- \frac{\|\mathcal{H}(\rho_{h}^{\gamma})\|_{L^{2}(\Omega)}}{\gamma} \Big( \frac{\gamma^{2}}{4L} -  C(L) |\Omega| \Big),
 \end{align*}
 and thus for $\gamma  > [4L\, C(L)\, |\Omega|]^{1/2}$ we also have that 
 \begin{equation*}   
 ((\mathcal{H}(\rho_{h}^{\gamma}), \pi_{h}(G_{\delta}^{L})'(\rho_{h}^{\gamma}))) < 0. \end{equation*}
 
 On the other hand, using the same argument that we applied to obtain \eqref{general DiscBound 1}, we deduce that
 \begin{align}
 ((\mathcal{H}(\rho_{h}^{\gamma}), \pi_{h}(G_{\delta}^{L})'(\rho_{h}^{\gamma}))) &\geq \frac{1}{\Delta t} \int_{\Omega} \pi_h\big((\rho_{h}^{\gamma} - \rhodLh^{n-1}) \pi_{h} ((G_{\delta}^{L})'(\rho_{h}^{\gamma})) \big)\dx + \frac{1}{L}  \int_{\Omega} |\nabla \rhodLh^{n}|^{2} \dx \nonumber \\
 & \hspace{5mm} - \int_{\Omega} A(x) \nabla c_{h}^{\gamma} \cdot \nabla \rho_{h}^{\gamma} \dx.\label{general BrouwerEqRef1}
 \end{align}
For the first term on the right-hand side of \eqref{general BrouwerEqRef1} we can use the convexity of $G^L_\delta$ to deduce, that 
 \begin{align*}
 \frac{1}{\Delta t} \int_{\Omega} \pi_h\big((\rho_{h}^{\gamma} - \rhodLh^{n-1}) \pi_{h} ((G_{\delta}^{L})'(\rho_{h}^{\gamma})) \big)\dx &\geq \frac{1}{\Delta t}  \int_\Omega \left( \pi_h(G_\delta^L (\rho_h^\gamma ) - \pi_h (G_\delta^L(\rho_h^{n-1}))\right) \dx \\ 
 &\geq \frac{1}{\Delta t} \Big( \frac{\gamma^{2}}{4L} - C_1(L) |\Omega| \Big) - \frac{1}{\Delta t} \int_{\Omega} \pi_{h}( G_{\delta}^{L}(\rho_{h}^{n-1}) )\dx,
 \end{align*}
 for a constant $C_1(L)>0$, depending on $L$.
For the second and third term on the right-hand side of \eqref{general BrouwerEqRef1} we proceed with a strategy similar to the one applied in the proof of Lemma \ref{general Fully disc uniform estimates}. Using Young's inequality we have 
\begin{equation*}
   \frac{1}{L}  \int_{\Omega} |\nabla \rhodLh^{n}|^{2} \dx \nonumber  - \int_{\Omega} A(x) \nabla c_{h}^{\gamma} \cdot \nabla \rho_{h}^{\gamma} \dx \geq \left( \frac{1}{L} - {\varepsilon}C  \right) \int_\Omega |\nabla \rhodLh^n| \dx - \frac{1}{\varepsilon} C\int_{\Omega} |\nabla c_{h, \delta, L}^{n}|^2, 
\end{equation*}
for any $\varepsilon>0$ and some $C>0$ and then, choosing $\varepsilon = \frac{1}{2LC} $, by the definition of $c_h$ as a solution of $-\LL_{\mathrm{N},h}^{\s} c_{h} = (\beta^L( \rhodLh^{n}))^{\diamond}$, together with the computations in the proof of Lemma \ref{general Fully disc uniform estimates}, used to show \eqref{uniform ineq gradient c general}, we have 
\begin{equation*}
     \frac{1}{L}  \int_{\Omega} |\nabla \rhodLh^{n}|^{2} \dx   - \int_{\Omega} A(x)  \nabla c_{h}^{\gamma} \cdot \nabla \rho_{h}^{\gamma} \dx \geq \frac{1}{2L} \int_{\Omega} |\nabla \rhodLh^{n}|^{2} \dx - C_2(L),
\end{equation*}
for a constant $C_2(L)$, depending on $L$.
 
Therefore,
\begin{align*}
 ((\mathcal{H}(\rho_{h}^{\gamma}), \pi_{h}((G_{\delta}^{L})'(\rho_{h}^{\gamma})))) &\geq \frac{1}{\Delta t} \Big( \frac{\gamma^{2}}{4L} - C_1(L) |\Omega| \Big) - \frac{1}{\Delta t} \int_{\Omega} \pi_{h} (G_{\delta}^{L}(\rho_{h}^{n-1})) \dx - C_2(L).
\end{align*}
Thus, for $\gamma  > [4L\, (C_1(L)\, |\Omega| + \| \pi_{h} (G_{\delta}^{L}(\rho_{h}^{n-1}))\|_{L^{1}(\Omega)}) + \Delta t C_2(L)]^{1/2}$ we have 
\[  ((\mathcal{H}(\rho_{h}^{\gamma}), \pi_{h}((G_{\delta}^{L})'(\rho_{h}^{\gamma})))) > 0, \]
and hence we have arrived at a contradiction;  therefore the function $\mathcal{H}$ has a zero, which means that at time $t_{n}$ there exists a solution $\rhodLh^{n}\in V_h$ to the fully discrete scheme (\ref{general FullDiscWeakForm}). Having shown the existence of $\rho_{h, \delta, L}^{n} \in V_{h}$, the existence of an associated (unique) $c_{h, \delta, L}^{n} \in V_{h} \cap L^{2}_{\diamond}(\Omega)$ follows from \eqref{general SpaceDiscFracPois}.
\end{proof}
\end{lemma}

\subsection{Passage to the limit $\delta, h \to 0_+$ }

We need a result on the convergence in space of the finite element approximation of $\LN$ \eqref{Discrete FracPowerOp}. We refer to \cite[Theorem 4.3]{bonito2015numerical} and  \cite[Theorem 6.2]{bonito2017numerical} for details and we report here the general statement applicable to our case.

\begin{prop}[Convergence of the spatial discretization of the fractional powers of $\LN$] \label{DiscFracOpConv}
Let $\alpha \in (0, 1]$ be the elliptic regularity order of the operator $\LN.$
Suppose that $s \in (1/2,1)$, $\sigma \in [0,1]$ and $f \in \mathcal{H}^{2\sigma}_{\diamond}(\Omega)$. Then, there exists a positive constant $C$ that is independent of $h$ and $\sigma$ such that 
\begin{equation*} 
    \|\LN^{-\s}f - \LL_{\mathrm{N},h}^{-\s} \pi_{h}f\|_{L^{2}(\Omega)} \leq C \varepsilon(h) \|f\|_{\mathcal{H}^{2\sigma}(\Omega)} \quad \textrm{for all } f \in \mathcal{H}^{2\sigma}_{\diamond}(\Omega),
\end{equation*}
where 
\begin{equation*} 
\varepsilon(h) = \left\{ \begin{array}{ll} 
h^{2\alpha} & \text{if } \alpha \geq s \textrm{ and } s + \sigma > \alpha, \\
h^{2\alpha} & \text{if } \alpha < s,
\\ \log(h^{-1}) h^{2(s+\sigma)} & \textrm{if } \s + \sigma < \alpha.
\end{array} \right.
\end{equation*}
\end{prop}

\begin{theo} \label{general convergenceFEMTheo}
The initial data $\{ \rho^0_h \}_{h > 0} $ defined in \eqref{InValDiscProb} are such that, for $\Delta t$ and $L$ fixed, as $h \to 0_{+}$ we have that
\begin{equation} \label{general Conv0}
    \rho^0_h \to \rho^{0}_{L} = \rho^0 := \rho^0(\Delta t) \quad \text{strongly in } L^{2}(\Omega). 
\end{equation}
    Furthermore there exists a subsequence of $\{ \rhodLh^{n} \}_{\delta, h >0}$, a nonnegative function $\rho_{L}^{n} \in V=H^1(\Omega)$ and $c_{L}^{n} \in H^{1}_{\diamond}(\Omega)$ such that the following convergence results hold, for each $n \in \{1,\ldots,N\}$, as $\delta, h \to 0_{+}$:
    \begin{subequations} \label{general Convergence}
        \begin{alignat}{2}
            \label{general Conv2}
            \nabla \rhodLh &\to \nabla \rho_{L}^{n} &&\quad\text{weakly in } L^{2}(\Omega;\mathbb{R}^d), \\
            \label{general Conv3}
            \rhodLh^{n} &\to \rho_{L}^{n} &&\quad\text{strongly in } L^{2}(\Omega), \\
            \label{general Conv4}
            \Theta_{\delta}^{L}(\rhodLh^{n}) &\to \beta^{L}(\rho^{n}_{L}) I&&\quad \text{strongly in } L^{2}(\Omega;\mathbb{R}^{d\times d}), \\
            \label{general Conv5}
            \cdLh^{n} &\to c_{L}^{n} &&\quad \text{strongly in } L^{2}_\diamond(\Omega;\mathbb{R}^d), \\
            \label{general Conv6}
              \nabla \cdLh^{n} &\to  \nabla c_{L}^{n} &&\quad \text{weakly in }  L^{2}_\diamond(\Omega;\mathbb{R}^d).
        \end{alignat}
    \end{subequations}
    Moreover $\{ \rho^{n}_{L} \}_{n=1, \dots, N}$ solves the problem \eqref{general WeakFormDisctime},  $\rho^{n}_{L} \geq 0$ a.e.~on $\Omega$ for all $n=1, \dots, N$, and, given $\rho^{0}_L=\rho^0$ such that $\frac{1}{|\Omega|} \int_{\Omega} \rho^{0} \dx = 1$, one has $\frac{1}{|\Omega|} \int_{\Omega} \rho_{L}^{n} \dx = 1$ for all $n = 1, \dots, N$.
\begin{proof}
    The proof of \eqref{general Conv0} follows the same way as that of the analogous result in the proof of \cite[Theorem 3.11]{Fronzoni2025}, for our choice of the initial datum $\rho_h^0$ for \eqref{general FullDiscWeakForm}. 
    
    Now, let $n \in \{1,\ldots,N\}$. The weak convergence \eqref{general Conv2} and the strong convergence \eqref{general Conv3} are implied by the inequality  \eqref{general DiscBound 2}. The nonnegativity of the limit function $\rho_{L}^{n}$ follows from \eqref{general Conv3} and the second bound in \eqref{general DiscBound 2}.
    
     Combining the second and third bounds in \eqref{general DiscBound 2} and Lemma \ref{LemmaMatrix} we have the result \eqref{general Conv4}.  
        
        The convergence result \eqref{general Conv5} can be proved in the following way. We have, by use of the stability inequality \eqref{general StabFracNeu} satisfied by the solution of the fractional Poisson equation, that
        \begin{align*}
&\|\LN^{-\s} (\beta^L(\rho_{L}^{n}))^{\diamond} - \LL_{\mathrm{N}, h}^{-\s} (\beta^L(\rhodLh^{n}))^{\diamond}\|_{L^{2}(\Omega)}\\ &\quad \leq
            \| \LN^{-\s}((\beta^L(\rho_{L}^{n}))^{\diamond} - (\beta^L(\rhodLh^{n}))^{\diamond})\|_{L^{2}(\Omega)}
            + \| (\LN^{-\s} - \LL_{\mathrm{N},h}^{-\s})(\beta^L(\rhodLh^{n}))^{\diamond}\|_{L^{2}(\Omega)} \\
            &\quad \leq C \|(\beta^L(\rho_{L}^{n}))^{\diamond} - (\beta^L(\rhodLh^{n}))^{\diamond}\|_{L^{2}(\Omega)} + \|(\LN^{-\s} - \LL_{\mathrm{N},h}^{-\s})(\beta^L(\rhodLh^{n}))^{\diamond}\|_{L^{2}(\Omega)},
        \end{align*}
        and by \eqref{general Conv3}, continuity of $\beta^L(\cdot)$ and the convergence result stated in Proposition \ref{DiscFracOpConv} we have \eqref{general Conv5}; here and henceforth $C$ signifies a generic positive constant, independent of $\delta$ and $h$.

        To prove \eqref{general Conv6} we need a uniform bound on the gradient of $\cdLh^{n}$. This fact, under the assumption of $s \in (1/2, 1)$ and using the equality \eqref{general SpaceDiscFracPois} is a direct consequence, as shown with \eqref{uniform ineq gradient c general} in the proof of Lemma \ref{general FullDiscWeakFormMain}.
Having shown that $\nabla c^n_{h,\delta,L}$ is uniformly bounded in $L^2(\Omega;\mathbb{R}^d)$ as $h, \delta \rightarrow 0_+$, the weak convergence result \eqref{general Conv6} follows from \eqref{general Conv5} thanks to the uniqueness of the weak limit. 
        
        We combine \eqref{general Conv2}--\eqref{general Conv4} to pass to the limit as $\delta, h \to 0_{+}$ in \eqref{general FullDiscWeakForm} with $\phi_{h}=\pi_{h} \phi$, for $\phi \in C^{\infty}(\overline{\Omega})$ together with the strong convergence of $\pi_h \phi$ to $\phi$ in the norm of $W^{1,\infty}(\Omega)$ as $h \rightarrow 0_+$
        (cf., for example, inequality (4.4.29) in \cite{BreSco94} with $s=1$, $m=2$, $l=1$, $r=0$ and $p=\infty$ there), to obtain equation \eqref{general WeakFormDisctime}. The fact that $C^{\infty}(\overline{\Omega})$ is dense in $V=H^1(\Omega)$ then concludes the argument. The nonnegativity of $\rhodLh^{n}$ on $\overline\Omega$ implies the nonnegativity of $\rhoL^{n}$ a.e.~on $\Omega$ for $n=1, \dots, N$. Finally, we note that if we choose $\phi\equiv 1$ as a test function in \eqref{general WeakFormDisctime} we have conservation of mass, i.e., $\int_\Omega \rho^n_L(x) \dx = \int_\Omega \rho^0 \dx = 1$ for $\rhoL^{n}$, $n=1, \dots, N$. 
    \end{proof}
\end{theo}

\section{The semidiscrete-in-time approximation} \label{sec:4}

We recall that $N \in \mathbb{N}_{\geq 2}$ and $\Delta t := T/N$. We define
\begin{subequations}
    \label{LinInterptime}
\begin{equation}
    \rhoL^{\Delta t}(\cdot, t) := \frac{t - t_{n-1}}{\Delta t} \rhoL^{n}(\cdot) + \frac{t_{n} - t}{\Delta t } \rhoL^{n-1}(\cdot), \quad t \in [t_{n-1}, t_{n}], \quad n = 1, \ldots, N,
\end{equation}
that is, the continuous piecewise affine interpolant in time of the sequence of discrete-in-time approxi\-mations $\{ \rhoL^{n} \}_{n =1, \ldots, N}$, in conjunction with the notation 
\begin{equation}
    \rhoL^{\Delta t, +}(\cdot, t) := \rhoL^{n}(\cdot), \quad \rhoL^{\Delta t, -}(\cdot, t) := \rhoL^{n-1}(\cdot), \quad t \in (t_{n-1}, t_{n}], \quad n =1, \ldots, N.
\end{equation}
\end{subequations}
We shall adopt $\rhoL^{\Delta t (,\pm)}$ as a collective symbol for $\rhoL^{\Delta t}$, $\rhoL^{\Delta t, \pm}$.

We can write, \eqref{general WeakFormDisctime} summed through $n = 1, \ldots, N$ as 
\begin{subequations} \label{general WeakFormSemiDiscTimeLMain}
\begin{gather}
    \text{Find } \rhoL^{\Delta t}(\cdot,t) \in V =H^1(\Omega)\, \mbox{ for } t\in (0, T], \text{ such that } \nonumber \\
    \int_{0}^{T} \int_{\Omega} \frac{\partial \rhoLDt}{\partial t} \phi \dx = - \int_{0}^{T} \int_{\Omega}  \nabla \rhoLDtp \cdot \nabla \phi \dx + \int_{0}^{T} \int_{\Omega} \beta^{L}(\rhoLDtp) A(x) \nabla c_{L}^{\Delta t, +} \cdot \nabla \phi \dx  \label{general WeakFormSemiDiscTimeL} \\ \quad \!\!\text{for all } \phi \in L^1(0,T;V),  \end{gather}
    subject to the initial condition $\rho_{L}^{\Delta t}(x, 0) = \rho^{0}(x)$ for a.e.~$x \in \Omega$, where $c_{L}^{\Delta t, +}(\cdot,t) \in \mathcal{H}^1_\diamond(\Omega)$ for all $t \in (0,T]$ and satisfies
    \begin{equation} \label{general FracPoiDiscTime} -\LN^{\s} c_{L}^{\Delta t, +}(\cdot,t) = (\beta^L(\rhoLDtp))^\diamond(\cdot,t)\quad \mbox{in $\Omega$},
\end{equation}
\end{subequations}
for all $t \in (0,T]$, and $\beta^L(s):=\min\{s,L\}$ for $s \in \mathbb{R}$.

We remark that, also in this more general setting, since $(\rhoL^{\Delta t})^{\diamond}$ is equal either to $\rhoL^{\Delta t}$ (in case $Q$ is not a.e. identically zero on $\Omega$)  or to $(\rhoL^{\Delta t})^{\diamond} = \rhoL^{\Delta t} - \int_\Omega \rhoL^{\Delta t} \dx$ (in case $Q$ is a.e. identically zero on $\Omega$), we have that
\begin{equation*}
  (\rhoL^{\Delta t})^{\diamond}(\cdot,t)  := \rhoLDt - \frac{1}{|\Omega|}\int_{\Omega} \rhoL^{\Delta t} \dx = \frac{t - t_{n-1}}{\Delta t} (\rhoL^{n})^{\diamond}(\cdot) + \frac{t_{n} - t}{\Delta t } (\rhoL^{n-1})^{\diamond}(\cdot), \quad t \in [t_{n-1}, t_{n}], \quad n= 1 \ldots, N,
\end{equation*}
and therefore we have 
\begin{equation} \label{general PropPseudoTimeDerAverage}
    \frac{\partial (\rhoLDt)^{\diamond}}{\partial t}(\cdot, t) = \frac{(\rhoLDtp)^{\diamond}(\cdot,t) - (\rhoLDtm)^{\diamond}(\cdot,t)}{\Delta t}= \frac{\partial \rho_L^{\Delta t}}{\partial t}(\cdot,t), \quad t \in (t_{n-1}, t_{n}], \quad n= 1 \ldots, N. 
\end{equation}

\begin{remark} \label{uniform control gradient c with large s}
We first observe that, in this case, we can still  obtain a bound on the $H^1(\Omega)$ seminorm of $c_{L}^{\Delta t, +}$, dependent on $\rhoLDtp$.
    We can do this using the definition of the bilinear form associated to $\LN$ \eqref{Bilinear From Operator LN}, together with the restriction on the fractional order $s\in(1/2, 1)$ and the assumption on the positivity and  uniform ellipticity of the operator $\LN$. In fact, we have 
\begin{align*}
    \int_{\Omega} |\nabla c_{L}^{\Delta t, +}|^2 \dx &\leq C \int_\Omega A(x) \nabla c_{L}^{\Delta t, +} \cdot \nabla c_{L}^{\Delta t, +} \dx \leq C \left( \int_\Omega A(x) \nabla  c_{L}^{\Delta t, +} \cdot \nabla c_{L}^{\Delta t, +} \dx + \int_\Omega Q(x) (c_{L}^{\Delta t, +})^2 \dx \right) \\ & \leq C (\LL_{\mathrm{N}} c_{L}^{\Delta t, +}, c_{L}^{\Delta t, +}) = C (\LL_{\mathrm{N}}^{1-s} (\beta^L(\rhoLDtp))^\diamond, \LL_{\mathrm{N}}^{-s} (\beta^L(\rhoLDtp))^\diamond) \\
    &= C (\LL_{\mathrm{N}}^{1-2s} (\beta^L(\rhoLDtp))^\diamond,  (\beta^L(\rhoLDtp))^\diamond)  \\
    &= C \sum_{k=1}^{\infty} \lambda_k^{1-2s}(\beta^L(\rhoLDtp))^\diamond)_k^2 \leq C \lambda_1^{1-2s}\sum_{k=1}^{N_h} (\beta^L(\rhoLDtp))^\diamond)_k^2\\
    &= C \lambda_1^{1-2s} \| (\beta^L( \rhoLDtp))^{\diamond} \|_{L^2(\Omega)}^2,
\end{align*}
where $C$ stands for a positive constant independent of $L$ and $\Delta t$. Therefore for a positive constant $C>0$ we have that  
\begin{equation} \label{uniform ineq gradient c general Disctime}
    \int_{\Omega} |\nabla c_{L}^{\Delta t, +}|^2 \dx \leq C \lambda_1^{1-2s} L |\Omega| .
\end{equation}
\end{remark}

As we anticipated, we want to make \eqref{general WeakFormSemiDiscTimeLMain} independent of the parameter $L$.
In order to do this, we shall now prove a uniform bound in time on the $L^{\infty}(\Omega)$ norm of the solution $\rhoLDt$. This property will allow us to eliminate the parameter $L$. The result will be a consequence of the assumptions that we have made for the operator $\LL$, in particular the fact that the operator has been assumed to be sub-Markovian and hence $L^\infty$-contractive.

\begin{lemma} \label{general InfNormDecayDiscTime}
Let $\rhoLDtpm$ be defined as in \eqref{LinInterptime}. The following bound holds true
\begin{equation} \label{general InfNormBoundDiscTime}
  \sup_{t \in (0, T)} \| \rhoLDtpm(t) \|_{L^{\infty}(\Omega)} \leq \| \rho_{0}\|_{L^{\infty}(\Omega)}. 
\end{equation}
\end{lemma}

\begin{proof}
The strategy of the proof starts with a similar idea to \cite[Section 4.2]{caffarelli2013regularity}, which goes back to De Giorgi's approach \cite{Caffarelli2010}. However, the procedure here is adapted to the discrete setting to which we are confined. We start by taking a smooth regularisation of the function $\zeta_M: \mathbb{R}_{\geq 0} \to \mathbb{R}$,  
\begin{equation*}
    \zeta_M(z) = \max( z - M, 0),
\end{equation*}
defined as $\zeta_M^{\epsilon} = \zeta_M \ast \theta_\epsilon$, where $\theta_\epsilon$ is a mollifier. By \cite[Theorem 4.22]{Brezis2011} we have that $\zeta_M^{\epsilon} \to \zeta_M$ in $L^p(\mathbb{R}_{\geq 0})$, $p\geq 1$ as $\epsilon \to 0$.

We use $\chi_{[0, t_n]} \zeta_M^{\epsilon}(\rhoLDtp)$, with $n=1, \dots, N$,
as a test function in \eqref{general WeakFormSemiDiscTimeL}. We then have \begin{align} 
    \int_{0}^{t_n} \int_{\Omega} \frac{\partial \rhoLDt}{\partial t}  \zeta_M^\epsilon(\rho_L^{\Delta t, +}) \dx \dt = &- \int_{0}^{t_n} \int_{\Omega}  \nabla \rhoLDtp \cdot \nabla  \zeta_M^\epsilon(\rho_L^{\Delta t, +}) \dx \dt \nonumber \\
    &+ \int_{0}^{t_n} \int_{\Omega} \beta^{L}(\rhoLDtp) A(x) \nabla c_{L}^{\Delta t, +} \cdot \nabla  \zeta_M^\epsilon(\rho_L^{\Delta t, +}) \dx \dt. \label{Linfty eq1}
\end{align}

By \cite[Lemma 9.1]{Brezis2011}, we have $(\zeta_M^\epsilon)'(z) = (\zeta_M' \ast \theta_\varepsilon)(z)$ in the sense of weak derivatives and therefore, since $\zeta_M'(z) =\chi_{(M, \infty)}(z)$, we have that $(\zeta_M^\epsilon)'(z) \geq 0$.
Thus the first term on the right-hand side of \eqref{Linfty eq1} is 
\begin{equation*}
   -\int_{0}^{t_n} \int_{\Omega}  \nabla \rhoLDtp \cdot \nabla  \zeta_M^\epsilon(\rho_L^{\Delta t, +}) \dx \dt  = - \int_0^{t_n} \int_\Omega (\zeta_M^\epsilon)'(\rho_L^{\Delta t, +})  |\nabla \rhoLDtp|^2 \dx \dt \leq 0.
\end{equation*}

For the second term on the right-hand side of \eqref{Linfty eq1} we have 
\begin{align*}
    \int_{0}^{t_n} \int_{\Omega} \beta^{L}(\rhoLDtp) A(x) \nabla c_{L}^{\Delta t, +} \cdot \nabla  \zeta_M^\epsilon(\rho_L^{\Delta t, +}) \dx \dt &= \int_0^{t_n} \int_\Omega (\zeta_M^\epsilon)'(\rho_L^{\Delta t, +}) \beta^L(\rho_L^{\Delta t, +}) A(x) \nabla c_{L}^{\Delta t, +} \cdot \nabla \rho_L^{\Delta t, +} \dx \dt  \\
    & = \int_0^{t_n} \int_\Omega  A(x) \nabla c_{L}^{\Delta t, +} \cdot \nabla H_M(\rho_L^{\Delta t, +}) \dx \dt,    
\end{align*}
where $H_{M}(z) = \int_0^z \beta^L(\omega) (\zeta_M^\epsilon)'(\omega) \textrm{d}\omega$. Notice that, because $\rho_L^{\Delta t, +}$ is nonnegative and $(\zeta_M^\epsilon)'(z) \geq 0$, $H_M(\cdot)$ is nonnegative and monotonically increasing.
 
Notice now that 
\begin{align*}
    \int_\Omega  A(x) \nabla c_{L}^{\Delta t, +} \cdot \nabla H_M(\rho_L^{\Delta t, +}) \dx = a_{\LN}( c_{L}^{\Delta t, +}, H_M(\rho_L^{\Delta t, +})) - \int_\Omega Q(x) c_{L}^{\Delta t, +}  H_M(\rho_L^{\Delta t, +}) \dx  
\end{align*}
and recall that
\begin{align*}
    c_{L}^{\Delta t, +} = - \LN^{-s}(\beta^L(\rhoLDtp))^\diamond = - \int_\Omega W_t(x,y) (\beta^L(\rhoLDtp))^\diamond \dy,
\end{align*}
where $W_t(x,y)$ is the heat kernel associated to $\LN$. 
Moreover either $Q \equiv 0$ or $(\beta^L(\rhoLDtp))^\diamond = \beta^L(\rhoLDtp)$ and therefore the term $\int_\Omega Q(x) c_{L}^{\Delta t, +}  H_M(\rho_L^{\Delta t, +}) \dx $ is always nonnegative. 

We have that
\begin{align}
    \int_\Omega  A(x) \nabla c_{L}^{\Delta t, +} \cdot \nabla H_M(\rho_L^{\Delta t, +}) \dx &= a_{\LN}( c_{L}^{\Delta t, +}, H_M(\rho_L^{\Delta t, +})) - \int_\Omega Q(x) c_{L}^{\Delta t, +}  H_M(\rho_L^{\Delta t, +}) \dx  \nonumber \\
    &= (\LN c_{L}^{\Delta t, +}, H_M(\rho_L^{\Delta t, +})) - \int_\Omega Q(x) c_{L}^{\Delta t, +}  H_M(\rho_L^{\Delta t, +}) \dx \nonumber \\
    &= - (\LN^{1-s} (\beta^L(\rhoLDtp))^\diamond, H_M(\rho_L^{\Delta t, +})) - \int_\Omega Q(x) c_{L}^{\Delta t, +}  H_M(\rho_L^{\Delta t, +}) \dx . \label{composite equality Linfty bound}
\end{align}

Using that either \[(\beta^L(\rhoLDtp))^\diamond = \beta^L(\rhoLDtp) \quad \textrm{or} \quad (\beta^L(\rhoLDtp))^\diamond = (\beta^L(\rhoLDtp))^\ast = \beta^L(\rhoDtp) - \int_\Omega \beta^L(\rhoDtp) \dx \] and the definition of the heat semigroup \eqref{GeneralHE}, we then have that either \[ \mathrm{e}^{-\tau \LN} (\beta^L(\rhoDtp))^\diamond=\mathrm{e}^{-\tau \LN} \beta^L(\rhoDtp) \quad \textrm{or} \quad \mathrm{e}^{-\tau \LN} (\beta^L(\rhoDtp))^\diamond=\mathrm{e}^{-\tau \LN} \beta^L(\rhoDtp) - \int_\Omega \beta^L(\rhoDtp) \dx.\]
Thus, denoting $\beta=2(1-s)$, using the representation formula \eqref{semigroupformula} for $\LN^s$, we can write 
\begin{align*}
(\LN^{\beta/2} c_{L}^{\Delta t, +}, H_M(\rho_L^{\Delta t, +})) & = \int_{\Omega} H_M(\rhoDtp(x)) \frac{1}{\Gamma(-\beta/2)} \int_{0}^{\infty} (\mathrm{e}^{-\tau \LN} (\beta^L(\rhoDtp))^\diamond(x) \\
&  \hspace{7mm} - (\beta^L(\rhoDtp))^\diamond(x)) \frac{1}{\tau^{1+\frac{\beta}{2}}} \,\dtau \dx \\
&\hspace{-10mm}= \int_{\Omega} H_M(\rhoDtp(x)) \frac{1}{\Gamma(-\beta/2)} \int_{0}^{\infty} (\mathrm{e}^{-\tau \LN} \beta^L(\rhoDtp)(x) - \beta^L(\rhoDtp)(x)) \frac{1}{\tau^{1+\frac{\beta}{2}}} \,\dtau \dx \\
&\hspace{-10mm}= \frac{1}{\Gamma(-\beta/2)} \int_{0}^{\infty} \bigg( \int_{\Omega} H_M(\rhoDtp(x)) \mathrm{e}^{-\tau\LN} \beta^L(\rhoDtp)(x) \dx \\
&\hspace{-10mm} \hspace{7mm} - \int_{\Omega} H_M(\rhoDtp(x)) \beta^L(\rhoDtp)(x) \dx \bigg) \frac{1}{\tau^{1+\frac{\beta}{2}}}\, \dtau\\
&\hspace{-10mm}= \frac{1}{\Gamma(-\beta/2)} \int_{0}^{\infty} \big(  (\mathrm{e}^{-\tau\LN} \beta^L(\rhoDtp), H_M(\rhoDtp) ) - ( \beta^L(\rhoDtp), H_M(\rhoDtp)) \big) \frac{1}{\tau^{1+\frac{\beta}{2}}}\, \dtau.
\end{align*}

Therefore, by plugging the heat kernel in the expression above, we deduce that 
\begin{align*}
(\LN^{\beta/2} c_{L}^{\Delta t, +}, H_M(\rho_L^{\Delta t, +})) &= \frac{1}{\Gamma(-\beta/2)} \int_{0}^{\infty} \int_{\Omega} \bigg( \int_{\Omega} W_{\tau}(x, y) \beta^L(\rhoDtp)(x) H_M(\rhoDtp(y)) \dx \\
& \hspace{7mm} - \beta^L(\rhoDtp)(y) H_M(\rhoDtp(y)) \bigg) \dy \,\frac{1}{\tau^{1+\frac{\beta}{2}}} \,\dtau \\
&\hspace{-18mm}= \frac{1}{\Gamma(-\beta/2)} \int_{0}^{\infty} \int_{\Omega} \bigg( \int_{\Omega} W_{\tau}(x, y)(\beta^L(\rhoDtp)(x) - \beta^L(\rhoDtp)(y)) H_M(\rhoDtp(y))\dx \\
&\hspace{-18mm}\quad + \beta^L(\rhoDtp)(y) H_M(\rhoDtp(y)) \bigg( \int_{\Omega} W_{\tau}(x, y)\dx - 1 \bigg) \bigg)\dy\, \frac{1}{\tau^{1+\frac{\beta}{2}}} \,\dtau \\
&\hspace{-18mm}= \frac{1}{\Gamma(-\beta/2)} \int_{0}^{\infty} \int_{\Omega} \int_{\Omega} W_{\tau}(x, y) (\beta^L(\rhoDtp)(x) - \beta^L(\rhoDtp)(y)) H_M(\rhoDtp(y)) \dy \dx \, \frac{1}{\tau^{1+\frac{\beta}{2}}} \,\dtau \\
&\hspace{-18mm}\quad + \frac{1}{\Gamma(-\beta/2)} \int_{0}^{\infty} \int_{\Omega} \beta^L(\rhoDtp)(y) H_M(\rhoDtp(y))(\mathrm{e}^{\tau \LN}1(y) - 1) \dy \,\frac{1}{\tau^{1+\frac{\beta}{2}}} \,\dtau,
\end{align*}
where in the passage from the second to the third line we have added and subtracted the term \newline $\frac{1}{\Gamma(-\beta/2)}  \int_0^\infty \int_\Omega \int_\Omega \beta^L(\rhoLDtp)(y) H_M(\rhoDtp(y)) W_t(x,y) \dx \dy  \,\, \tau^{-(1+\beta/2)} \dtau $.
By exchanging $x$ and $y$ and using the symmetry of the heat kernel we have also 
\begin{align*}
(\LN^{\beta/2} c_{L}^{\Delta t, +}, H_M(\rho_L^{\Delta t, +})) = &-  \frac{1}{\Gamma(-\beta/2)} \int_{0}^{\infty} \int_{\Omega} \int_{\Omega} W_{\tau}(x, y) (\beta^L(\rhoDtp)(x) \\
&\hspace{7mm}- \beta^L(\rhoDtp)(y)) H_M(\rhoDtp(x)) \dy\dx \, \frac{1}{\tau^{1+\frac{\beta}{2}}} \,\dtau \\
& +\frac{1}{\Gamma(-\beta/2)} \int_{0}^{\infty} \int_{\Omega} \rho(y) H_M(\rhoDtp(y))(\mathrm{e}^{\tau \LN}1(y) - 1) \dy\, \frac{1}{\tau^{1+\frac{\beta}{2}}} \,\dtau. 
\end{align*}
By summing the two expressions we then find that  
\begin{align}
(\LN^{\beta/2} c_{L}^{\Delta t, +}, H_M(\rho_L^{\Delta t, +})) = &-\frac{1}{2\Gamma(-\beta/2)} \Bigg( \int_{0}^{\infty} \int_{\Omega} \int_{\Omega} W_{\tau}(x, y) (\beta^L(\rhoDtp)(x) - \beta^L(\rhoDtp)(y))(H_M(\rhoDtp(x)) \nonumber \\
& \hspace{7mm} - H_M(\rhoDtp(y)) \dx \dy \,\frac{1}{\tau^{1+\frac{\beta}{2}}} \dtau \nonumber \\
& + \int_{0}^{\infty} \int_{\Omega} \beta^L(\rhoDtp)(y) H_M(\rhoDtp(y))(1- \mathrm{e}^{\tau \LN }1(y)) \dy \,\frac{1}{\tau^{1+\frac{\beta}{2}}} \,\dtau \Bigg). \label{composite equality Linfty bound 2}
\end{align}
Thanks to the definition of $\LN$ stated and \eqref{mass heat kernel} we have that $\mathrm{e}^{\tau \LN}1(y) \leq  1$. Moreover the monotonicity of $\beta^L(\cdot)$ and $H_M(\cdot)$ ensure the nonpositivity of both terms on the right-hand side of the above expression.

Thus, combining \eqref{composite equality Linfty bound} and \eqref{composite equality Linfty bound 2}, overall we have 
\begin{align}
    \int_{0}^{t_n} \int_{\Omega} \frac{\partial \rhoLDt}{\partial t}  \zeta_M^\epsilon(\rho_L^{\Delta t, +}) \dx = & - \int_0^{t_n} \int_\Omega (\zeta_M^\epsilon)'(\rho_L^{\Delta t, +})  |\nabla \rhoLDtp|^2 \dx \dt \nonumber \\
    & +\frac{1}{2\Gamma(-\beta/2)} \Bigg( \int_{0}^{\infty} \int_{\Omega} \int_{\Omega} W_{\tau}(x, y) (\beta^L(\rhoDtp)(x) - \beta^L(\rhoDtp)(y))(H_M(\rhoDtp(x)) \nonumber \\
    &- H_M(\rhoDtp(y)) \dx \dy \,\frac{1}{\tau^{1+\frac{\beta}{2}}} \dtau \nonumber \\
& + \int_{0}^{\infty} \int_{\Omega} \beta^L(\rhoDtp)(y) H_M(\rhoDtp(y))(1- \mathrm{e}^{\tau
 \LN}1(y)) \dy \,\frac{1}{\tau^{1+\frac{\beta}{2}}} \,\dtau \Bigg) \nonumber \\ 
 &- \int_\Omega Q(x) c_{L}^{\Delta t, +}  H_M(\rho_L^{\Delta t, +}) \dx . \label{overall equality Linfty lemma}
\end{align}

All the terms in the above expression are nonpostive expect for the last term,
\begin{equation*}
  - \int_0^{t_n}  \int_\Omega Q(x) c_L^{\Delta t, +} H_M(\rho_L^{\Delta t, +}) \dx \dt = \int_0^{t_n}  \int_\Omega Q(x) \LN^{-s} (\beta^L(\rho_L^{\Delta t, +}))^\diamond H_M(\rho_L^{\Delta t, +}) \dx \dt \geq 0.
\end{equation*}
If $Q=0$ a.e. in $\Omega$ the term vanishes and hence we have immediately that the whole expression \eqref{overall equality Linfty lemma} is nonpositive.
In case $Q$ is not vanishing, we need to absorb the term in the first term on the right-hand side of \eqref{overall equality Linfty lemma}. We have that $(\beta^L(\rho_L^{\Delta t, +}))^\diamond = \beta^L(\rho_L^{\Delta t, +})$
\begin{align*}
  \left| \int_0^{t_n}  \int_\Omega Q(x) c_L^{\Delta t, +} H_M(\rho_L^{\Delta t, +}) \dx \dt \right| &\leq  \| Q \|_{L^\infty(\Omega)} \int_0^{t_n}  \int_\Omega  \LN^{-s} \beta^L(\rho_L^{\Delta t, +}) L \rho_L^{\Delta t, +}  \dx \dt \\
  &\leq L \| Q \|_{L^\infty(\Omega)} \int_0^{t_n} \int_\Omega \LN^{-s} \rho_L^{\Delta t, +} \rho_L^{\Delta t, +} \dx \dt \\
  &\leq L \| Q \|_{L^\infty(\Omega)} \lambda_1^{-s} \int_0^{t_n} \int_\Omega (\rho_L^{\Delta t, +})^2 \dx \dt.
\end{align*}

On the other hand we have 
\begin{align*}
    - \int_0^{t_n} \int_\Omega (\zeta_M^\epsilon)'(\rho_L^{\Delta t, +})  |\nabla \rhoLDtp|^2 \dx &= - \frac{1}{2} \varepsilon \int_0^{t_n} \int_\Omega ((\zeta_M^\epsilon)'(\rho_L^{\Delta t, +}))^2 |\nabla \rhoLDtp|^2 \dx \dt \\
    &\quad - \frac{1}{2 \varepsilon} \int_0^{t_n} \int_\Omega |\nabla \rhoLDtp|^2 \dx \dt 
\end{align*}
and 
\begin{align*}
      \frac{1}{2 \varepsilon } \int_0^{t_n} \int_\Omega |\nabla \rhoLDtp|^2 \dx \dt  \geq \frac{1}{2 \varepsilon} \lambda_1 \int_0^{t_n} \int_\Omega  (\rhoLDtp)^2 \dx \dt.
\end{align*}

Therefore in the case $Q$ is non vanishing a.e, by choosing
\begin{equation} \label{eq: mnecessary diffusion for Linfty bound}
    \varepsilon < \frac{\lambda_1^{1+s}}{2 L \| Q \|_{L^\infty(\Omega)}}, 
\end{equation}
from \eqref{overall equality Linfty lemma} we have overall that 
\begin{equation} \label{mon equality 1}
       \int_{0}^{t_n} \int_{\Omega} \frac{\partial \rhoLDt}{\partial t}  \zeta_{M}^\epsilon(\rho_L^{\Delta t, +}) \dx \leq 0, \quad \textrm{for any } M 
\end{equation}

Consider now the function $Z_M: \mathbb{R}_{\geq 0} \to \mathbb{R}$
\begin{equation*}
    Z_M(z) = \frac{1}{2} \max( (z - M)^2, 0)
\end{equation*}
and its regularisation $Z_M^\epsilon = Z_M \ast \theta_\epsilon$;
the weak derivative of $Z$ is clearly $\zeta_M$ and therefore, as before, we have $(Z_M^\epsilon)'(z) = (Z_M' \ast \theta_\epsilon)(z) = (\zeta_M \ast \theta_\epsilon)(z)=\zeta_M^\epsilon(z) $.

 Using a Taylor expansion of $Z_M^\epsilon (\cdot)$ and noting that $t\in [0, T] \mapsto \rhoLDt(\cdot,t)$ is piecewise affine relative to the temporal partition $0 = t_{0} < t_{1} < \cdots < t_{N} = T$, we have that
        \begin{align*}
            \int_{0}^{t_n} \int_{\Omega} \frac{\partial \rhoLDt}{\partial t} \zeta_M^\varepsilon(\rhoLDtp) \dx\,\dtau &= \int_{\Omega} Z^\epsilon_M(\rhoLDtp(x,t_n)) \dx - \int_{\Omega} Z_M^\varepsilon(\rho^{0}(x)) \dx \\
            & \quad + \frac{1}{2 \Delta t} \int_{0}^{t_n} \int_{\Omega} (\zeta_M^\epsilon)'(\theta_n \rhoLDtp + (1-\theta_n) \rhoLDtm) (\rhoLDtp - \rhoLDtm)^{2} \dx\,\dtau,
        \end{align*}
for $\theta_n \in (0,1)$.

Thus, from \eqref{mon equality 1} and the fact that $(\zeta_M^\epsilon)' \geq 0$, we directly have 
\begin{equation} \label{monotonicity inequality 2}
    \int_{\Omega} Z^\epsilon_{M}(\rhoLDtp(x,t_n)) \dx - \int_{\Omega} Z_{M}^\varepsilon(\rho^{0}(x)) \dx   \leq 0.
\end{equation}
If we set $M = \| \rho_0 \|_{L^\infty(\Omega)}$, we  have 
\begin{equation*}
    \int_{\Omega} Z_{\| \rho_0 \|_{L^\infty(\Omega)}}^\epsilon(\rhoLDtp(x,t_n)) \dx - \int_{\Omega} Z_{\| \rho_0 \|_{L^\infty(\Omega)}}^\varepsilon(\rho^{0}(x)) \dx  \leq 0.
\end{equation*}
Recalling that for any $M>0$ we have $Z_M^{\epsilon} \to Z_M$ in $L^p(\mathbb{R}_{\geq 0})$, $p\geq 1$, as $\epsilon \to 0$ (c.f. \cite[Theorem 4.22]{Brezis2011}), it follows that 
\begin{equation*}
    \int_{\Omega} Z_{\| \rho_0 \|_{L^\infty(\Omega)}}(\rhoLDtp(x,t_n)) \dx \leq 0.
\end{equation*}
Using the fact that $Z_M \geq 0$ this then implies that 
\begin{equation*}
    \int_{\Omega} Z_{\| \rho_0 \|_{L^\infty(\Omega)}}(\rhoLDtp(x,t_n)) \dx = 0
\end{equation*}
Thus, we deduce that, for any $n=1, \dots, N$, \eqref{general InfNormBoundDiscTime} must hold. 
\end{proof}

\normalcolor
Therefore, we can drop immediately the upper cut-off parameter $L$ in \eqref{general WeakFormSemiDiscTimeLMain}, by choosing $L> \| \rho_0\|_{L^\infty(\Omega)}$ and consider the problem 
\begin{subequations} \label{general WeakFormSemiDiscTimeMain noL}
\begin{gather}
    \text{Find } \rho^{\Delta t}(\cdot,t) \in V =H^1(\Omega)\, \mbox{ for } t\in (0, T], \text{ such that } \nonumber \\
    \int_{0}^{T} \int_{\Omega} \frac{\partial \rhoDt}{\partial t} \phi \dx = - \int_{0}^{T} \int_{\Omega}  \nabla \rhoDtp \cdot \nabla \phi \dx + \int_{0}^{T} \int_{\Omega}\rhoDtp A(x) \nabla c^{\Delta t, +} \cdot \nabla \phi \dx  \label{general WeakFormSemiDiscTime noL} \\ \quad \!\!\text{for all } \phi \in L^1(0,T;V),  \end{gather}
    subject to the initial condition $\rho^{\Delta t}(x, 0) = \rho^{0}(x)$ for a.e.~$x \in \Omega$, where $c_{L}^{\Delta t, +}(\cdot,t) \in \mathcal{H}^1_\diamond(\Omega)$ for all $t \in (0,T]$ and satisfies
    \begin{equation} \label{general FracPoiDiscTime noL} -\LN^{\s} c^{\Delta t, +}(\cdot,t) = (\rhoDtp)^\diamond(\cdot,t)\quad \mbox{in $\Omega$},
\end{equation}
\end{subequations}
for all $t \in (0,T]$.

\begin{remark}
   We comment on a difference between the proof of Lemma \ref{general InfNormDecayDiscTime} and a similar result that can be obtained for the standard porous medium equation with a fractional pressure, given by the fractional Laplacian (i.e. $\LL = -\Delta$). For the latter the parabolic regularisation term is not necessary to prove the $L^\infty(\Omega)$ bound for the solution $\rhoDtp$; in fact in \cite{caffarelli2010nonlinear, caffarelli2010asymptotic} the authors proved this type of result for the solution of the porous medium equation with a fractional pressure on $\mathbb{R}^d$ \eqref{FracPor FracLap}, without the parabolic regularisation term in place. The same property was also showed subsequently in \cite{chen2022analysis}, without using the presence of standard diffusion. In our case, in the proof of Lemma \ref{general InfNormDecayDiscTime}, the presence of standard diffusion has been crucial in the case the source term $Q$ of the operator $\LN$ does not vanish a.e. in $\Omega$. More precisely the parabolic regularisation was necessary in controlling the quantity $\frac{\lambda_1^{1+s}}{2 L \| Q \|_{L^\infty(\Omega)}}$ in \eqref{eq: mnecessary diffusion for Linfty bound}. Clearly, the first eigenvalue depends on $Q$, and we can investigate further the quantity that gets controlled by the standard diffusion term in the equation. 
    
    Let $\{ \lambda_k\}_{k \geq 1}$ be the eigenvalues of the self-adjoint elliptic operator induced by the bilinear form  
    \begin{equation*}
        a(\phi, \psi) = \int_\Omega \nabla \phi \cdot \nabla \psi \dx, \qquad \phi, \psi \in H^1(\Omega).
    \end{equation*}
We can compare $\{ \lambda_k\}_{k \geq 1}$ with the eigenvalues $\{ \lambda_k\}_{k \geq 1}$ of the bilinear form $a_{\LN}(\cdot, \cdot)$, defined in \eqref{Bilinear From Operator LN}.
    For $\phi \in H^1(\Omega)$ with $\| \phi \|_{L^2(\Omega)} = 1$ we have 
    \begin{align*}
     a_{\LN}(\phi, \phi) &=  \int_\Omega A(x) \nabla \phi \cdot \nabla \phi \dx  + \int_\Omega Q(x) \phi^2 \dx \geq \Lambda_1 \int_\Omega  \nabla \phi \cdot \nabla \phi \dx + \inf_{x \in \Omega}  Q(x)  \\
     &\geq \Lambda_1 a(\phi, \phi) + \inf_{x \in \Omega}  Q(x),
    \end{align*}
    where $\Lambda_1$ is the uniform ellipticity constant of the operator \eqref{mytypeofoperator}.
    
    Using the version of Courant-Fisher min-max theorem for self-adjoint operators on Hilbert spaces (see \cite[Theorem XIII.1 and XIII.2]{Reed1978}), having assumed that $\{ \lambda_k \}_{k\geq 1}$ and  $\{ \lambda_k \}_{k \geq 1}$ are ordered, we have then that 
    \begin{equation}
        \lambda_k \geq \Lambda_1 \lambda_k + \inf_{x \in \Omega}  Q(x).
    \end{equation}

Therefore the quantity in \eqref{eq: mnecessary diffusion for Linfty bound} is such that 
\begin{equation*}
\frac{\lambda_1^{1+s}}{2 L \| Q \|_{L^\infty(\Omega)}} > \frac{(\Lambda_1 \lambda_1 + \inf_{x \in \Omega}  Q(x))^{1+s}}{2 \| Q \|_{L^\infty(\Omega)}}.
\end{equation*}
The above inequality shows that the more the difference $\| Q \|_{L^\infty(\Omega)}  - \inf_{x\in \Omega} Q(x)$ increases, the more the term needs to be balanced by part of the diffusion term. Similarly, the magnitude of $\Lambda_1$, related to the matrix $A(x)$, influences the amount of standard diffusion that is needed to balance the term.
Roughly speaking, a function $Q$ with high variations on $\Omega$ requires a larger portion of the standard diffusion term to ensure the convergence of the scheme. 
\end{remark}

We can now start to derive a bound that is uniform in $\Delta t$. Our candidate test function for \eqref{general WeakFormSemiDiscTime noL} will be $\phi = \chi_{[0, t]} G'(\rhoDtp)$, where $\chi_{[0,t]}$ denotes the characteristic function of the interval $[0,t]
$. We recall that  $s G''(s) = 1$  for all $s \in \mathbb{R}$.

Thus we take $\alpha \in (0,1)$ and we choose $\phi =\chi_{[0, t]} G'(\rhoDtp+ \alpha)$, with $t = t_{n}$, $n \in \{ 1, \ldots, N \}$. We have 
\begin{align*}
    \underbrace{\int_{0}^{T} \int_{\Omega} \frac{\partial \rhoDt}{\partial s} \chi_{[0,t_n]} G'(\rhoDtp + \alpha) \dx \dt}_{\circled{1}} &+ \underbrace{\int_{0}^{T} \int_{\Omega} \nabla \rhoDtp \cdot \nabla \chi_{[0,t_n]} G'(\rhoDtp + \alpha) \dx \dt}_{\circled{2}} \\
    &= \underbrace{\int_{0}^{T} \int_{\Omega} \rhoDtp A(x) \nabla c^{\Delta t, +} \cdot \nabla\chi_{[0,t_n]} G'(\rhoDtp + \alpha) \dx \dt}_{\circled{3}}. 
\end{align*}
Let us now manipulate the three terms one at a time. First, by Taylor expansion with a remainder of the function
\[ s \in [0, \infty) \mapsto G(s + \alpha) \in [0, \infty), \]
we have that, for any $b \in [0, \infty)$,
\[ (s-b) G'(s + \alpha) = G(s + \alpha) - G(b + \alpha) + \frac{1}{2}(s - b)^{2} G''(\theta s + (1-\theta)b + \alpha),\] with $\theta \in (0,1)$. Noting that for $t \in [0,T]$ the function $t \mapsto \rhoDt(\cdot,t)$ is, by definition, piecewise affine relative to the partition $0 = t_{0}<t_{1}< \cdots <t_{N} = T$ of the interval $[0,T]$, it follows that 
\begin{align*}
\circled{1} &= \int_{\Omega} G(\rhoDtp(x,t_n) + \alpha) \dx - \int_{\Omega} G(\rho^{0}(x) + \alpha) \dx \\
& \quad + \frac{1}{2 \Delta t} \int_{0}^{t_n} \int_{\Omega} G''(\theta \rhoDtp + (1-\theta) \rhoDtm + \alpha) (\rhoDtp - \rhoDtm)^{2} \dx\,\dtau.
\end{align*}
As $G''(s + \alpha) = \frac{1}{s+\alpha}$ for all $s \geq 0$ and \eqref{general InfNormBoundDiscTime} implies that $G''(\theta \rhoDtp + (1-\theta) \rhoDtm + \alpha) \geq \frac{1}{\| \rho_{0}\|_{L^{\infty}(\Omega)} + \alpha}$ almost everywhere in $\Omega\times (0,T)$ it follows that
\begin{align*}
    \circled{1} \geq \int_{\Omega} G(\rhoDtp(x,t_n) + \alpha) \dx &- \int_{\Omega} G(\rho^{0}(x) + \alpha) \dx \\
    &+ \frac{1}{2 \Delta t \, (\| \rho_{0}\|_{L^{\infty}(\Omega)} + \alpha)} \int_{0}^{t_n} \int_{\Omega} (\rhoDtp - \rhoDtm)^{2} \dx\,\dtau.
\end{align*}

For the second term we have 
\begin{align*}
    \circled{2} = \int_{0}^{t_n} \int_{\Omega} G''(\rhoDtp + \alpha) |\nabla \rhoDtp|^{2} \dx \, \,\dtau = \int_{0}^{t_n} \int_{\Omega} \frac{|\nabla \rhoDtp|^{2}}{\rhoDtp + \alpha} \dx \, \,\dtau,
\end{align*}
again, since $G''(s+\alpha) = \frac{1}{s+\alpha}$ for all $s \geq 0$.

For the third term we have
\begin{align*}
    \circled{3} &= \int_{0}^{t_n} \int_{\Omega} \rhoDtp(G^L)''(\rhoDtp + \alpha) A(x) \nabla c^{\Delta t, +} \cdot \nabla \rhoDtp \dx\,\dtau \\
    &= \int_{0}^{t_n} \int_\Omega \frac{\rhoDtp}{\rhoDtp + \alpha} A(x) \nabla c^{\Delta t, +} \cdot \nabla \rhoDtp \dx\,\dtau.
\end{align*}

Putting everything together we therefore have that  
\begin{align}
    \int_{\Omega} G(\rhoDtp(x,t_n) + \alpha) \dx &+ \frac{1}{2 \Delta t \, (\| \rho_{0}\|_{L^{\infty}(\Omega)} + \alpha)} \int_{0}^{t_n} \int_{\Omega} (\rhoDtp - \rhoDtm)^{2} \dx\,\dtau \nonumber \\
    &+ \int_{0}^{t_n} \int_{\Omega} \frac{|\nabla \rhoDtp|^{2}}{\rhoDtp + \alpha} \dx \, \,\dtau \nonumber  \\
    &- \int_{0}^{t_n} \int_{\Omega}\frac{\rhoDtp}{\rhoDtp + \alpha} A(x) \nabla c^{\Delta t, +} \cdot \nabla \rhoDtp  \dx \, \,\dtau  \label{StabEstAlpha} \\ &\hspace{-7mm}  \leq \int_{\Omega} G(\rho^{0}(x) + \alpha) \dx  \nonumber.
\end{align}
We can now pass to the limit as $\alpha \to 0_+$ to obtain the following bound for $n=1,\ldots,N$, which is uniform in $\Delta t$ thanks to (\ref{IneqPropIn}):
\begin{align} \int_{\Omega} G(\rhoDtp(x,t_n)) \dx &+ \frac{1}{2 \Delta t \, \| \rho_{0}\|_{L^{\infty}(\Omega)}} \int_{0}^{t_n} \int_{\Omega} (\rhoDtp - \rhoDtm)^{2} \dx\,\dtau  \nonumber \\  & + 4\int_{0}^{t_n} \int_{\Omega} \bigg|\nabla \sqrt{\rhoDtp}\bigg|^{2} \dx\,\dtau  - \int_0^{t_n} \int_\Omega A(x) \nabla c^{\Delta t, +} \cdot \nabla \rhoDtp  \dx \, \,\dtau \label{RhoIneqEnergy1} \\ &\hspace{-7mm}\leq \int_{\Omega}G(\rho^{0}(x)) \leq \int_{\Omega} G(\rho_{0})  \dx . \nonumber \end{align}  
Finally we notice that
\begin{align}
    \int_{\Omega} A(x) \nabla c^{\Delta t, +}(x,t) \cdot \nabla \rhoDtp(x,t) \dx &=  \int_{\Omega} A(x) \nabla c^{\Delta t, +}(x,t) \cdot \nabla (\rhoDtp(x,t))^\diamond \dx \nonumber \\
    &= (\LN c^{\Delta t, +}, (\rhoDtp)^\diamond) - \int_\Omega Q(x) c^{\Delta t, +} (\rhoDtp)^\diamond \dx. \nonumber \\
    &= -(\LN^{1-s} (\rhoDtp)^{\diamond}, (\rhoDtp)^\diamond) - \int_\Omega Q(x) c^{\Delta t, +} (\rhoDtp)^\diamond \dx. \label{eq: interaction term expansion Dt}
\end{align}
Therefore, in case $Q=0$ a.e. in $\Omega$ the whole of the above expression is nonpositive and therefore the term $-\int_{\Omega} A(x) \nabla c^{\Delta t, +}(x,t) \cdot \nabla \rhoDtp(x,t) \dx$ in \eqref{RhoIneqEnergy1} can be discarded since it contributes constructively to the inequality. If $Q$ is not vanishing a.e. in $\Omega$ then the term $- \int_\Omega Q(x) c^{\Delta t, +} (\rhoDtp)^\diamond \dx$ is nonnegative and cannot be discarded. In this case, we firstly observe  that $(\rhoDtp)^\diamond = \rhoDtp$ and secondly that, by \eqref{LInftyPoiStab} in Lemma \ref{LInftyPoiStabFracOp} we have 
\begin{equation*}
 \| c^{\Delta t, +} \|_{L^\infty(\Omega)} \leq K_{\Omega, s,d} \| \rhoDtp \|_{L^\infty(\Omega)},
\end{equation*}
for a positive constant $K_{\Omega, s,d}$ depending on the domain $\Omega$, the fractional order $s$ and the dimension $d$.
Therefore, using \eqref{general InfNormBoundDiscTime},  
\begin{align}
    \int_\Omega Q(x) c^{\Delta t, +} (\rhoDtp)^\diamond \dx &\leq \| Q\|_{L^\infty(\Omega)}  K_{\Omega, s,d} \| \rho_0 \|_{L^\infty(\Omega)} \int_\Omega \rhoDtp \dx \nonumber \\
    &= \| Q\|_{L^\infty(\Omega)}  K_{\Omega, s,d} \| \rho_0 \|_{L^\infty(\Omega)} |\Omega|. \label{eq: uniform bound for Q term}
\end{align}
The constant $  \| Q\|_{L^\infty(\Omega)}  K_{\Omega, s,d} \| \rho_0 \|_{L^\infty(\Omega)} |\Omega|$ is independent of $\Delta t$ and it depends only on the data of the problem, thus it can be merged with the left-hand side of \eqref{RhoIneqEnergy1} into a unique constant, independent of $\Delta t$, that we denote by $C(\rho_0)$.
 
Therefore, combining \eqref{RhoIneqEnergy1}, \eqref{eq: interaction term expansion Dt} and \eqref{eq: uniform bound for Q term} we have, for $n=1,\ldots,N$, that
\begin{align}\label{ExtUniformBoundSpatialDer}
\begin{aligned}
 \int_{\Omega} G(\rhoDtp(x,t_n)) \dx & + \frac{1}{2 \Delta t \, \| \rho_{0}\|_{L^{\infty}(\Omega)} } \int_{0}^{t_n} \int_{\Omega} (\rhoDtp - \rhoDtm)^{2} \dx\,\dtau  + 4\int_{0}^{t_n} \int_{\Omega} \bigg|\nabla \sqrt{\rhoDtp}\bigg|^{2} \dx\,\dtau \\
 &  \leq  \int_{\Omega} G(\rho_{0}) \dx +  \| Q\|_{L^\infty(\Omega)}  K_{\Omega, s,d} \| \rho_0 \|_{L^\infty(\Omega)} |\Omega| \coloneq C(\rho_{0}). 
\end{aligned}
\end{align}
Notice that, in the case $Q=0$ a.e. in $\Omega$ we would have $C(\rho_0)$ simply equal to $\int_{\Omega} G(\rho_{0}) \dx$.

Using property (\ref{IneqPropIn}) we can supplement the above inequality with additional bounds. The first of these is arrived at by noticing that 
\begin{align*}
 4 \int_{0}^{T} \int_{\Omega}  \bigg| \nabla \sqrt{\rhoDtm} \bigg|^{2} \dx  \dt 
 &= 4 \Delta t \int_{\Omega} \bigg| \nabla \sqrt{\rho^{0}} \bigg|^{2} \dx + 4 \int_{\Delta t}^{T } \int_{\Omega} \bigg| \nabla \sqrt{\rhoDtm} \bigg|^{2} \dx \dt \\
 &= 4 \Delta t \int_{\Omega} \Big| \nabla \sqrt{\rho^{0}} \bigg|^{2} \dx + 4 \int_{0}^{T - \Delta t} \int_{\Omega} \Big| \nabla \sqrt{\rhoDtp} \bigg|^{2} \dx  \dt \\
 & \leq \int_{\Omega} G(\rho_{0}) \dx  + 4 \int_{0}^{T - \Delta t} \int_{\Omega} \Big| \nabla \sqrt{\rhoDtp} \bigg|^{2} \dx \dt \leq C_{\star},
\end{align*}
where in the two inequalities appearing in the last line we used  \eqref{IneqPropIn} and \eqref{ExtUniformBoundSpatialDer}, respectively; here and henceforth $C_{\star}$ signifies a generic positive constant, independent of $\Delta t$.

A simple calculation then shows that we can also get a similar inequality for $\rhoDt$, and therefore we have the uniform bound 
\begin{equation} \label{ExtUniformBoundSpatialDerPlus}
    4 \int_{0}^{T} \int_{\Omega} \bigg| \nabla \sqrt{\rhoDtpm} \bigg|^{2} \dx  \dt \leq C_{\star}.
\end{equation}

In this generalised setting, obtaining a uniform bound on the norm of $c^{\Delta t, +}$ is not achievable by choosing $-\chi_{[0,t_n]} c^{\Delta t, +}$ as a test function in \eqref{general WeakFormSemiDiscTime noL} because, the term  
$ \int_\Omega \nabla \rhoDtp \cdot \nabla c^{\Delta t, +} \dx$
 does not have a fixed sign, as it would have had it if $\LL=-\Delta$ (see \cite{Fronzoni2025}). Nevertheless, the restriction on the fractional order $s \in (1/2, 1)$ and similar observations to the ones in Remark \ref{uniform control gradient c with large s} can give us control on the gradient of $c^{\Delta t, +}$ in terms of the $L^2(0, T, L^2(\Omega))$ norm of $(\rhoDtp)^\diamond$, which will allow us to perform the passage to the limit in $\Delta t$ for the potential as well. In fact, we have 
\begin{align}
    \int_{\Omega} |\nabla c^{\Delta t, +}|^2 \dx &\leq C (\LL_{\mathrm{N}} c^{\Delta t, +}, c^{\Delta t, +}) = C (\LL_{\mathrm{N}}^{1-s} (\rhoDtp)^\diamond, \LL_{\mathrm{N}}^{-s} (\rhoDtp)^\diamond) \nonumber \\
    &= C (\LL_{\mathrm{N}}^{1-2s} (\rhoDtp)^\diamond,  (\rhoDtp)^\diamond) \nonumber \\
    &= C \sum_{k=1}^{\infty} \lambda_k^{1-2s}((\rhoDtp)^\diamond)_k^2 \leq C \lambda_1^{1-2s}\sum_{k=1}^{N_h} ((\rhoDtp)^\diamond)_k^2 \nonumber \\
    &= C \lambda_1^{1-2s} \| ( \rhoDtp)^{\diamond} \|_{L^2(\Omega)}^2. \label{ext general control on grad cDt}
\end{align}

We shall assume throughout the rest of the section that the test functions belong to the function space $L^{2}(0, T, W^{1, \infty}(\Omega))$. Our objective is to derive a (uniform in $\Delta t$) bound on the time-derivative of $\rhoDt$ in terms of the $L^{2}(0, T, W^{1, \infty}(\Omega))$ norm of the test function, followed by the use of Sobolev embedding, which will enable us to apply a variant of the Aubin--Lions compactness lemma, known as Dubinski\u{\i}'s compactness theorem.  From \eqref{general WeakFormSemiDiscTime noL} we have that, for a test function $\phi \in L^{2}(0, T, W^{1, \infty}(\Omega))$,
\begin{align*}
    \Bigg| \int_{0}^{T} \int_{\Omega} \frac{\partial \rhoDt}{\partial t} \phi \dx\dt \Bigg| \leq \underbrace{\Bigg| \int_{0}^{T} \int_{\Omega} \nabla \rhoDtp \cdot \nabla \phi \dx\dt \Bigg|}_{\circled{1}} + \underbrace{\Bigg| \int_{0}^{T} \int_{\Omega} \rhoDtp A(x)\nabla c^{\Delta t, +} \cdot \nabla \phi \dx\dt \Bigg|}_{\circled{2}}.
\end{align*}
We recall that we have 
\begin{subequations}\label{SmallProp1}
    \begin{equation}\label{Positivity}
    \rhoDtpm \geq 0 \quad \textrm{a.e.~on } \Omega \times (0, T), \quad \mbox{and}
    \end{equation}
    \begin{equation} \label{UnitDensityboundDt}
        \frac{1}{|\Omega|} \int_{\Omega} \rhoDtpm(t) = 1 \quad \textrm{for a.e. } t \in (0, T).
    \end{equation}
\end{subequations}
Moreover we notice that the bound on the first term on the left-hand side of \eqref{ExtUniformBoundSpatialDer} gives us a (uniform in $\Delta t$) bound on the $L^\infty(0,T;L^{1}(\Omega))$ norm of $\rhoDtp$, because the function $s \in [0,\infty) \mapsto G(s) \in [0,\infty)$ has superlinear growth as $s \rightarrow +\infty$. 

For the term $\circled{1}$ we then have by \eqref{ExtUniformBoundSpatialDer} and (\ref{SmallProp1}) that
\begin{align*}
    \circled{1} &\leq \int_{0}^{T} \|\nabla \phi\|_{L^{\infty}(\Omega)} \int_{\Omega} |\nabla \rhoDtp| \dx \dt= 2 \int_{0}^{T} \|\nabla \phi\|_{L^{\infty}(\Omega)} \int_{\Omega} \sqrt{\rhoDtp}\, \bigg| \nabla \sqrt{\rhoDtp} \bigg| \dx \dt\\
    &\leq 2 \int_{0}^{T} \|\nabla \phi\|_{L^{\infty}(\Omega)} \Bigg( \int_{\Omega} \rhoDtp \dx \Bigg)^{\frac{1}{2}} \Bigg( \int_{\Omega}\bigg| \nabla \sqrt{\rhoDtp} \bigg|^{2} \dx \Bigg)^{\frac{1}{2}} \dt  \leq C_{\ast} \Bigg( \int_{0}^{T} \|\nabla \phi\|_{L^{\infty}(\Omega)}^{2} \dt \Bigg)^{\frac{1}{2}}.
\end{align*}
For the term $\circled{2}$  we use the Cauchy-Schwarz inequality combined with \eqref{ext general control on grad cDt}, the uniform ellipticity of $A(x)$ and \eqref{general InfNormBoundDiscTime} and, with a similar computation as for term $\circled{1}$ above, we get 
\begin{equation*}
    \circled{2} \leq C_{\ast} \Bigg( \int_{0}^{T} \|\nabla \phi\|_{L^{\infty}(\Omega)}^{2} \dt \Bigg)^{\frac{1}{2}}.
\end{equation*}
By combining the bounds on the terms $\circled{1}$ and $\circled{2}$ and noting that by  the Sobolev embedding theorem $H^{d+1}(\Omega)$ is continuously embedded in $W^{1,\infty}(\Omega)$ for $d=2,3$, we then have that 
\begin{equation} \label{ExtUniformBoundTimeDer}
\Bigg| \int_{0}^{T} \int_{\Omega} \frac{\partial \rhoDt}{\partial t} \phi \,\dx  \dt\Bigg| \leq C_{\ast} \Bigg( \int_{0}^{T} \|\nabla \phi\|_{L^{\infty}(\Omega)}^{2} \dt \Bigg)^{\frac{1}{2}}\leq C_{\ast} \Bigg( \int_{0}^{T} \|\phi\|_{H^{d+1}(\Omega)}^{2} \dt \Bigg)^{\frac{1}{2}}.
\end{equation}

\subsection{Passage to the limit $\Delta t \rightarrow 0_{+}$}

We are now ready to take the limit as $\Delta t \to 0_+$.  
We first collect the $\Delta t$-independent bounds \eqref{ExtUniformBoundSpatialDer}, \eqref{ExtUniformBoundSpatialDerPlus}, \eqref{ExtUniformBoundTimeDer} and the bound \eqref{ext general control on grad cDt};  we have shown that there exists a constant $C_{\star}$, independent on $\Delta t$, such that, for $\beta=d+1$, 
\begin{align}\label{ExtUniformBoundDer}
\begin{aligned}
    \sup_{t \in (0,T]} \int_{\Omega} G(\rhoDtp(t)) \dx & + \int_{0}^{T} \int_{\Omega} (\rhoDtp - \rhoDtm)^{2} \dx\dt \\
    &+ \int_{0}^{T} \int_{\Omega} \bigg|\nabla \sqrt{\rhoDtpm}\bigg|^{2} \dx\dt  + \int_{0}^{T}\bigg\| \frac{\partial \rhoDt}{\partial t} \bigg\|^{2}_{H^{-\beta}(\Omega)} \dt \leq C_{\star},
\end{aligned}
\end{align}
\begin{align}\label{ExtBoundDer c}
\begin{aligned}    
     \int_{\Omega} |\nabla c_{L}^{\Delta t, +}|^2 \leq C_\ast \| ( \rhoLDtp)^{\diamond} \|_{L^2(\Omega)}^2
\end{aligned}
\end{align}

For passing to the limit in $\Delta t \to 0_+$ the crucial result that we will use is Dubinski\u{\i}'s compactness theorem in seminormed sets; cf.  \cite{dubinskii1965weak} and \cite{barrett2012dubinskii}. 
The statement of the theorem is contained in the Appendix \ref{dubinski appendix}, Theorem \ref{Dubinsky}.

\begin{theo} \label{ExtTheoConvLDt}
 For any initial datum $\rho_0 \in  L^{\infty}(\Omega)$, there exists a subsequence of $\{ \rhoDtpm \}_{\Delta t>0}$ (not indicated) and a function $\widehat{\rho}$ such that 
\begin{equation*}
    \widehat{\rho} \in L^{\infty}(0, T, L^{\infty}(\Omega)) \cap H^{1}(0, T, H^{-\beta}(\Omega)), \quad \beta = d+1, 
\end{equation*}
with $\widehat{\rho} \geq 0$ almost everywhere on $\Omega\times (0,T)$ and $\frac{1}{|\Omega|} \int_{\Omega}\widehat{\rho}(x.t) \dx = 1$ for a.e.~$t \in [0, T]$, and $\widehat{c}(\cdot,t) \in \mathcal{H}^{s}_\diamond(\Omega) \cap H^1(\Omega)$ defined as $-\LN^{s} \widehat{c}(\cdot,t) = \widehat{\rho}^{\diamond}(\cdot,t)$ in $\Omega$  for $t \in (0,T]$ such that, for all $p \in [1, \infty)$, 
\begin{subequations}
    \begin{alignat}{2} \label{ExtConvL1}
\rho^{\Delta t(,\pm)} & \to \widehat{\rho} &&\qquad\text{strongly in } L^{p}(0,T;L^{1}(\Omega)), \\ \label{ExtConvL2}
\nabla \sqrt{\rho^{\Delta t(,\pm)}} & \to \nabla \sqrt{\widehat{\rho}} &&\qquad\text{weakly in } L^{2}(0,T;L^{2}(\Omega;\mathbb{R}^d)), \\
\label{ExtConvL3}
\frac{\partial \rhoDt}{\partial t} & \to \frac{\partial \widehat{\rho}}{\partial t} &&\qquad \text{weakly in } L^{2}(0,T; H^{-\beta}(\Omega)), \\
\label{ExtConvL3a}
\rho^{\Delta t(,\pm)} & \to \widehat{\rho} &&\qquad \text{weakly-$\ast$ in } L^{\infty}(0,T;L^{\infty}(\Omega)),\\
\label{ExtConvL5}
\rho^{\Delta t(,\pm)} & \to \widehat{\rho} &&\qquad \text{strongly in } L^{p}(0,T;L^{p}(\Omega)), \\
\label{ExtConvL6}
c^{\Delta t(,\pm)} & \to \widehat{c} &&\qquad \text{strongly in }  L^{p
}(0, T; \mathcal{H}^{s}(\Omega)), \\
\label{ExtConvL7}
\nabla c^{\Delta t, +} & \to \nabla \widehat{c} &&\qquad  \text{weakly in }  L^{2}(0, T; L^{2}(\Omega;\mathbb{R}^d)).
    \end{alignat}
The function $\widehat{\rho}$ is a global weak solution to the problem 
\begin{gather} \label{ExtWeakSolutionEq}
    -\int_{0}^{T} \int_{\Omega} \widehat{\rho} \frac{\partial \phi}{\partial t}  \dx \dt+ \int_{0}^{T} \int_{\Omega}  \nabla \widehat{\rho} \cdot \nabla \phi  \dx \dt+ \int_{0}^{T} \int_{\Omega} \widehat{\rho} \, A(x) \nabla \widehat{c} \cdot \nabla \phi  \dx \dt = \int_{\Omega} \rho_{0} \;\phi|_{t=0} \dx \\
    \nonumber \text{for all } \phi \in W^{1,1}(0, T; H^{\beta}(\Omega)) \text{ such that } \phi(\cdot, T) = 0 . \end{gather}
In addition, the function $\widehat{\rho}$ is weak-$\ast$ continuous as a mapping from $[0, T]$ to $L^{\infty}(\Omega)$
and it is weakly continuous as a mapping from $[0,T]$ to $L^1(\Omega)$. 
    
\end{subequations}
\end{theo}
 \begin{proof}
     In the notation of Theorem \ref{Dubinsky}, we choose $\mathcal{A}_{0}, \mathcal{A}_{1}$ and $\mathcal{C}$ as in the discussion following the statement of  Dubinski\u{\i}'s theorem (Theorem \ref{Dubinsky}). Then, by taking $p=1$ and $p_{1}=2$ we have that 
    \begin{equation*}
        \Bigg \{ \varphi: [0, T] \to \mathcal{C}: [\varphi]_{L^{p}(0, T; \mathcal{C})} + \bigg\| \frac{\mathrm{d}\varphi}{\dt}\bigg\|_{L^{p_{1}}(0, T; \mathcal{A}_{1})} < \infty \Bigg \}
    \end{equation*}
    is compactly embedded in $L^{1}(0, T; \mathcal{A}_{0}) = L^{1}(0, T; L^{1}(\Omega))$. By using the bounds on the last two terms on the left-hand side of \eqref{ExtUniformBoundDer}, we have that there exists a subsequence of $\{\rhoDt \}_{\Delta t >0}$ (not indicated), which converges strongly in $L^{1}(0, T; L^{1}(\Omega))$ to an element $\widehat{\rho} \in L^{1}(0, T; L^{1}(\Omega))$ as $\Delta t \to 0$. The strong convergence in $L^{1}(0, T; L^{1}(\Omega))$ ensures almost everywhere convergence of a subsequence (not indicated) to $\widehat{\rho}$.
    We have  by  \eqref{LinInterptime} and the bound on the second term on the left-hand side of \eqref{ExtUniformBoundDer} that 
    \begin{equation} \label{IneqDeltaT}
        \Bigg( \int_{0}^{T} \int_{\Omega} | \rhoDt - \rhoDtp| \dx\dt \Bigg)^{2} \leq 
        T|\Omega| \int_{0}^{T} \int_{\Omega} (\rhoDtp - \rhoDtm)^{2} \dx\dt \leq C_{\star} T |\Omega|.
    \end{equation}
    We apply the triangle inequality in $L^{1}(0, T; L^{1}(\Omega))$, together with the strong convergence of the sequence $\{ \rhoDt \}_{\Delta t>0}$ in $L^{1}(0,T; L^{1}(\Omega))$ to deduce strong convergence in $L^{1}(0,T; L^{1}(\Omega))$ to the same element $\widehat{\rho} \in L^{1}(0,T; L^{1}(\Omega))$. ~The inequality  (\ref{IneqDeltaT}) also implies strong convergence of $\{ \rhoDtm \}_{\Delta t>0}$ in $L^{1}(0,T; L^{1}(\Omega))$ to $\widehat{\rho}$. This completes the proof of \eqref{ExtConvL1} for $p=1$. 
    
    From (\ref{SmallProp1}) we have that 
    \begin{equation*}
        \|\rho^{\Delta t(,\pm)}(t)\|_{L^{1}(\Omega)} \leq |\Omega|
    \end{equation*}
    for a.e.~$t \in (0,T]$. This means that the sequences $\{ \rho^{\Delta t (,\pm)} \}_{\Delta t>0}$ are bounded in $L^{\infty}(0,T;L^{1}(\Omega))$. We can use \cite[Lemma 4.4]{Fronzoni2025} together with the strong convergence of the sequences to $\widehat{\rho}$ in $L^{1}(0,T;L^{1}(\Omega))$ to deduce that we have strong convergence in $L^{p}(0,T;L^{1}(\Omega))$ to the same limit for all values of $p \in [1, \infty)$. We have then completed the proof of \eqref{ExtConvL1}. 
    
    Because strong convergence in $L^{p}(0,T;L^{1}(\Omega))$ for $p \in [1, \infty)$ implies convergence almost everywhere on $\Omega \times (0,T])$ of a subsequence (not indicated), it follows from (\ref{Positivity}) that $\widehat{\rho} \geq 0$ a.e.~on $\Omega \times (0,T)$. By applying Fubini's theorem we have that 
    \begin{equation*}
        \int_{\Omega} |\rho^{\Delta t(,\pm)} - \widehat{\rho} | \dx \to 0 \quad \text{as } \Delta t \to 0_+ \text{ for 
 a.e. } t \in (0, T].
    \end{equation*}
Hence we have by (\ref{UnitDensityboundDt}) that
    \begin{equation*} 
       \frac{1}{|\Omega|} \int_{\Omega} \widehat{\rho}(t) \dx = 1 \quad \text{for a.e. } t \in (0, T].
    \end{equation*}
    
    We shall now prove \eqref{ExtConvL2}. We notice that since $|\sqrt{c_{1}} - \sqrt{c_{2}}| \leq \sqrt{|c_{1}-c_{2}|}$ for any two nonnegative real numbers $c_{1}$ and $c_{2}$, the strong convergence \eqref{ExtConvL1} directly implies that, as $\Delta t \to 0_+$,
    \begin{equation*}
        \sqrt{\rho^{\Delta t(,\pm)}} \to \sqrt{\widehat{\rho}} \qquad  \text{strongly in } L^{p}(0, T; L^{2}(\Omega)) \quad \text{for all } p \in [1, \infty).
    \end{equation*}
The bound on the third term on the left-hand side of \eqref{ExtUniformBoundDer} implies the existence of a subsequence (not indicated) and an element $G \in L^{2}(0, T; L^{2}(\Omega;\mathbb{R}^d))$ such that 
    \begin{equation*}
        \nabla \sqrt{\rho^{\Delta t(,\pm)}} \to G \qquad \text{weakly in } L^{2}(0, T; L^{2}(\Omega;\mathbb{R}^d)).
    \end{equation*}
We have therefore that, for a test function $\eta \in C([0,T]; C^{\infty}_{0}(\Omega;\mathbb{R}^d))$,
    \begin{equation*}
       \int_{0}^{T} \int_{\Omega} G \cdot \eta \,\dx \dt \leftarrow -\int_{0}^{T} \int_{\Omega} \sqrt{\rho^{\Delta t(,\pm)}} \, \text{div} \, \eta \dx \dt \rightarrow - \int_{0}^{T} \int_{\Omega} \sqrt{\widehat{\rho}} \,\, \text{div} \, \eta \dx \dt.
    \end{equation*}
Therefore we have the equality 
    \begin{equation*}
        \int_{0}^{T}\int_{\Omega} G \cdot \eta\, \dx \dt =  - \int_{0}^{T}\int_{\Omega} \sqrt{\widehat{\rho}} \,\, \text{div} \, \eta \dx \dt,
    \end{equation*}
    and this means that $G$ is the distributional gradient of $\sqrt{\widehat{\rho}}$. As $G \in L^{2}(0, T; L^{2}(\Omega;\mathbb{R}^d))$ it follows that 
    \begin{equation*}
        G = \nabla \sqrt{\widehat{\rho}} \in L^{2}(0, T; L^{2}(\Omega;\mathbb{R}^d))
    \end{equation*}
    and hence \eqref{ExtConvL2} has been proved. 
    
    The weak convergence result \eqref{ExtConvL3} follows from the uniform bound on the last term on the left-hand side of \eqref{ExtUniformBoundDer} and the weak compactness of bounded balls in the Hilbert space $L^{2}(0,T;H^{-\beta}(\Omega))$. 

    The weak-$\ast$ convergence result \eqref{ExtConvL3a} follows from the uniform bound \eqref{general InfNormBoundDiscTime} stated in Lemma \ref{general InfNormDecayDiscTime}, the weak-$\ast$ compactness of bounded balls in the Banach space $L^\infty(0,T;L^\infty(\Omega))$, and the uniqueness of the weak limit. 

    The strong convergence result \eqref{ExtConvL5} is a direct consequence of \eqref{ExtConvL1} and the uniform bound \eqref{general InfNormBoundDiscTime} stated in Lemma \ref{general InfNormDecayDiscTime}, using \cite[Lemma 4.4]{Fronzoni2025}.

    Then, \eqref{ExtConvL6} follows directly from \eqref{ExtConvL5} \eqref{general StabFracNeu} and \eqref{FracPoincareUse}; that is,
    \begin{align*}
\int_0^T \| c^{\Delta t (,\pm)}(t) - \widehat{c}(t) \|_{{\mathcal{H}}^s_\diamond(\Omega)}^p \dt \leq C \int_0^T \| \rhoDtpm(t) - \widehat{\rho}(t) \|_{L^{2}_\diamond(\Omega)}^p \dt
    \end{align*}
    for all $p \in [1,\infty)$.
    
    Let us now consider the $L^{2}(0, T; L^{2}(\Omega;\mathbb{R}^d))$ norm of $\nabla c^{\Delta t, +}$. By the inequality \eqref{ExtBoundDer c}
    we have 
    \begin{align*} 
    \int_{0}^{T} \int_{\Omega} \big|\nabla c^{\Delta t, +} \big|^{2} \dx\dt &\leq C \int_{0}^{T} \int_{\Omega} ((\rhoDtp)^\diamond)^2 \dx \dt
    \end{align*}
    and thanks to the uniform bound on the $L^{2}(0, T; L^{2}(\Omega))$ norm of $\rhoDtp$ from the convergence \eqref{ExtConvL5} we have that the $L^{2}(0, T; L^{2}(\Omega;\mathbb{R}^d))$ norm of the gradient of $c^{\Delta t, +}$ is uniformly bounded and therefore $\nabla c^{\Delta t, +}$ converges weakly to an element that we call $G$, as $\Delta t \rightarrow 0_+$. By \eqref{ExtConvL6} with $s=0$ we have that $\nabla c^{\Delta t, \pm}$ converges strongly to $\nabla \widehat{c}$ in $L^{2}(0, T, H^{-1}(\Omega;\mathbb{R}^d))$, but then $\nabla \widehat{c}$ must coincide with $G$ 
    because of the uniqueness of the weak limit. 
    
    We now want to use the convergence results we have just proved to pass to the limit as $\Delta t \to 0_{+}$ in equation \eqref{general WeakFormDisctime}. Throughout the argument we will consider test functions $\phi \in C^{1}([0, T]; C^{\infty}(\overline{\Omega}))$ such that $\phi(\cdot, T) = 0$. Note that the set of all such test functions is dense in the set of functions belonging to $W^{1,1}(0, T; H^{\beta}(\Omega))$ and vanishing at $t=T$ (in the sense of the trace theorem in $W^{1,1}(0,T)$), which is continuously embedded in $L^{2}(0, T; H^{\beta}(\Omega))$; therefore the use of such test functions is fully justified for the purposes of our argument. 
    
    We begin by considering the first term in \eqref{general WeakFormSemiDiscTime noL} and use integration by parts with respect to $t$ to deduce that
    \begin{equation*}
        \int_{0}^{T} \int_{\Omega} \frac{\partial \rhoDt}{\partial t} \phi \dx\dt = -\int_{0}^{T} \int_{\Omega} \rhoDt \frac{\partial \phi}{\partial t} \dx\dt - \int_{\Omega} \rho^{0} \; \phi|_{t=0} \dx
    \end{equation*}
    for all $\phi \in C^{1}([0,T], C^{\infty}(\overline{\Omega}))$ such that $\phi(\cdot , T)=0$. Moreover, 
\begin{equation*}
        \lim_{\Delta t \to 0_{+}} \rho^{0} = \rho_{0} \quad \text{weakly in } L^{2}(\Omega).
    \end{equation*}
    Therefore using \eqref{ExtConvL1} we immediately have that, as $\Delta t \to 0_{+}$,
    \begin{equation*}
        \int_{0}^{T} \int_{\Omega} \frac{\partial \rhoLDt}{\partial t} \phi \dx\dt \to -\int_{0}^{T} \int_{\Omega} \widehat{\rho} \frac{\partial \phi}{\partial t} \dx\dt - \int_{\Omega} \rho_{0} \;\phi|_{t=0} \dx.
    \end{equation*}
    The second term in \eqref{general WeakFormSemiDiscTime noL} will be dealt with by decomposing it as follows:
\begin{align*} \int_{0}^{T} \int_{\Omega}  \nabla \rhoDtp \cdot \nabla \phi \dx\dt &= \underbrace{2\int_{0}^{T} \int_{\Omega} \bigg( \sqrt{\rhoDtp} - \sqrt{\widehat{\rho}} \bigg) \nabla \sqrt{\rhoDtp} \cdot \nabla \phi \dx\dt}_{\circled{1}} \\
    & \quad + \underbrace{2\int_{0}^{T} \int_{\Omega} \sqrt{\widehat{\rho}} \,\nabla \sqrt{\rhoDtp} \cdot \nabla \phi \dx\dt}_{\circled{2}}. 
    \end{align*}
For term $\circled{1}$ we have by H\"older's inequality that 
    \begin{align*}
        |\circled{1}| & \leq 2\int_{0}^{T} \bigg( \int_{\Omega}  \Big| \sqrt{\rhoDtp} - \sqrt{\widehat{\rho}} \Big|^{2} \dx \bigg)^{\frac{1}{2}} \bigg( \int_{\Omega} \Big| \nabla \sqrt{\rhoDtp} \Big|^{2} \dx \bigg)^{\frac{1}{2}} \|\nabla \phi\|_{L^{\infty}(\Omega)} \dt \\
        &\leq 2\,\Bigg( \int_{0}^{T} \int_{\Omega} \Big| \nabla \sqrt{\rhoDtp} \Big|^{2} \dx\dt  \Bigg)^{\frac{1}{2}} \Bigg( \int_{0}^{T} \Big\| \sqrt{\rhoDtp} - \sqrt{\widehat{\rho}} \Big\|_{L^{2}(\Omega)}^{r}  \dt\Bigg)^{\frac{1}{r}} \\
        & \quad \times \Bigg( \int_{0}^{T} \|\nabla \phi\|_{L^{\infty}(\Omega)}^{\frac{2r}{r-2}} \dt \Bigg)^{\frac{r-2}{2r}},
    \end{align*}
    where $r \in (2, \infty)$. Using the bound on the third term in the inequality stated in \eqref{ExtUniformBoundDer} we have 
    \begin{align*} |\circled{1}| &\leq 2C_{\star}^{\frac{1}{2}} \Big\| \sqrt{\rhoDtp} - \sqrt{\widehat{\rho}} \Big\|_{L^{r}(0, T; L^{2}(\Omega))} \|\nabla \phi\|_{L^{\frac{2r}{r-2}}(0, T; L^{\infty}(\Omega))} \\
    & \leq 2C_{\star}^{\frac{1}{2}} \|\rhoDtp - \widehat{\rho}\|_{L^{\frac{r}{2}}(0, T; L^{1}(\Omega))}^{\frac{1}{2}} \|\nabla \phi\|_{L^{\frac{2r}{r-2}}(0, T; L^{\infty}(\Omega))},   \end{align*}
    where we have used the elementary inequality $|\sqrt{c_{1}} - \sqrt{c_{2}}| \leq \sqrt{|c_{1} - c_{2}|}$ with $c_{1}, c_{2} \geq 0$. The first factor in the last line converges to 0 as $\Delta t \to 0_+$ thanks to  \eqref{ExtConvL1}, and therefore the term $\circled{1}$ converges to 0
    for every $\phi \in C^1([0,T];C^\infty(\overline\Omega))$ as $\Delta t \to 0_+$.
    
    Concerning the term $\circled{2}$, as $\widehat\rho \in L^\infty(0,T;L^\infty(\Omega))$ and $\widehat\rho \geq 0$ a.e.~on $\Omega \times (0,T)$, we have that  $\sqrt{\widehat{\rho}}\, \nabla \phi$ belongs to $L^{2}(0, T; L^{2}(\Omega))$, and the weak convergence result \eqref{ExtConvL2} then directly implies that
    \begin{align*} \lim_{\Delta t \rightarrow 0_+} 2\int_{0}^{T} \int_{\Omega} \sqrt{\widehat{\rho}}\, \nabla \sqrt{\rhoDtp} \cdot \nabla \phi \,\dx\dt 
    &= 2\int_{0}^{T} \int_{\Omega} \sqrt{\widehat{\rho}}\, \nabla \sqrt{\widehat{\rho}} \cdot \nabla \phi \,\dx\dt  \\
    &= \int_{0}^{T} \int_{\Omega} \nabla \widehat{\rho} \cdot \nabla \phi \,\dx\dt. \end{align*}
Let us consider the final term in \eqref{general WeakFormSemiDiscTime noL}.
    By \eqref{ExtConvL7}, $\nabla c^{\Delta t, +} \rightharpoonup \nabla\widehat{c}$ weakly in $L^2(0,T;L^2(\Omega;\mathbb{R}^d))$ and therefore $A(x) \nabla c^{\Delta t, +} \rightharpoonup A(x) \nabla\widehat{c}$ weakly in $L^2(0,T;L^2(\Omega;\mathbb{R}^d))$, as $\Delta t \to 0_+$, i.e.
    \begin{align*}
        \int_{0}^{T} \int_{\Omega} A(x) \nabla c^{\Delta t, +} \cdot \nabla \phi \dx \dt  = \int_{0}^{T} \int_{\Omega}  \nabla c^{\Delta t, +} \cdot A(x) \nabla \phi \dx \dt &\to \int_{0}^{T} \int_{\Omega}  \nabla \widehat{c} \cdot A(x) \nabla \phi \dx \dt  \\
        &\quad = \int_{0}^{T} \int_{\Omega} A(x)  \nabla \widehat{c} \cdot  \nabla \phi \dx \dt,
    \end{align*}
    by the fact that if $\nabla \phi \in L^2(0, T, L^2(\Omega))$ then $A(x) \nabla \phi \in L^2(0, T, L^2(\Omega))$, thanks to $A$ being bounded in $L^\infty(\Omega)$.
    Moreover, by \eqref{ExtConvL5}, $\rho^{\Delta t,+} \rightarrow \widehat\rho$ strongly in $L^2(0,T;L^2(\Omega))$ and thus it follows that $\rho^{\Delta t,+} A(x) \nabla c^{\Delta t, +}\rightharpoonup \widehat\rho\, A(x) \nabla\widehat{c}$ weakly in $L^1(0,T;L^1(\Omega;\mathbb{R}^d))$. Therefore
    \begin{align*}
        \lim_{\Delta t \to 0_+} \int_{0}^{T} \int_{\Omega}\rho^{\Delta t, +} A(x) \nabla c^{\Delta t, +} \cdot \nabla \phi \dx\dt = \int_{0}^{T} \int_{\Omega} \widehat{\rho} \, A(x) \nabla \widehat{c} \cdot \nabla \phi \dx\dt
    \end{align*}
    for all $\phi \in C^1([0,T];C^\infty(\overline\Omega))$.

Putting these convergence results together, and noting that $C^1([0,T];C^\infty(\overline\Omega))$ is dense in the function space $W^{1,1}(0.T;H^\beta(\Omega))$ with $\beta=d+1$, we get \eqref{ExtWeakSolutionEq}.

In the end, we prove the weak-$\ast$ continuity of the function $\widehat{\rho}$ as a map from $[0, T]$ to $L^{\infty}(\Omega)$. 
As $\widehat\rho \in L^\infty(0,T;L^\infty(\Omega))$ and $\widehat\rho \in H^1(0,T; {H}^{-\beta}(\Omega)) \subset C([0,T];{H}^{-\beta}(\Omega))$, $\beta=d+1$, and $L^\infty(\Omega) = L^1(\Omega)'$, it follows from \cite[part (b) of Lemma 1.20]{Fronzoni2025} with $X= L^\infty(\Omega)$, $Y = {H}^{-\beta}(\Omega)$ and $E=L^1(\Omega)$ that $\widehat \rho \in C_{w^\ast}([0,T]; L^\infty(\Omega))$. Because $\Omega$ is bounded, $L^\infty(\Omega) \subset L^1(\Omega)$, and it then directly follows from this weak-$\ast$ continuity property of $\widehat\rho$ that $\widehat \rho \in C_w([0,T]; L^1(\Omega))$, i.e.,  $\widehat{\rho}$ is weakly continuous as a mapping from $[0,T]$ to $L^1(\Omega)$.
\end{proof}

\section{Numerical Experiments} \label{sec:5}

We report in this section some numerical experiments of our scheme for \eqref{ExtendedProb intro} and we comment on them. From a computational perspective these experiments have been obtained using a similar strategy to the one that \cite{Fronzoni2025} adopted for the porous medium equation with a fractional pressure (the case $\LL = -\Delta$). 
More precisely, the algorithm that was described in \cite{Fronzoni2024} and used in \cite{Fronzoni2025} can be easily adapted to this case. The method presented in \cite{Fronzoni2024} to rapidly compute the finite element approximation of the fractional Laplacian is based on the use of rational approximations; this technique is the result of a series contribution \cite{Hale2008, Burrage2012, Harizanov2018, filip2018rational, Zhang2025} that developed algorithms for the computation of rational approximations of functions and their application to problems in numerical linear algebra. Roughly speaking, computing the finite element approximation of the spectral fractional Laplacian reduces to approximate the fractional power $M^{s}$ of a matrix $M$, for $0<s<1$, where $M$ arises from a finite element approximation of the Laplacian. The use of rational approximations for computing the fractional power of a matrix is particularly efficient for time-dependent problems and fine meshes (see \cite{Fronzoni2024} for a detailed discussion of the benefits of this approach). 

In our case, the use of rational approximations can still be applied for computing an approximated solution of the problem 
\begin{equation} \label{eq: general fractional Poisson equation}
    \LN^s(u) = f \qquad \textrm{in } \Omega,
\end{equation}
for right-hand side datum $f \in L^2(\Omega)$.
Since the operator $\mathcal{L}_\mathrm{N}$ is positive and self-adjoint, a finite element discretization of \eqref{eq: general fractional Poisson equation} will have the form 
\begin{equation}  \label{eq: fe linear system}
    X^{-1} \Lambda^s X U = F,
\end{equation}
where the matrix $\Lambda$ is diagonal and its diagonal elements are the discrete eigenvalues $\{ \lambda_k^h \}_{k=1}^{N_h}$ of the finite-dimensional version of $\mathcal{L}_\mathrm{N}$, given in \eqref{Discrete FracPowerOp} (the computations are a straightforward analogue of \cite[Section 2.2]{Fronzoni2024}). 
Once one has an approximation of the endpoints of the interval $[\lambda_{min}^h, \lambda_{max}^h]$, corresponding to the minimum and maximum eigenvalues in $\{ \lambda_k^h \}_{k=1}^{N_h}$, a rational approximation of degree $n$ for $x^{-s}$ can be easily computed via the minimax algorithm \cite{filip2018rational}, as the authors of \cite{Fronzoni2024} did for the fractional Laplacian. If we denote the rational approximation of $x^{-s}$  by $r(x)$ and by $M$ the matrix $X^{-1} \Lambda X$, obtaining an approximation of the solution to \eqref{eq: fe linear system}, will result in computing $r(M) F$ (for comments on the efficiency of this procedure we refer the reader to \cite[Remark 2.2]{Fronzoni2024}).

Pairing the rational approximation algorithm with the finite element discretization of the equation leverages its full capabilities and allows to rapidly explore numerically different configurations. 

The most relevant aspect of the family of generalised models \eqref{ExtendedProb intro} lies in the features introduced by employing a general elliptic operator of the form $\mathcal{L} u = -\textrm{div}(A(x)\nabla u) + Q(x) u$. This choice enriches the nonlocal diffusion process that these equations describe. By varying $A(x)$ and $Q(x)$, it is possible to model different spatially heterogeneous configurations, which interact with the nonlocal effects arising from the fractional power of the operator. The advantage of having a fast algorithm in this case allows us to perform simulations with fine meshes, using different choices of $A(x)$ and $Q(x)$. 

In order to better observe the effects of employing an operator $\LN$ for the potential term $c$, we can control the parabolic regularisation term in the model through a parameter $\mu>0$. Thus, we compute simulations for the equation
\begin{equation*}
\frac{\partial \rho}{\partial t} = \mu \Delta \rho - \nabla \cdot (\rho A(x) \nabla c), \quad  - \LN^{\s} c = \rho^\diamond, \quad \frac{1}{2} < s <1.
\end{equation*}

In Figure \ref{fig: experiment 1 cross sections} and Figure \ref{fig: experiment 1 heat maps} we report the results of two numerical experiments starting from two different initial data $\rho_1, \rho_2$, with the following settings: 
\begin{enumerate}[label=I.\emph{\alph*}., ref=I.\emph{\alph*}.]
   \item  $\Omega=(-2,2)^2, \quad N_h = 2^8 \times 2^8, \quad \Delta t = 0.05$, \quad s =0.75,
    \item $A_{11}(x_1,x_2) = 10.0, \quad A_{22}(x_1,x_2)=0.1, \quad A_{12}(x_1,x_2)=A_{21}(x_1,x_2)=0$, \label{A10|01}
    \item $Q(x_1,x_2) = 100 |(x_1,x_2)|^2$,   \label{Q=|x|^2}  \item $\rho_1 = \exp(-(|x_1+1|^2 + x_2^2)/\sigma),\quad$ $\rho_2 = \exp(-(|x_1+1|^2 + x_2^2)/\sigma) + 3 \exp((|x_1-1|^2 + x_2^2)/\sigma), \quad \sigma=0.1$.  \label{eq: initial data comparison principle} 
\end{enumerate}

We can see that the spatial anisotropy introduced by the matrix $A(x)$ and the confinement introduced by the potential $Q(x)$ interact with the nonlocal interactions, generated by $c = \LN^{-s} \rho$. The spreading is stronger along the $x$ direction and slowed down away from the center of the domain. In this example, it is worth noting an important feature of this class of models: nonlocal diffusion leads to the failure of the comparison principle. In fact, in \cite[Section 6]{caffarelli2010nonlinear} the authors proved that the comparison principle does not hold for the porous medium equation with a fractional potential ($\LL = -\Delta)$ (this is shown numerically in two dimensions in \cite{Fronzoni2024}). In our case, we can observe the failure of the comparison principle in Figure \ref{fig: experiment 1 cross sections}: two initial data, ordered at time $t=0$ $\rho_1(x_1,0) \leq \rho_2(x_1,0)$, are no longer ordered as time advances. If we look at the potential $c_2 = \LN^{-s} \rho_2$ in Figure \ref{fig: experiment 1 heat maps} we can comment on the effect of using a fractional power of the operator $\LN$: $c_2$ is nonvanishing in the region between the two initial concentrations of $\rho_2$, and hence the density itself is dragged along the gradient of $c_2$, that is nonvanishing in between the concentrations; this intuitively explains why the comparison principle fails in this setting. 

In Figure \ref{fig: experiment 2 heat maps} we report the results of a numerical experiment starting from an initial datum $\rho_0$ with the following configuration:

\begin{enumerate}[label=II.\emph{\alph*}., ref=II.\emph{\alph*}.]
    \item $\Omega=(-2,2)^2, \quad N_h = 2^8 \times 2^8, \quad \Delta t = 0.05$, \quad s =0.67, 
    \item $A_{11}(x_1,x_2) = 0.1, \quad A_{22}(x_1,x_2)=10.0, \quad A_{12}(x_1,x_2)=A_{21}(x_1,x_2)=0$,  \label{A01|10}
    \item $Q(x,y) = \left\{ \begin{array}{ll} 100.0 &\textrm{if } x_2 > 0 \\ 1.0 \label{Q=step} & \textrm{if } x_2 < 0 \end{array} \right.$,     \item $\rho_0 = \exp(-(|x_1+1|^2 + x_2^2)/\sigma) + \exp(-(|x_1-1|^2 + x_2^2)/\sigma) +  \exp(-(|x_1-1|^2 + |x_2-1.95|^2)/(\sigma/2)) \newline \hspace{7mm} +  \exp(-(|x_1+1|^2 + |x_2+1.95|^2)/(\sigma/2)), \quad \sigma = 0.1$.  \label{eq: initial datum experiment explore}
\end{enumerate}

The experiment whose results are displayed in Figure \ref{fig: experiment 2 heat maps} shows that the choice of $A(x)$ and $Q(x)$  affects the nonlocality introduced by the fractional power of the differential operator $\LN$.  
The expression for $Q(x)$ reduces the value of the potential $c$ on the upper part of the domain $\Omega$, where $y>0$. Therefore, even if equal densities are placed respectively at the positions $x_1=-1.0, x_2=-1.95$ and $x_1=1.0, x_2=1.95$,  only the first induces an effect on the part of the density concentrated initially at the point $x_1=-1.0, x_2=0.0$, which becomes higher than the one symmetrically placed in $x_1=1.0, x_2=0.0$. In fact, we can see that the density at $x_1=-1.0, x_2=0.0$ starts immediately to absorb some of the concentration below it. This effect finds further validation in the fact that, at time $t=0.1$, the potential $c$ is already present in the bottom left area of the domain, in the region between the two initial concentrations placed along the axis $x_1=-1.0$. The interactions between the concentrations at $x_1=-1.0, x_2=0.0$ and $x_1=-1.0, x_2=1.95$ are absent up to larger times, when the concentrations have diffused in the upper right region of $\Omega$.

These two experiments illustrate the considerable versatility of the general model for modeling purposes, particularly for those applications that could require spatial dependence of the potential $c$. The hypotheses on $A(x)$ and $Q(x)$ allow for a variety of configurations, able to take into account many possible spatial effects.

\begin{figure}
\begin{subfigure}{0.5\textwidth}
\centering
\includegraphics[width = 0.6 \textwidth]{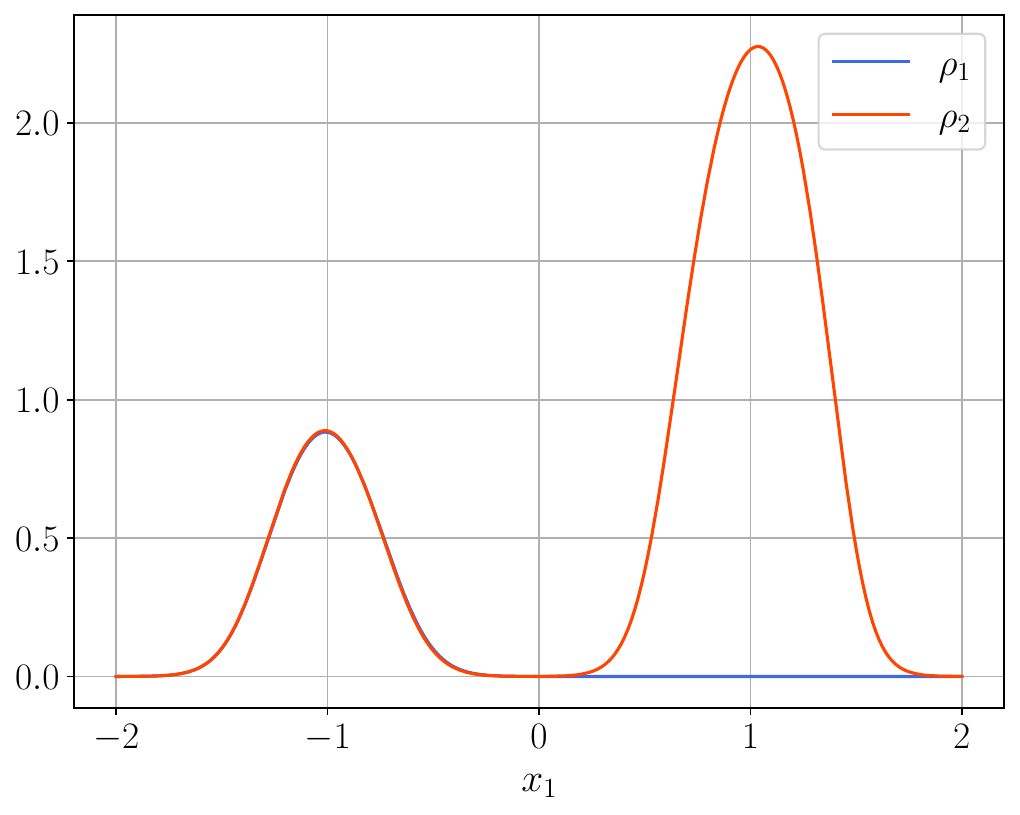}
\caption{$t=0.1$}
\end{subfigure}
\hfill
\begin{subfigure}{0.5 \textwidth}
\centering
\includegraphics[width = 0.6 \textwidth]{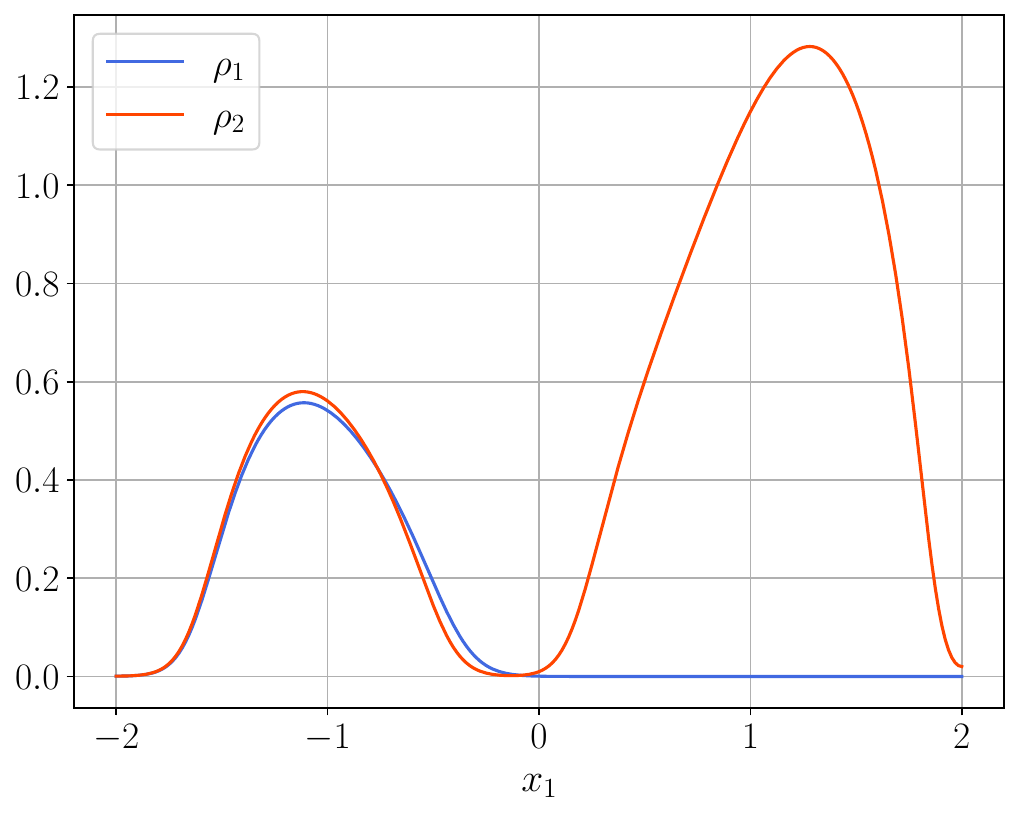}
\caption{$t=0.5$}
\end{subfigure}

\begin{subfigure}{0.5 \textwidth}
\centering
\includegraphics[width = 0.6 \textwidth]{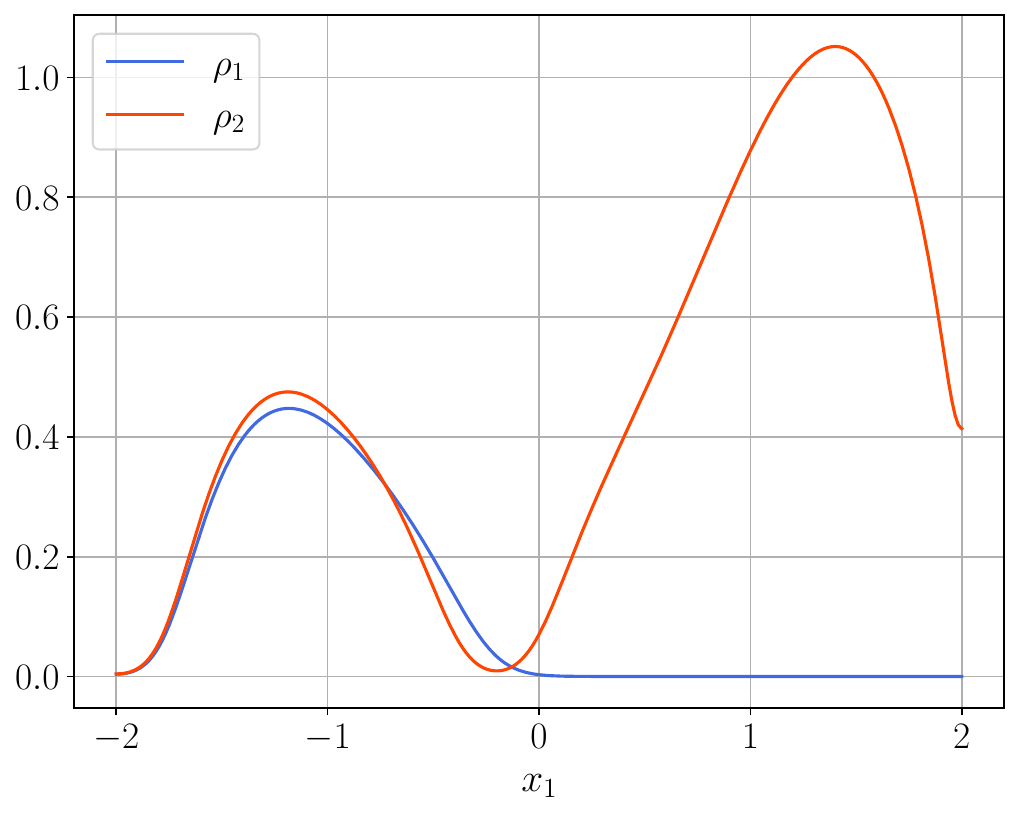}
\caption{$t=1.0$}
\end{subfigure}
\hfill
\begin{subfigure}{0.5 \textwidth}
\centering
\includegraphics[width = 0.6 \textwidth]{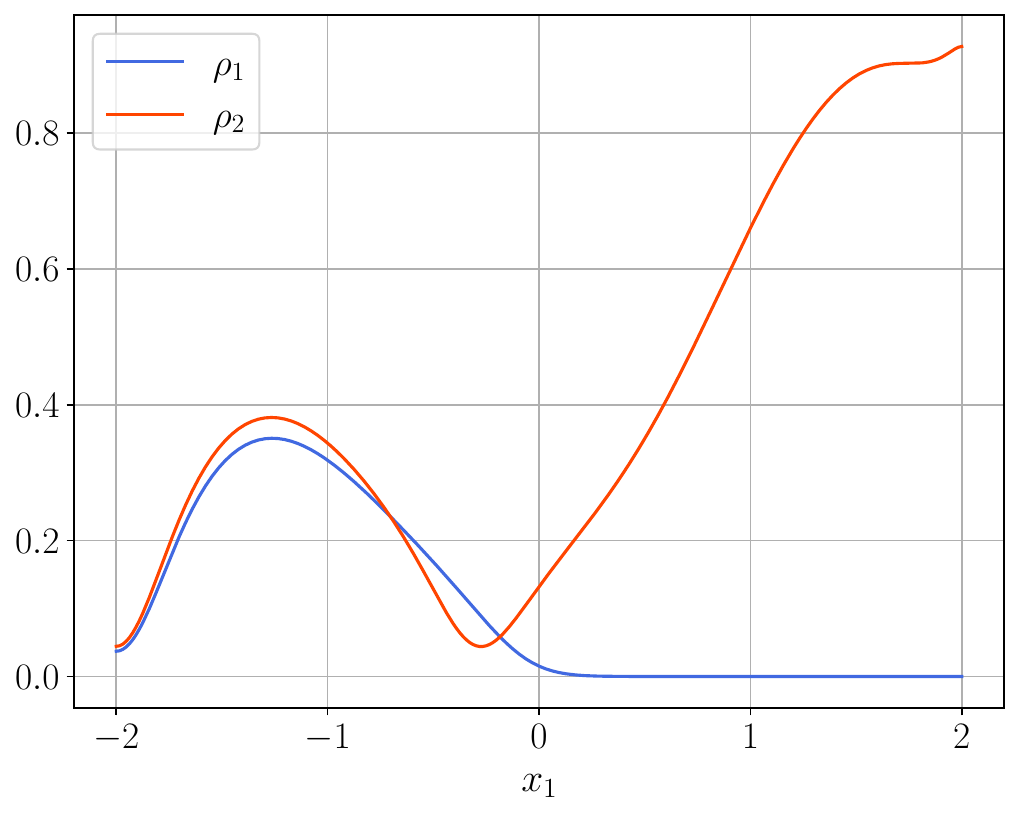}
\caption{$t=1.8$}
\end{subfigure}
\caption{Numerical solutions to the generalised porous medium equation with fractional pressure \eqref{general MainFullProb}, cross-section at $x_2 = 0$. The simulations were performed on $\Omega =(-2,2)^2$ for the fractional order $s=0.75$ with $h = |\Omega|/(2^8 \cdot 2^8)$, and time step $\Delta t = 0.05$. The matrix $A(x)$ was given by \ref{A10|01} and the function $Q(x)$ by \ref{Q=|x|^2}. The initial data were \ref{eq: initial data comparison principle}. The parabolic regularisation coefficient was $\mu = 0.01$. }
\label{fig: experiment 1 cross sections}
\end{figure}

\begin{figure}
\begin{subfigure}{0.48 \textwidth}
\centering
\includegraphics[width = \textwidth]{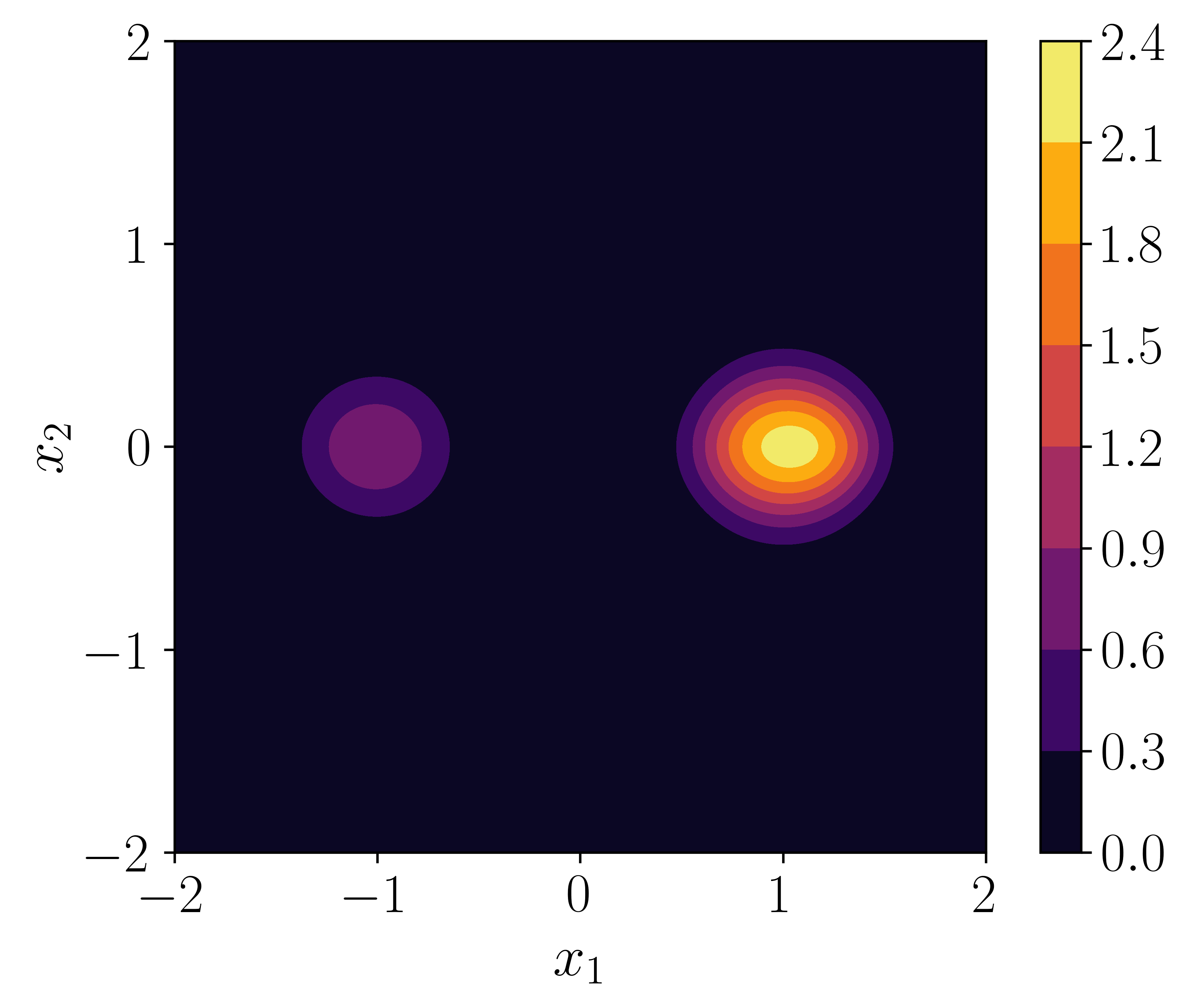}
\caption{$\rho_2$ at $t=0.1$}
\end{subfigure}
\hfill
\begin{subfigure}{0.48 \textwidth}
\centering
\includegraphics[width = \textwidth]{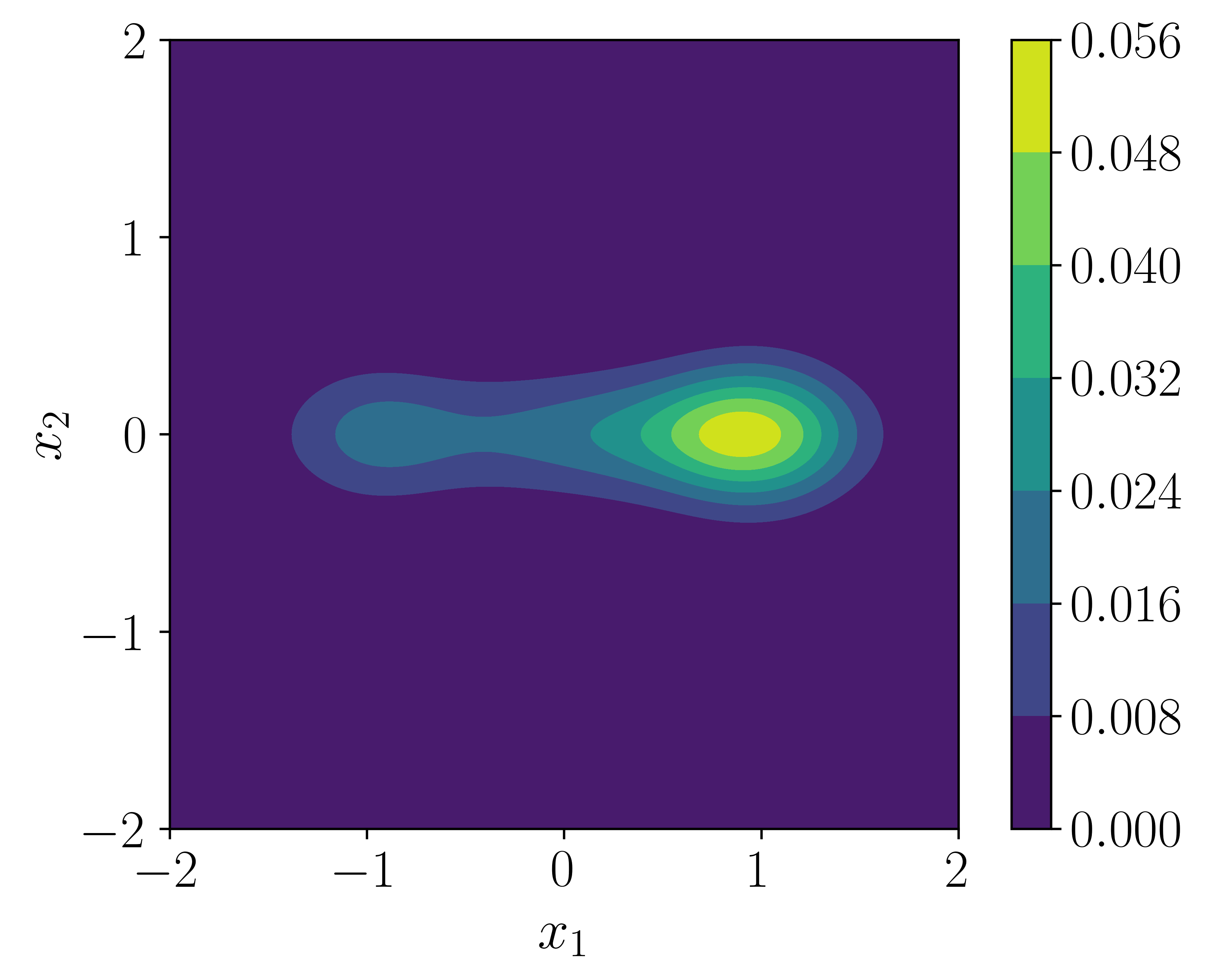}
\caption{$c_2$ at $t=0.1$}
\end{subfigure}
\hfill
\begin{subfigure}{0.48 \textwidth}
\centering
\includegraphics[width = \textwidth]{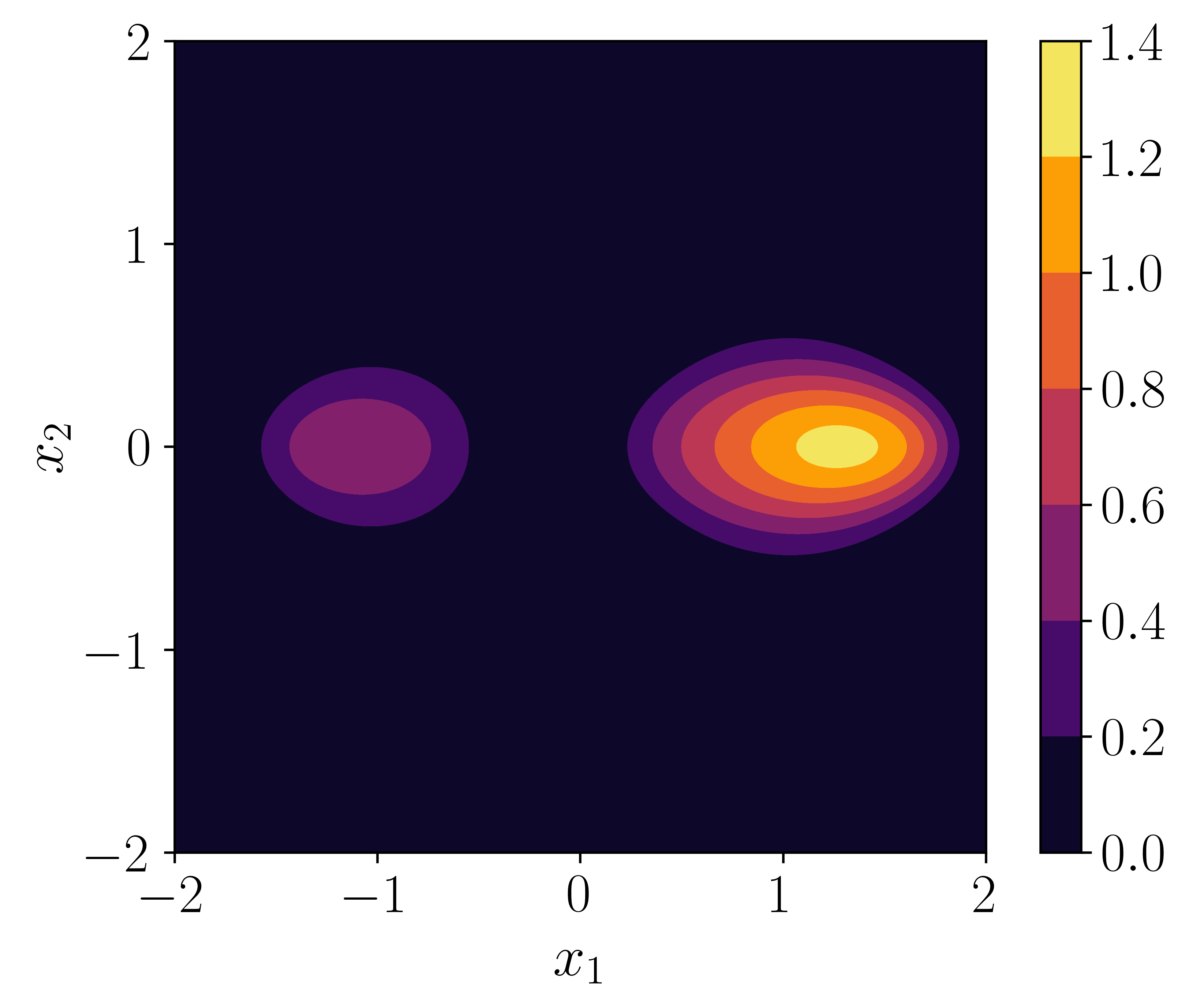}
\caption{$\rho_2$ at $t=1.0$}
\end{subfigure}
\hfill
\begin{subfigure}{0.48 \textwidth}
\centering
\includegraphics[width = \textwidth]{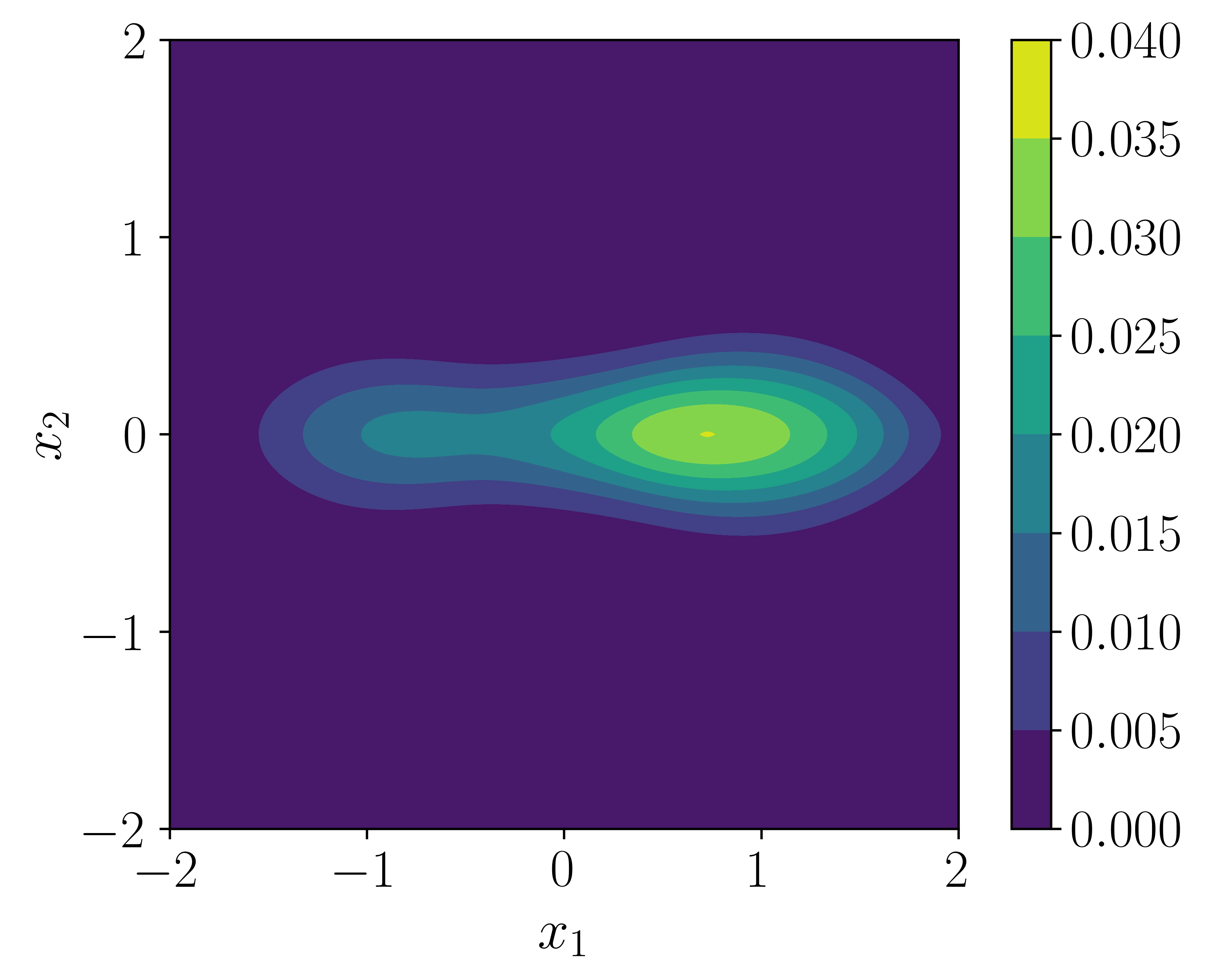}
\caption{$c_2$ at $t=1.0$}
\end{subfigure}
\hfill
\begin{subfigure}{0.48 \textwidth}
\centering
\includegraphics[width = \textwidth]{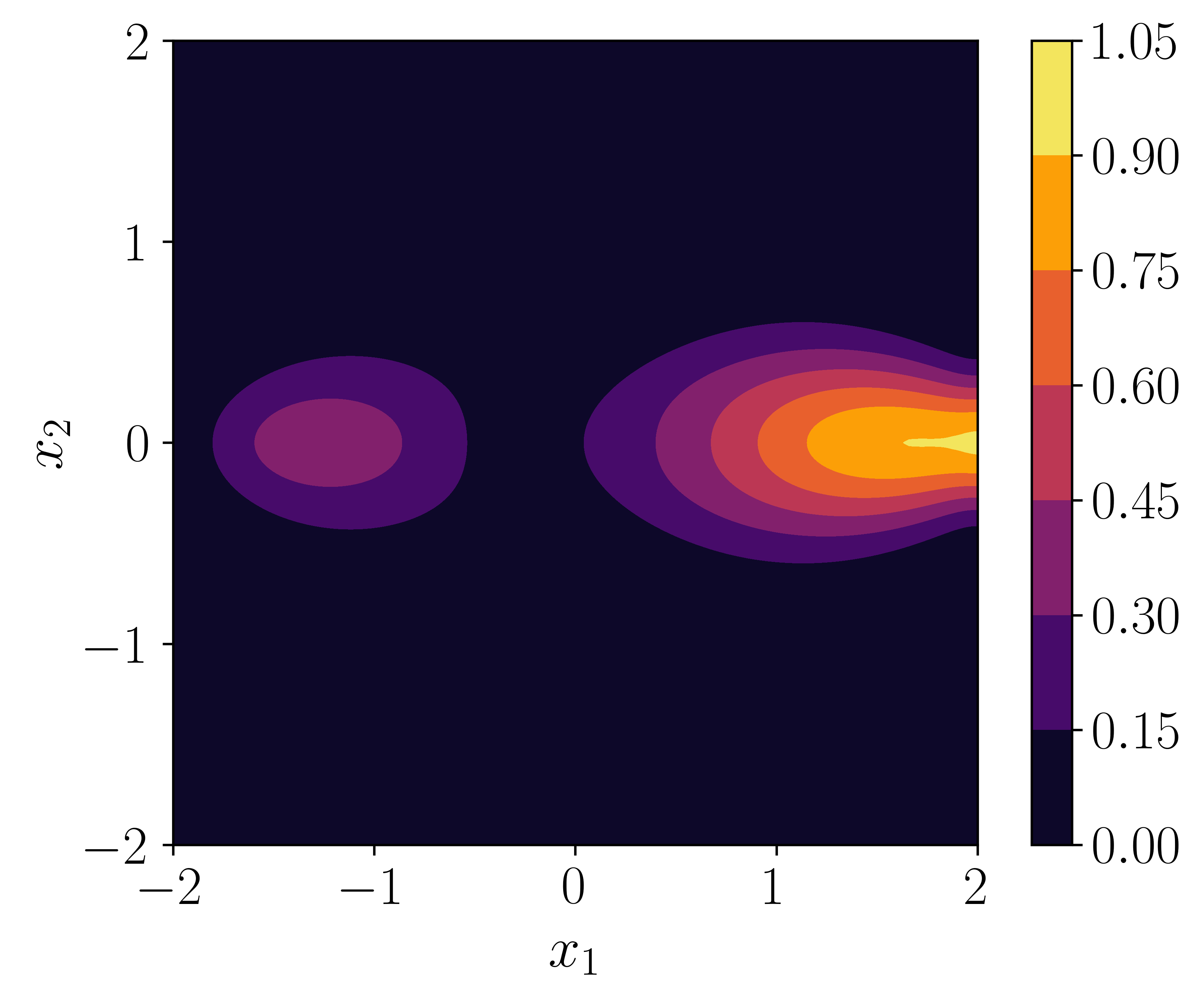}
\caption{$\rho_2$ at $t=1.8$}
\end{subfigure}
\hfill
\begin{subfigure}{0.48 \textwidth}
\centering
\includegraphics[width = \textwidth]{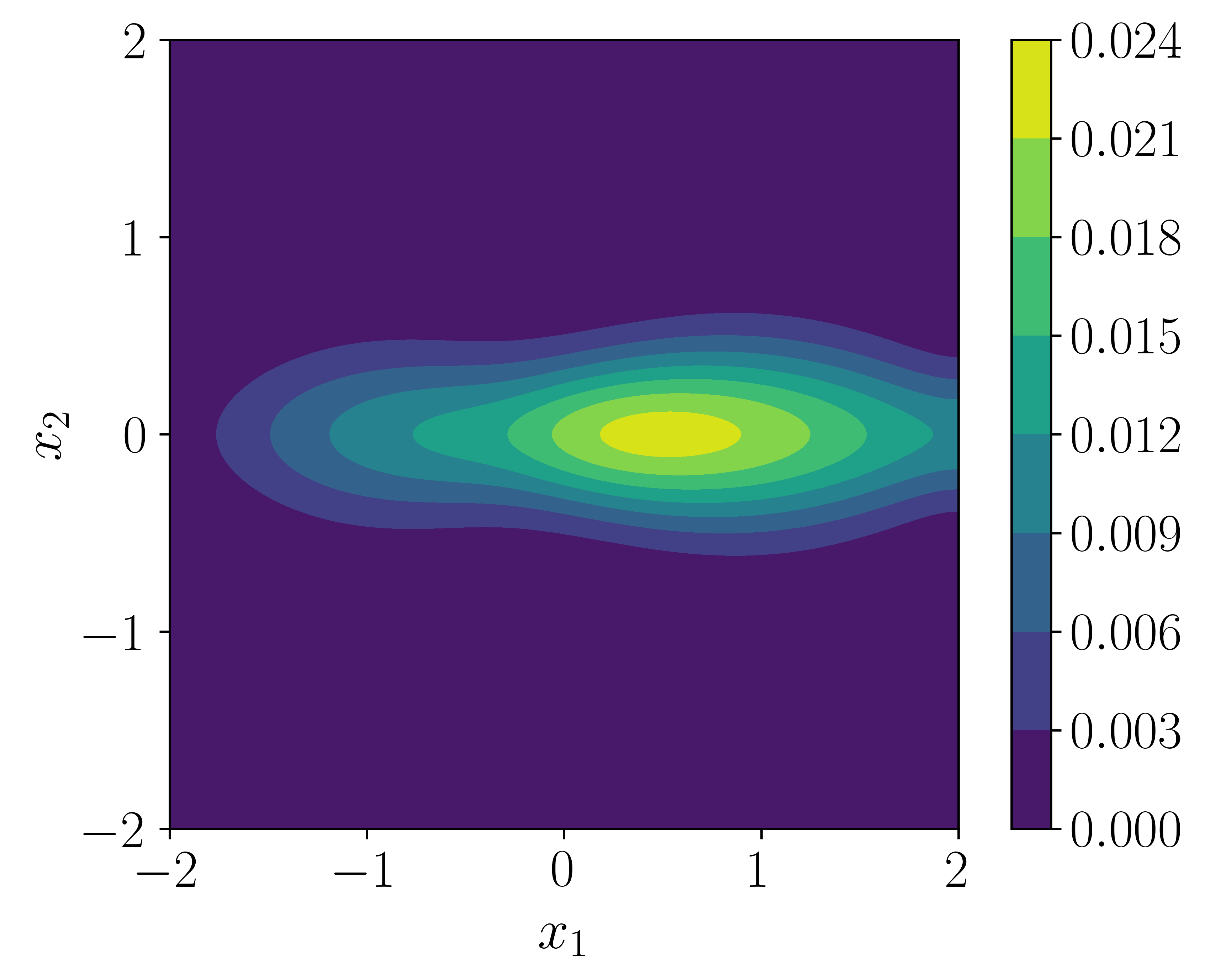}
\caption{$c_2$ at $t=1.8$}
\end{subfigure}
\caption{Numerical solutions to the generalised porous medium equation with fractional pressure \eqref{general MainFullProb}. The simulations were performed on $\Omega =(-2,2)^2$ for the fractional order $s=0.75$ with $h = |\Omega|/(2^8 \cdot 2^8)$, and time step $\Delta t = 0.05$. The matrix $A(x)$ was given by \ref{A10|01} and the function $Q(x)$ by \ref{Q=|x|^2}. The initial datum was \ref{eq: initial data comparison principle}. The parabolic regularisation coefficient was $\mu = 0.01$. }
\label{fig: experiment 1 heat maps}
\end{figure}

\begin{figure} 
\begin{subfigure}{0.48 \textwidth}
\centering
\includegraphics[width = \textwidth]{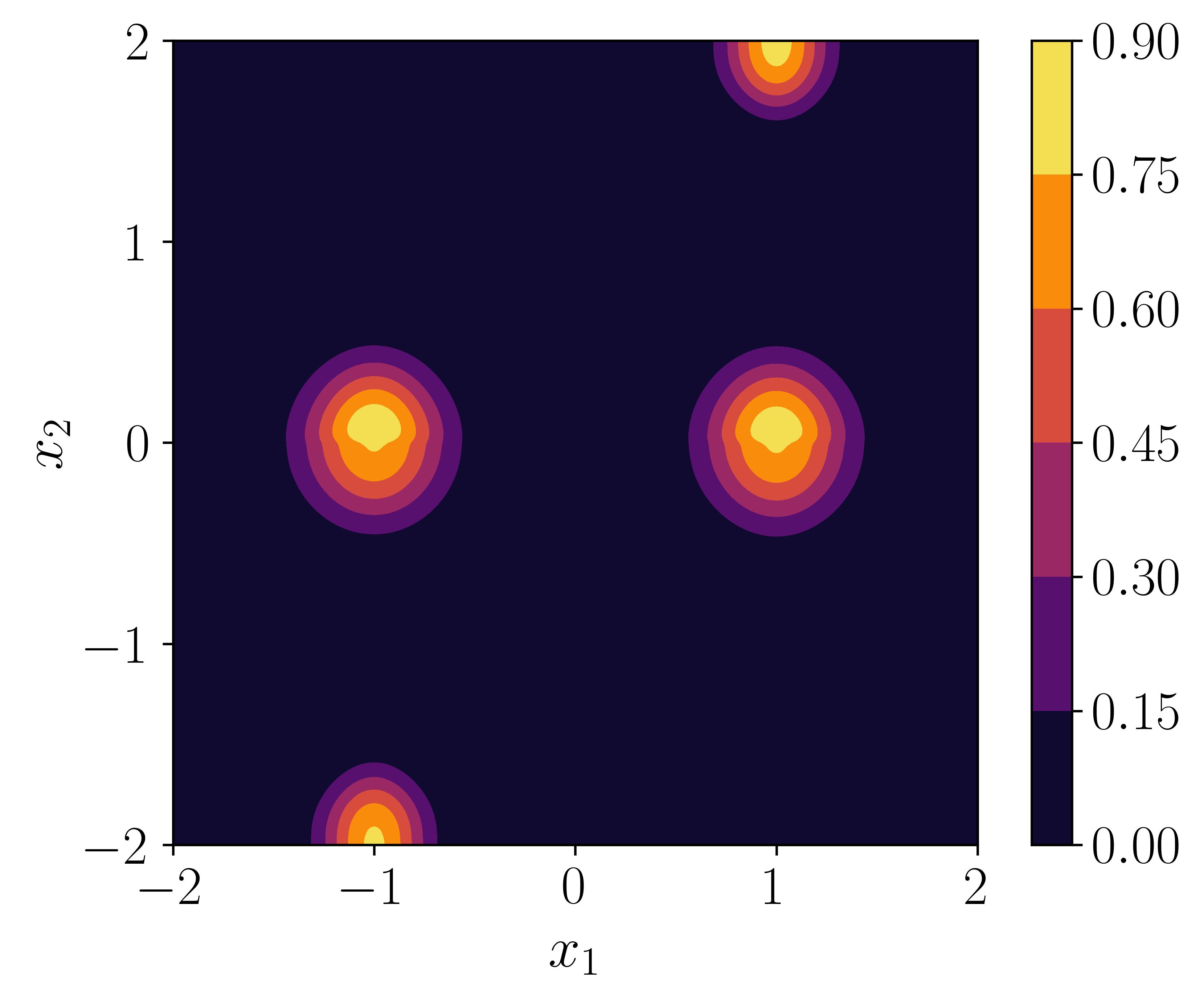}
\caption{$\rho$ at $t=0.1$}
\end{subfigure}
\hfill
\begin{subfigure}{0.48 \textwidth}
\centering
\includegraphics[width = \textwidth]{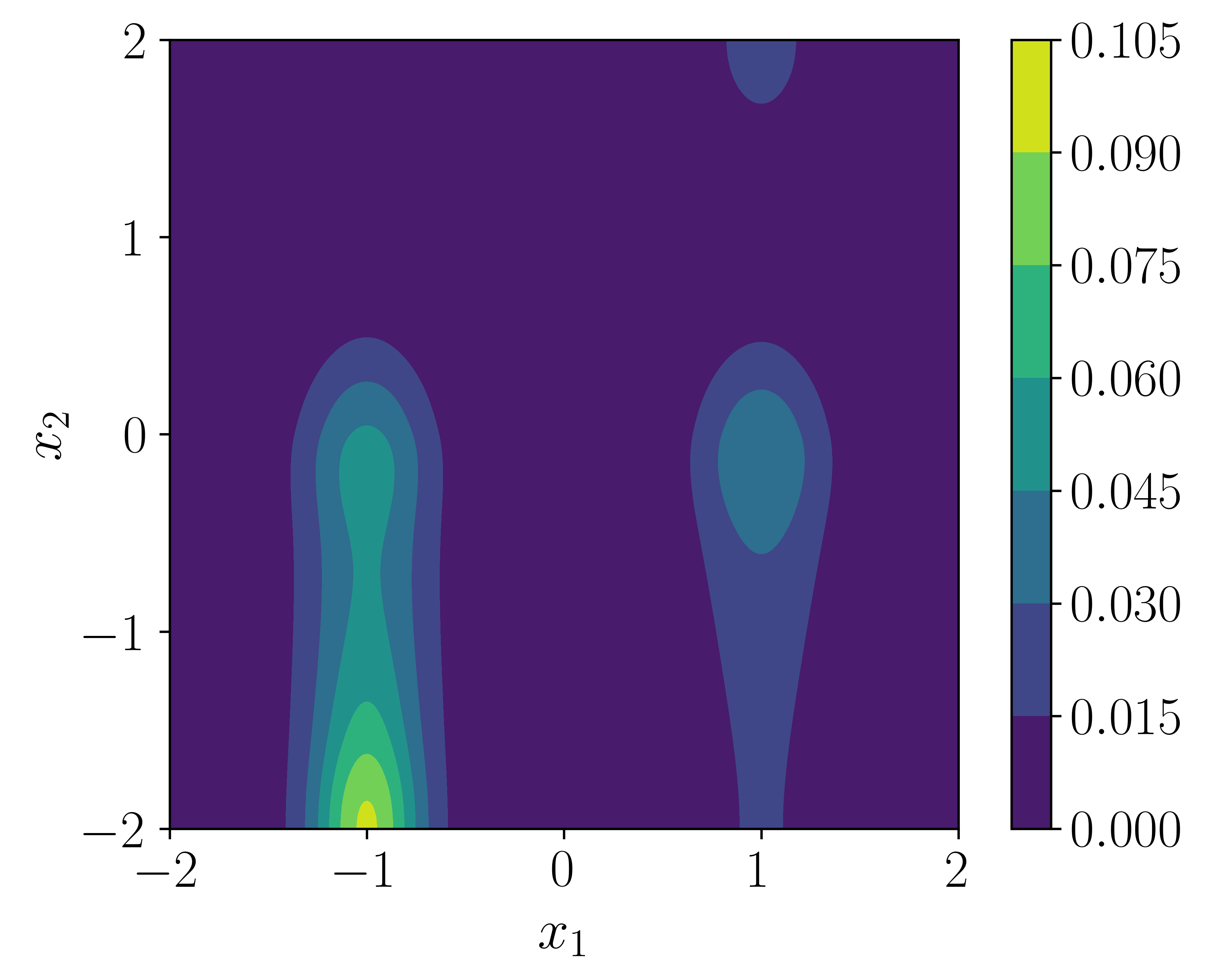}
\caption{$c$ at $t=0.1$}
\end{subfigure}
\hfill
\begin{subfigure}{0.48 \textwidth}
\centering
\includegraphics[width = \textwidth]{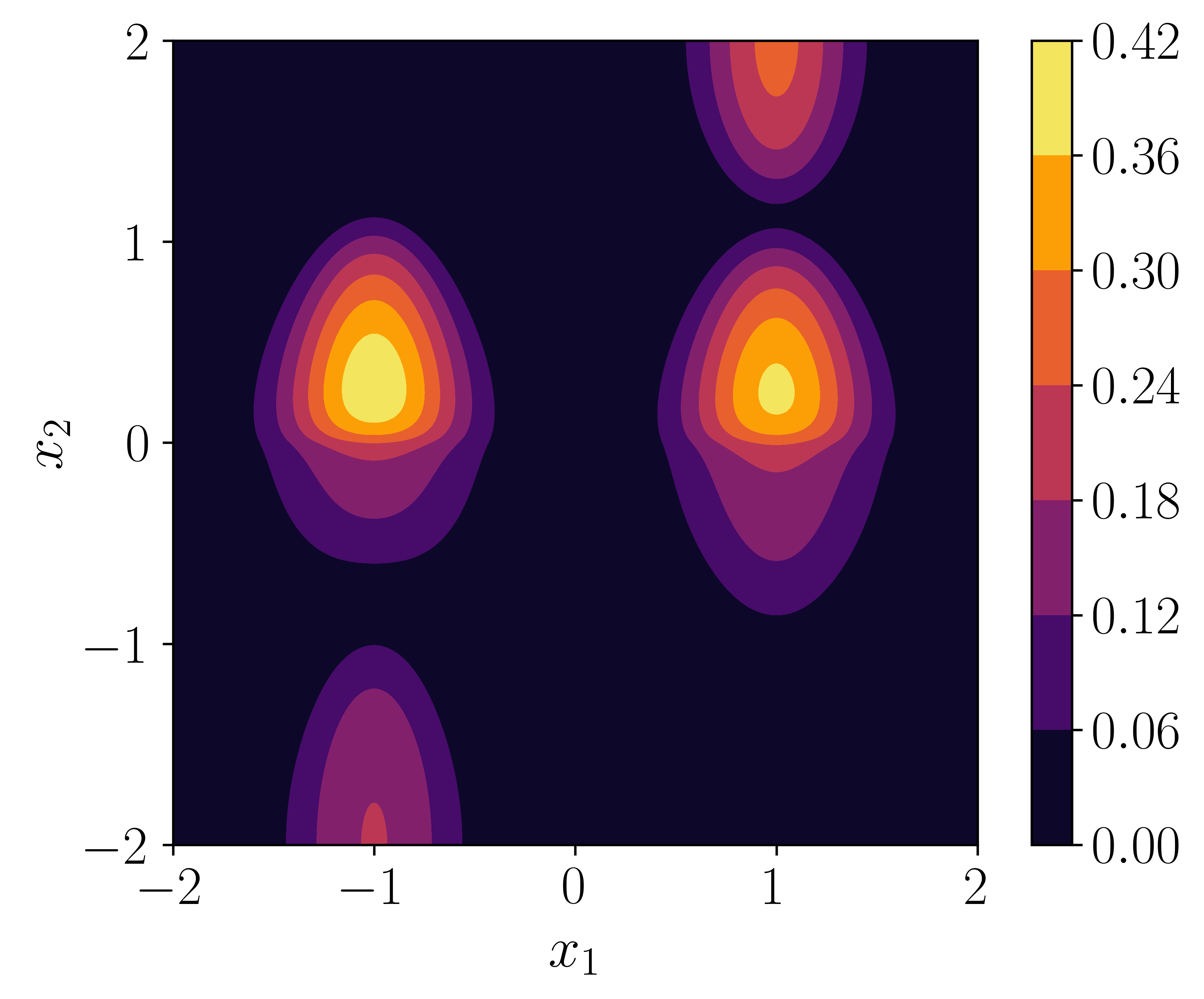}
\caption{$\rho$ at $t=1.5$}
\end{subfigure}
\hfill
\begin{subfigure}{0.48 \textwidth}
\centering
\includegraphics[width = \textwidth]{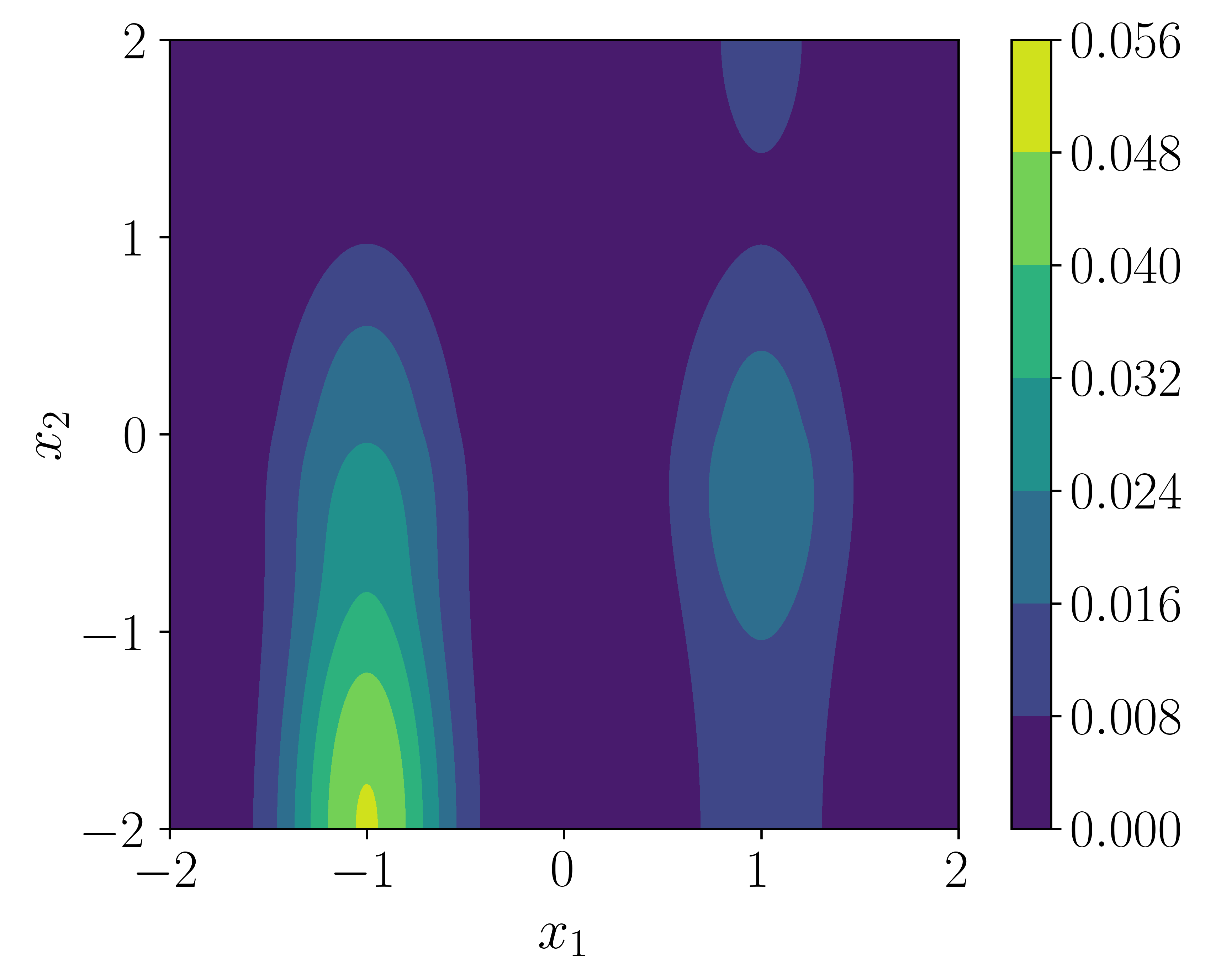}
\caption{$c$ at $t=1.5$}
\end{subfigure}
\hfill
\begin{subfigure}{0.48 \textwidth}
\centering
\includegraphics[width = \textwidth]{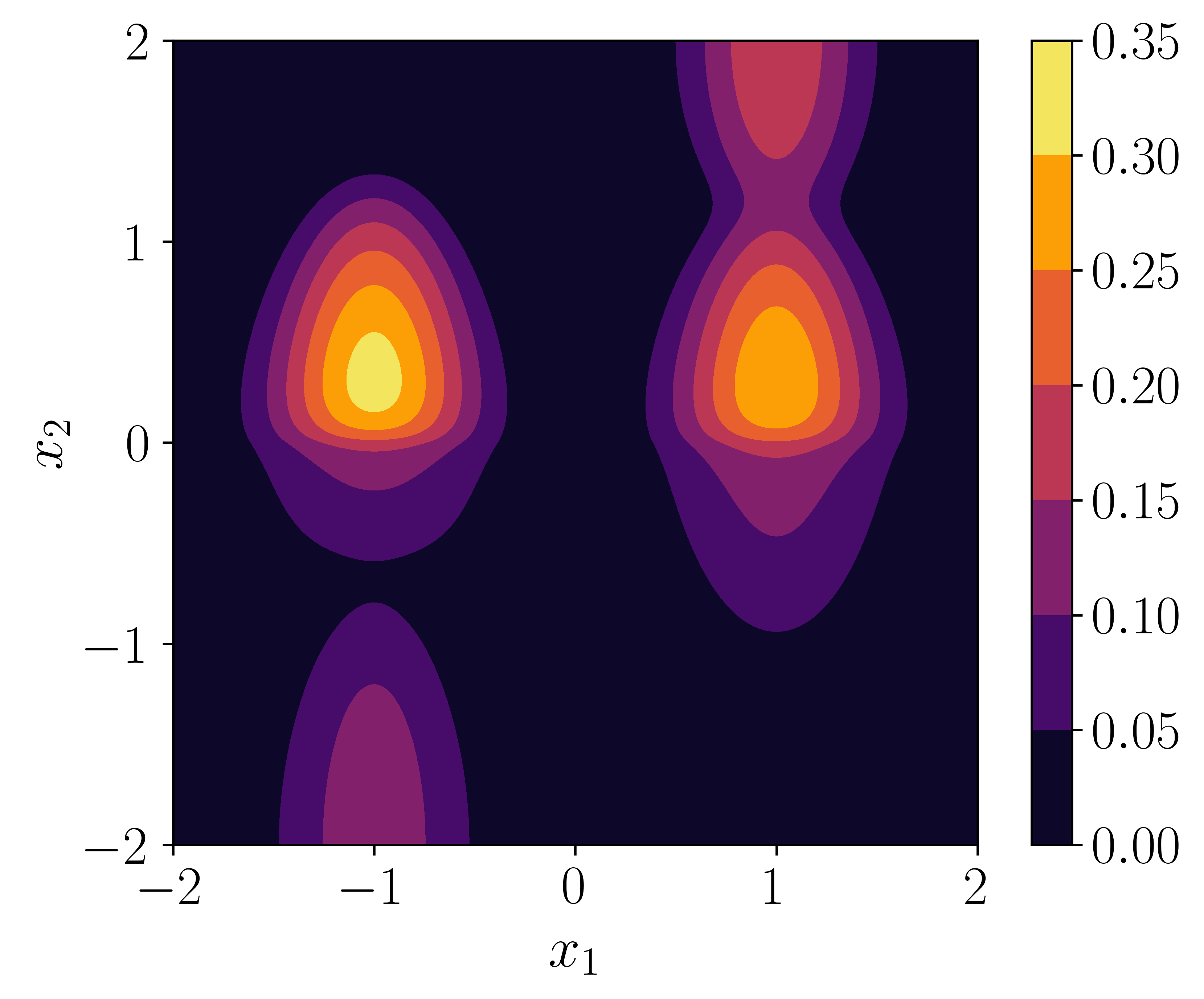}
\caption{$\rho$ at $t=2.5$}
\end{subfigure}
\hfill
\begin{subfigure}{0.48 \textwidth}
\centering
\includegraphics[width = \textwidth]{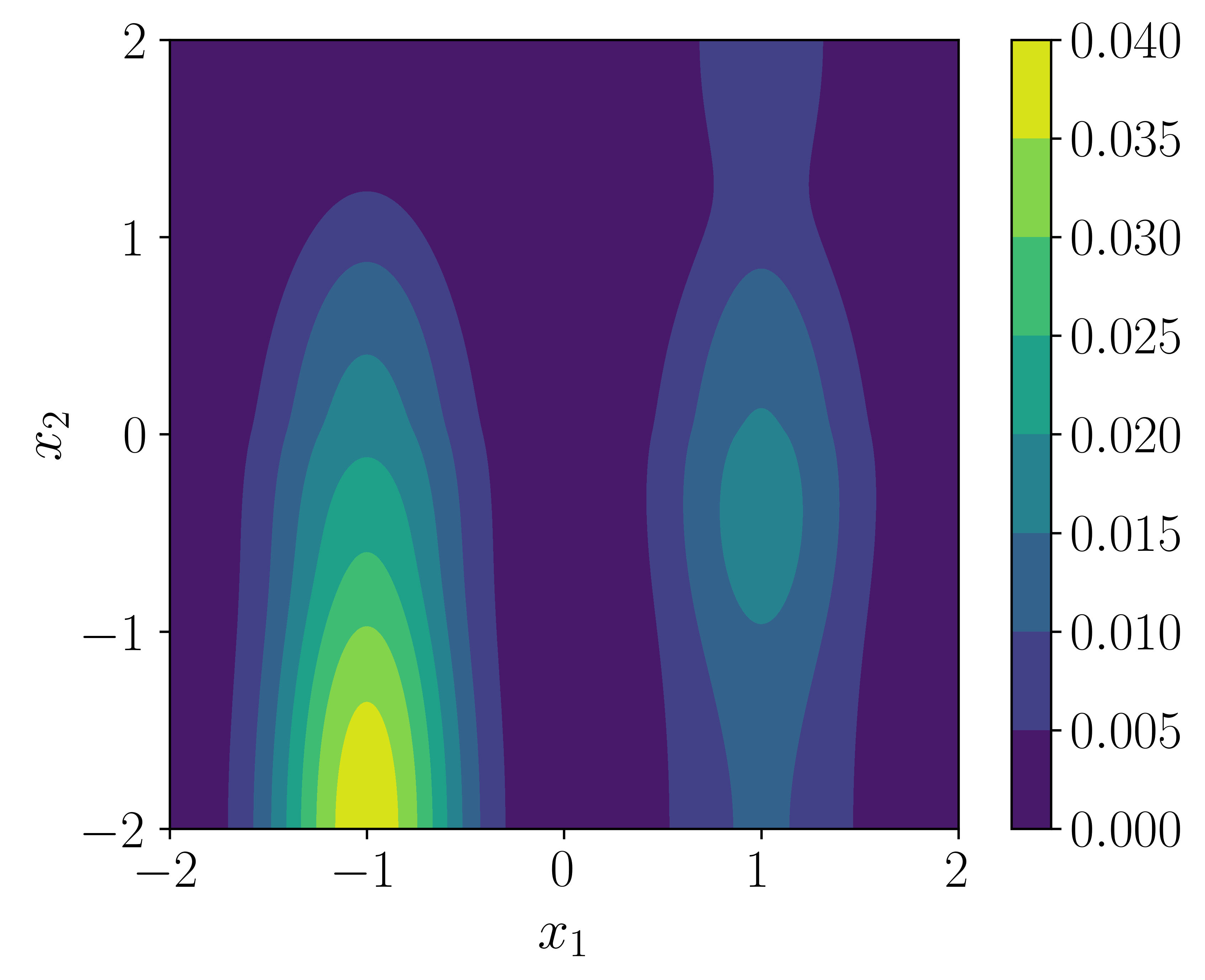}
\caption{$c$ at $t=2.5$}
\end{subfigure}
\caption{Numerical solutions to the generalised porous medium equation with fractional pressure \eqref{general MainFullProb}. The simulations were performed on $\Omega =(-2,2)^2$ for the fractional order $s=0.67$ with $h = |\Omega|/(2^8 \cdot 2^8)$, and time step $\Delta t = 0.05$.  The matrix $A(x)$ was given by \ref{A01|10} and the function $Q(x)$ by \ref{Q=step}. The initial data was \ref{eq: initial datum experiment explore}. The parabolic regularisation coefficient was $\mu = 0.01$. }
\label{fig: experiment 2 heat maps}
\end{figure}

We now investigate the long-time behaviour of the model. This is particularly interesting because in \cite[Theorem 4.8]{Fronzoni2025} the authors proved exponential decay in $L^1(\Omega)$ norm to the uniform steady state for the case $\LL =-\Delta$. In our case, we are unable to use the same technique: a closer inspection to the proof of \cite[Theorem 4.8]{Fronzoni2025} shows that the presence on the right-hand side of \eqref{ExtUniformBoundSpatialDer} of an additional term deriving from the potential $Q(x)$ does not allow us to use the same argument to infer exponential decay of the term $\int_\Omega G(\widehat{\rho}(x,t)) \dx$ and consequently exponential decay of the solution to the uniform steady state via the Csiszár–Kullback inequality. As a matter of fact, we observe that if the function $Q(x)$ is not identically equal to zero, the uniform steady state is no longer a (trivial) solution for the model \eqref{general MainFullProb}. We can observe this fact in Figure \ref{fig: non trivial steady state Q=|x|^2}: taking $Q(x)=|x|^2$ we see that the numerical solution converges to a non-trivial steady state. Conversely, in Figure \ref{fig: exponential decay Q=0} we observe what happens when $Q =0$  even in the presence of anisotropy $A(x) \neq I$: we computed simulations for differential initial data; measuring the $L^1(\Omega)$ norm distance between the numerical solution and the uniform distribution, we see that all solutions converge to the constant steady state. 

Although we are not able to prove exponential decay of the solution to a steady state we can comment on our computational results. We first observe in Figure \ref{fig: exponential decay Q=0} (the case \ref{(K1)}, $Q(x)=0$) that, for all the initial data that we tested, for large times we see an exponential decay towards the steady state. Moreover, for the two initial data that are the furthest from the equilibrium (a Gaussian and a blob centred at a corner of the square domain $\Omega$) the slope of the decay, quite interestingly, matches the one of \cite[Theorem 4.8]{Fronzoni2025}: $-2t/C_\Omega$ for $C_\Omega = 1/\lambda_1$ with $\lambda_1$ being the first positive eigenvalue of the Neumann Laplacian. 
Even more interestingly, in the case \ref{(K0)} (nonvanishing $Q(x)$) we can make similar observations. The experiments in this case are reported in Figure \ref{fig: exponential decay Q=|x|^2} and, since the steady state does not have an explicit formula, we compute the $L^1(\Omega)$ norm of the distance between successive time steps, i.e. $\| \widehat{\rho}(\cdot, t_n) - \widehat{\rho}(. t_{n-1}) \|_{L^1(\Omega)}$. We  observe exponential decay and that the presumed sharp rate of  $-2t/C_\Omega$ is also attained for the Gaussian and the blob centred at a corner of the domain. We believe that this computational evidence enforces our intuition that exponential decay to the steady state holds also in our case and we are confident that this observation could open directions for proving this fact.    

\begin{figure} 

\begin{subfigure}{0.5 \textwidth}
\centering
\includegraphics[width = 0.55 \textwidth]{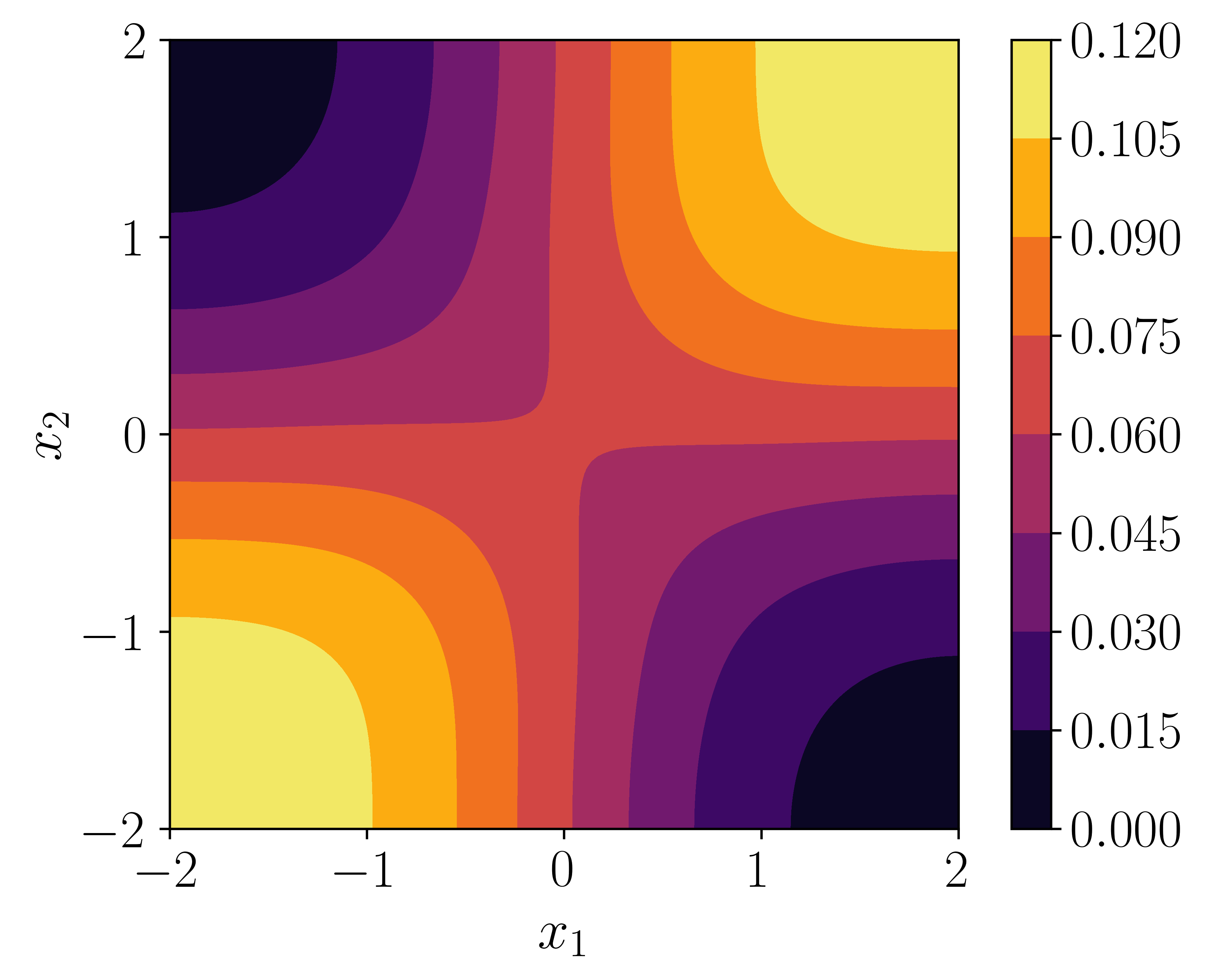}
\caption{$\rho$ at $t=0.5$}
\end{subfigure}
\hfill
\begin{subfigure}{0.5 \textwidth}
\centering
\includegraphics[width = 0.55 \textwidth]{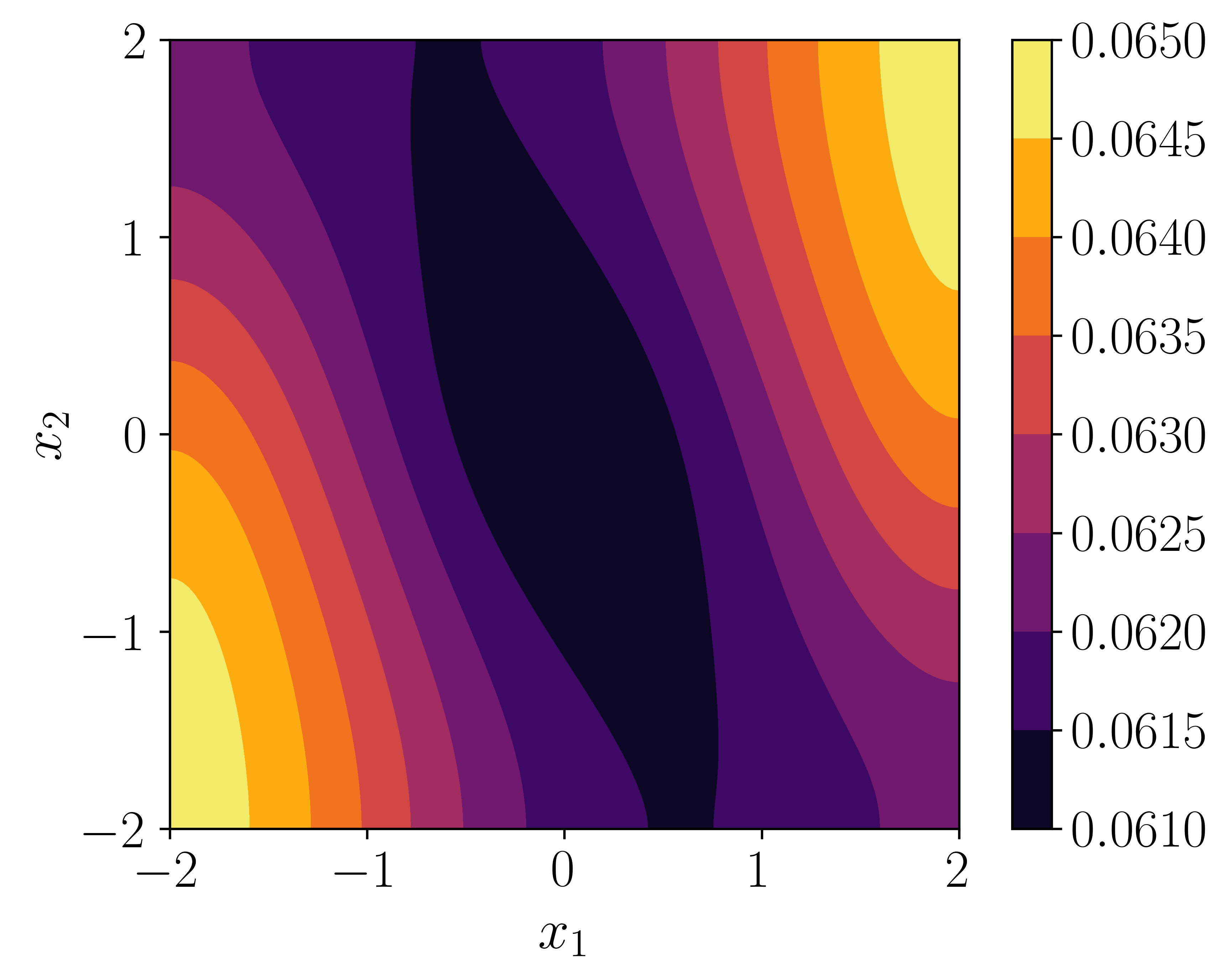}
\caption{$\rho$ at $t=3.50$}
\end{subfigure}
\hfill
\begin{subfigure}{0.5 \textwidth}
\centering
\includegraphics[width = 0.55 \textwidth]{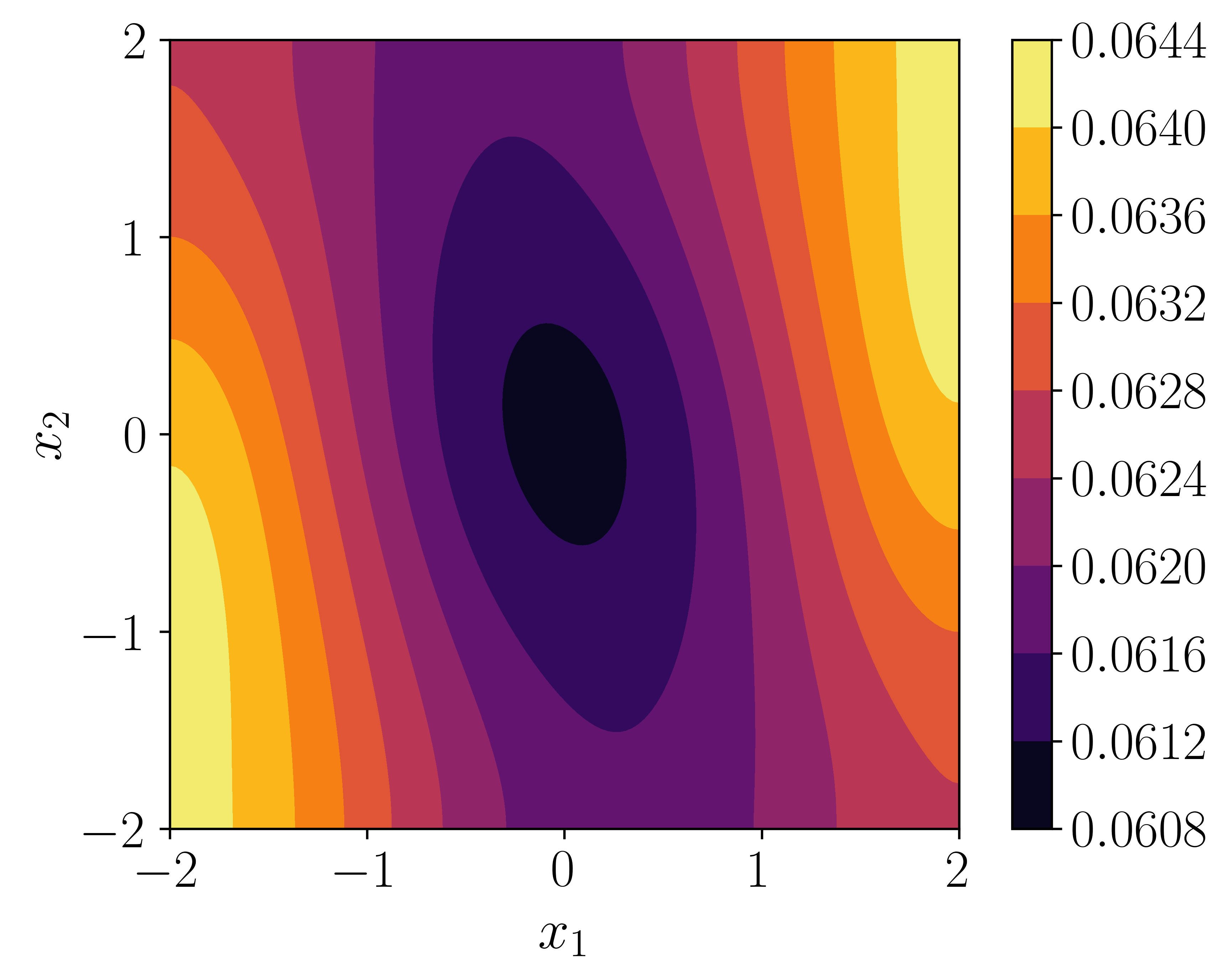}
\caption{$\rho$ at $t=4.0$}
\end{subfigure}
\hfill
\begin{subfigure}{0.5 \textwidth}
\centering
\includegraphics[width = 0.55 \textwidth]{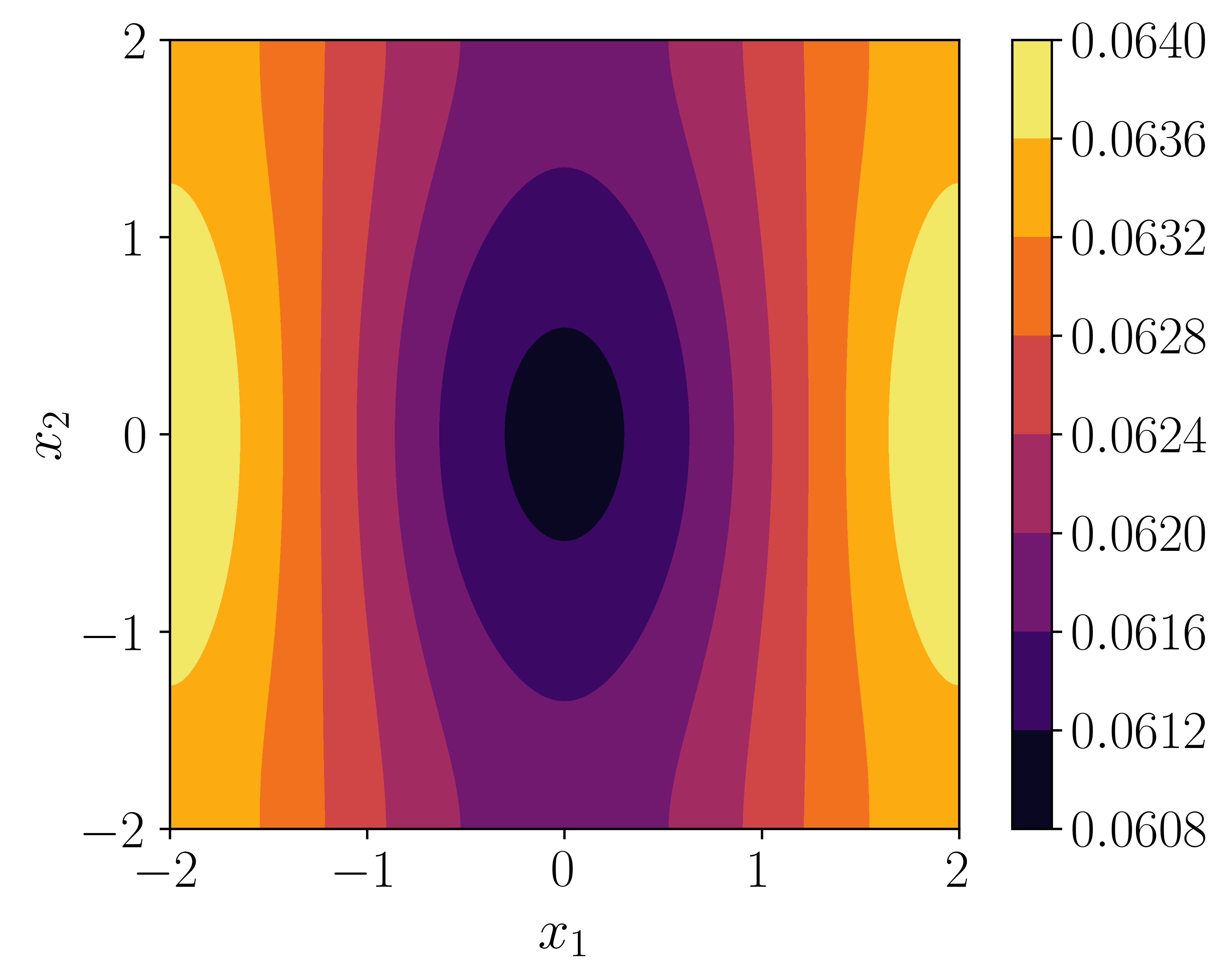}
\caption{$\rho$ at $t=10.0$}
\end{subfigure}
\caption{Numerical solutions to the generalised porous medium equation with fractional pressure \eqref{general MainFullProb}. The simulations were performed on $\Omega =(-2,2)^2$ for the fractional order $s=0.75$ with $h = |\Omega|/(2^7 \cdot 2^7)$, and time step $\Delta t = 0.01$. The matrix $A(x)$ was given by \ref{A10|01} and the function $Q(x)=|x|^2$. The initial datum was two Gaussians $\rho=\frac{1}{\sqrt{2\pi} \sigma} \mathrm{e}^{-\frac{|x-P_1|^2}{\sigma^2}} + \frac{1}{\sqrt{2\pi} \sigma} \mathrm{e}^{-\frac{|x-P_2|^2}{\sigma^2}}$, with $P_1=(1,1)$ and $P_2=(-1,-1)$, and $\sigma=0.4$ The parabolic regularisation coefficient was $\mu = 1$. }
\label{fig: non trivial steady state Q=|x|^2}
\end{figure}

\begin{figure}
\centering
\includegraphics[width = 0.64\textwidth]{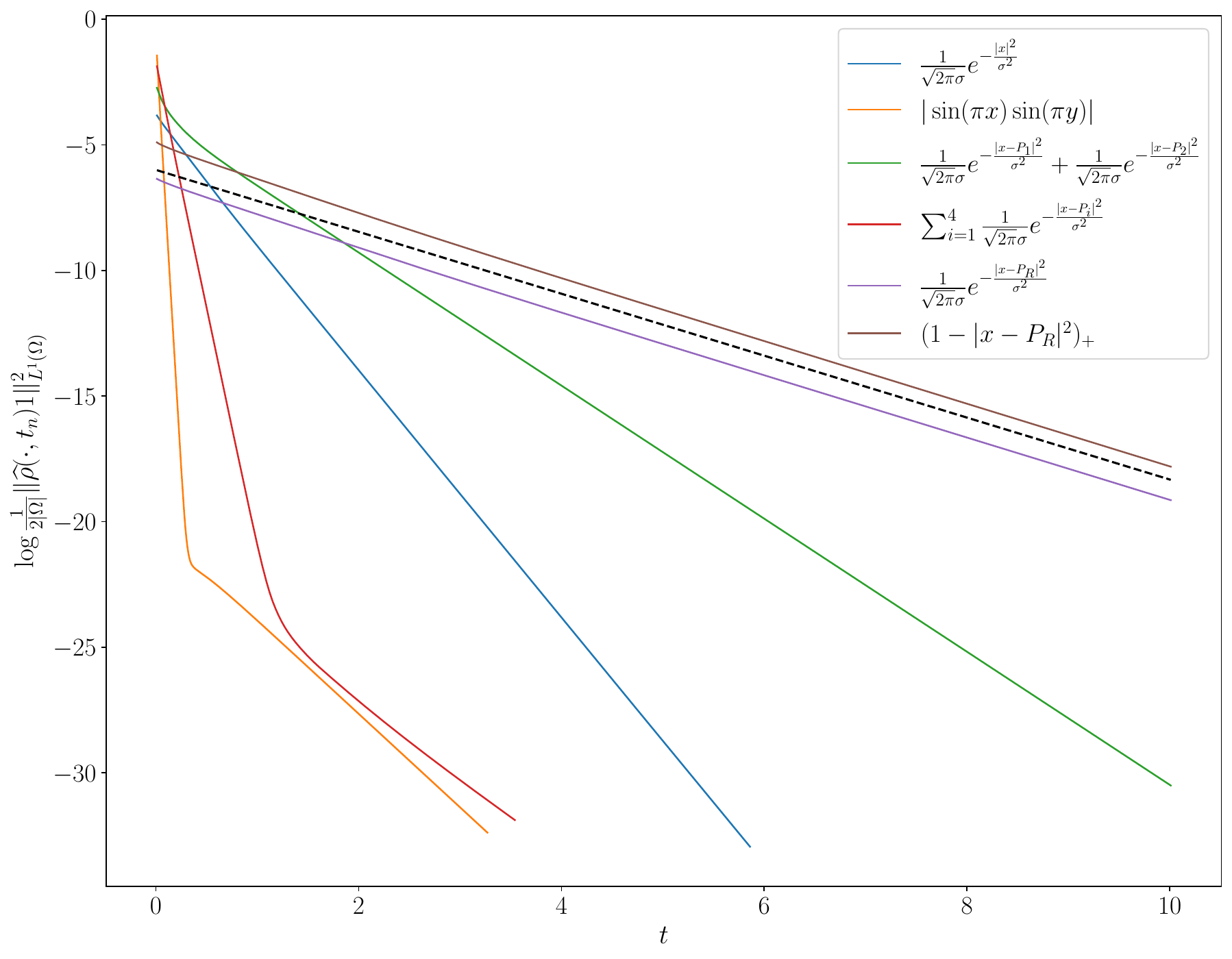}
\caption{Numerical solutions to the generalised porous medium equation with fractional pressure \eqref{general MainFullProb}. The matrix $A(x)$ was given by \ref{A10|01} and the function $Q(x)=0$. The simulations were performed on $\Omega =(-2,2)^2$ for the fractional order $s=0.75$ with $h = |\Omega|/(2^7 \cdot 2^7)$, and time step $\Delta t = 0.01$. The parabolic regularisation coefficient was $\mu = 1$. Parameters for the initial data were $\sigma=0.4$, $P_1=(1,1)$, $P_1=(-1,-1)$, $P_3=(1,-1)$, $P_4=(-1,1)$, $P_R=(2,2)$. The dashed black reference line has slope $-2t/C_\Omega$ for $C_\Omega = 1/\lambda_1$, $\lambda_1$ the first positive eigenvalue of the Neumann Laplacian.}
\label{fig: exponential decay Q=0}
\end{figure}

\begin{figure} 
\centering
\includegraphics[width = 0.64\textwidth]{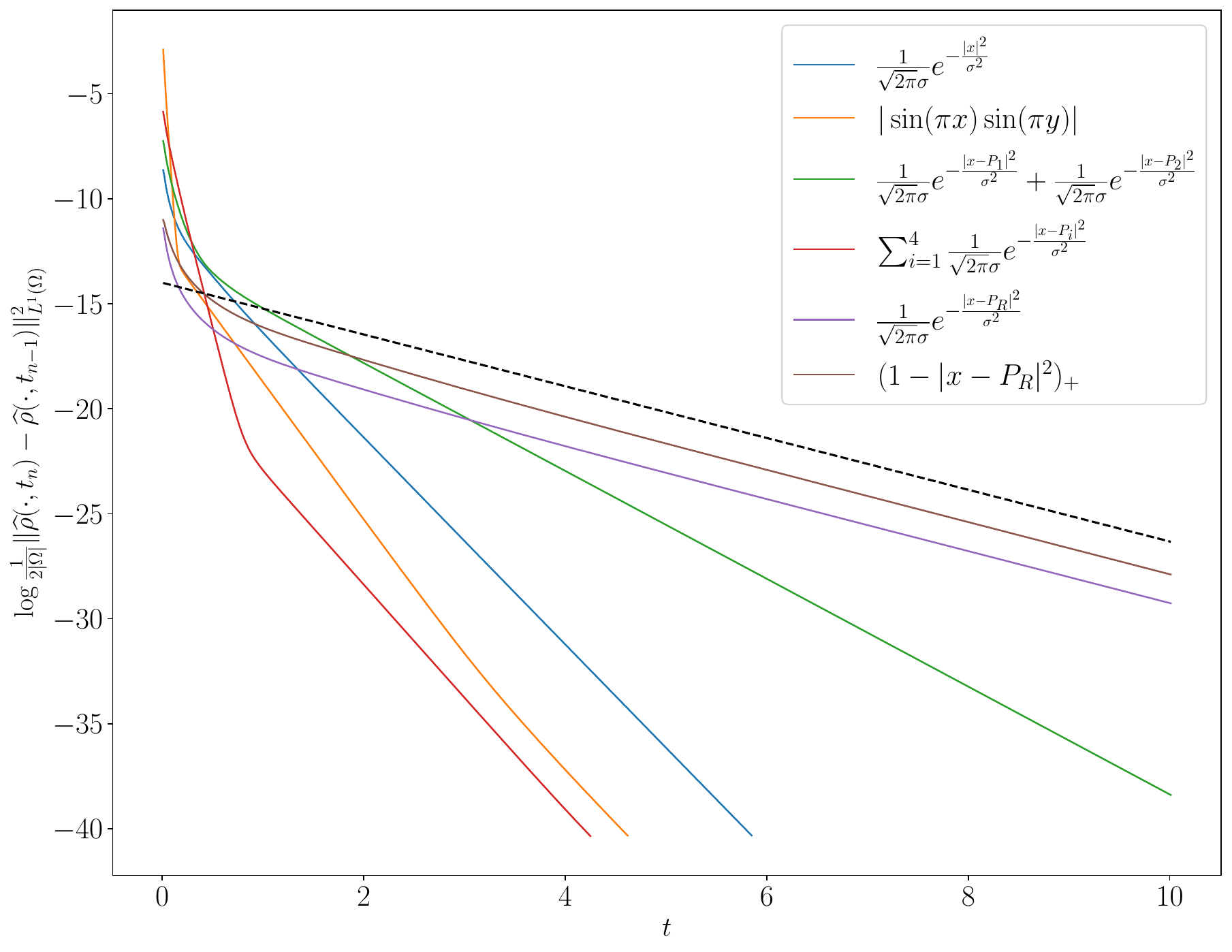}
\caption{Numerical solutions to the generalised porous medium equation with fractional pressure \eqref{general MainFullProb}. The matrix $A(x)$ was given by \ref{A01|10} and the function $Q(x)=|x|^2$. The simulations were performed on $\Omega =(-2,2)^2$ for the fractional order $s=0.75$ with $h = |\Omega|/(2^7 \cdot 2^7)$, and time step $\Delta t = 0.01$. The parabolic regularisation coefficient was $\mu = 1$. Parameters for the initial data were $\sigma=0.4$, $P_1=(1,1)$, $P_1=(-1,-1)$, $P_3=(1,-1)$, $P_4=(-1,1)$, $P_R=(2,2)$. The dashed black reference line has slope $-2t/C_\Omega$ for $C_\Omega = 1/\lambda_1$, $\lambda_1$ the first positive eigenvalue of the Neumann Laplacian.}
\label{fig: exponential decay Q=|x|^2}
\end{figure}

\FloatBarrier

\bibliographystyle{abbrv}
\bibliography{./fractional-porous-FE.bib}

@article {chen2022analysis,
    AUTHOR = {Chen, Li and Holzinger, Alexandra and J\"{u}ngel, Ansgar and
              Zamponi, Nicola},
     TITLE = {Analysis and mean-field derivation of a porous-medium equation
              with fractional diffusion},
   JOURNAL = {Comm. Partial Differential Equations},
  FJOURNAL = {Communications in Partial Differential Equations},
    VOLUME = {47},
      YEAR = {2022},
    NUMBER = {11},
     PAGES = {2217--2269},
      ISSN = {0360-5302},
   MRCLASS = {60H10 (35K65 35R11 60H30)},
  MRNUMBER = {4526892},
       DOI = {10.1080/03605302.2022.2118608},
       URL = {https://doi.org/10.1080/03605302.2022.2118608},
}

@article {bonito2015numerical,
    AUTHOR = {Bonito, Andrea and Pasciak, Joseph E.},
     TITLE = {Numerical approximation of fractional powers of elliptic
              operators},
   JOURNAL = {Math. Comp.},
  FJOURNAL = {Mathematics of Computation},
    VOLUME = {84},
      YEAR = {2015},
    NUMBER = {295},
     PAGES = {2083--2110},
      ISSN = {0025-5718},
   MRCLASS = {65N30 (65R20)},
  MRNUMBER = {3356020},
MRREVIEWER = {Igor Bock},
       DOI = {10.1090/S0025-5718-2015-02937-8},
       URL = {https://doi.org/10.1090/S0025-5718-2015-02937-8},
}

@article {barrett2012finite,
    AUTHOR = {Barrett, John W. and S\"{u}li, Endre},
     TITLE = {Finite element approximation of finitely extensible nonlinear
              elastic dumbbell models for dilute polymers},
   JOURNAL = {ESAIM Math. Model. Numer. Anal.},
  FJOURNAL = {ESAIM. Mathematical Modelling and Numerical Analysis},
    VOLUME = {46},
      YEAR = {2012},
    NUMBER = {4},
     PAGES = {949--978},
      ISSN = {2822-7840},
   MRCLASS = {35Q35 (65M60 76A05 82D60)},
  MRNUMBER = {2891476},
       DOI = {10.1051/m2an/2011062},
       URL = {https://doi.org/10.1051/m2an/2011062},
}

@article {barrett2011finite,
    AUTHOR = {Barrett, John W. and S\"{u}li, Endre},
     TITLE = {Finite element approximation of kinetic dilute polymer models
              with microscopic cut-off},
   JOURNAL = {ESAIM Math. Model. Numer. Anal.},
  FJOURNAL = {ESAIM. Mathematical Modelling and Numerical Analysis},
    VOLUME = {45},
      YEAR = {2011},
    NUMBER = {1},
     PAGES = {39--89},
      ISSN = {2822-7840},
   MRCLASS = {82D60 (35A35 35D30 35Q35 65M12 65M60 76A05 82C31)},
  MRNUMBER = {2781131},
MRREVIEWER = {Reinhard Farwig},
       DOI = {10.1051/m2an/2010030},
       URL = {https://doi.org/10.1051/m2an/2010030},
}

@article {caffarelli2016fractional,
    AUTHOR = {Caffarelli, Luis A. and Stinga, Pablo Ra\'{u}l},
     TITLE = {Fractional elliptic equations, {C}accioppoli estimates and
              regularity},
   JOURNAL = {Ann. Inst. H. Poincar\'{e} C Anal. Non Lin\'{e}aire},
  FJOURNAL = {Annales de l'Institut Henri Poincar\'{e} C. Analyse Non Lin\'{e}aire},
    VOLUME = {33},
      YEAR = {2016},
    NUMBER = {3},
     PAGES = {767--807},
      ISSN = {0294-1449},
   MRCLASS = {35R11 (35B45 35B65 46E35)},
  MRNUMBER = {3489634},
MRREVIEWER = {Mark Allen},
       DOI = {10.1016/j.anihpc.2015.01.004},
       URL = {https://doi.org/10.1016/j.anihpc.2015.01.004},
}

@article {bonito2017numerical,
    AUTHOR = {Bonito, Andrea and Pasciak, Joseph E.},
     TITLE = {Numerical approximation of fractional powers of regularly
              accretive operators},
   JOURNAL = {IMA J. Numer. Anal.},
  FJOURNAL = {IMA Journal of Numerical Analysis},
    VOLUME = {37},
      YEAR = {2017},
    NUMBER = {3},
     PAGES = {1245--1273},
      ISSN = {0272-4979},
   MRCLASS = {65J10 (26A33 65R10)},
  MRNUMBER = {3671494},
MRREVIEWER = {C. Ilioi},
       DOI = {10.1093/imanum/drw042},
       URL = {https://doi.org/10.1093/imanum/drw042},
}

@article {dubinskii1965weak,
    AUTHOR = {Dubinski\u{\i}, Ju. A.},
     TITLE = {Weak convergence for nonlinear elliptic and parabolic
              equations},
   JOURNAL = {Mat. Sb. (N.S.)},
  FJOURNAL = {Matematicheski\u{\i} Sbornik. Novaya Seriya},
    VOLUME = {67(109)},
      YEAR = {1965},
     PAGES = {609--642},
      ISSN = {0368-8666},
   MRCLASS = {35.63 (35.43)},
  MRNUMBER = {190546},
MRREVIEWER = {P. C. Fife},
}

@article {barrett2011existence,
    AUTHOR = {Barrett, John W. and S\"{u}li, Endre},
     TITLE = {Existence and equilibration of global weak solutions to
              kinetic models for dilute polymers {I}: {F}initely extensible
              nonlinear bead-spring chains},
   JOURNAL = {Math. Models Methods Appl. Sci.},
  FJOURNAL = {Mathematical Models and Methods in Applied Sciences},
    VOLUME = {21},
      YEAR = {2011},
    NUMBER = {6},
     PAGES = {1211--1289},
      ISSN = {0218-2025},
   MRCLASS = {82C31 (35Q30 35Q84 74A25 82C40 82D60 92C40)},
  MRNUMBER = {2819196},
MRREVIEWER = {Beno\^{i}t P. Desjardins},
       DOI = {10.1142/S0218202511005313},
       URL = {https://doi.org/10.1142/S0218202511005313},
}

@article {barrett2012dubinskii,
    AUTHOR = {Barrett, John W. and S\"{u}li, Endre},
     TITLE = {Reflections on {D}ubinski\u{\i}'s nonlinear compact embedding
              theorem},
   JOURNAL = {Publ. Inst. Math. (Beograd) (N.S.)},
  FJOURNAL = {Institut Math\'{e}matique. Publications. Nouvelle S\'{e}rie},
    VOLUME = {91(105)},
      YEAR = {2012},
     PAGES = {95--110},
      ISSN = {0350-1302},
   MRCLASS = {46B50 (35A23 46E40)},
  MRNUMBER = {2963813},
MRREVIEWER = {Ale\v{s} Nekvinda},
       DOI = {10.2298/PIM1205095B},
       URL = {https://doi.org/10.2298/PIM1205095B},
}

@article {choi2021classical,
    AUTHOR = {Choi, Young-Pil and Jeong, In-Jee},
     TITLE = {Classical solutions for fractional porous medium flow},
   JOURNAL = {Nonlinear Anal.},
  FJOURNAL = {Nonlinear Analysis. Theory, Methods \& Applications. An
              International Multidisciplinary Journal},
    VOLUME = {210},
      YEAR = {2021},
     PAGES = {Paper No. 112393, 13},
      ISSN = {0362-546X},
   MRCLASS = {35Q35},
  MRNUMBER = {4253949},
       DOI = {10.1016/j.na.2021.112393},
       URL = {https://doi.org/10.1016/j.na.2021.112393},
}

@article {carrillo2015exponential,
    AUTHOR = {Carrillo, J. A. and Huang, Y. and Santos, M. C. and V\'{a}zquez,
              J. L.},
     TITLE = {Exponential convergence towards stationary states for the 1{D}
              porous medium equation with fractional pressure},
   JOURNAL = {J. Differential Equations},
  FJOURNAL = {Journal of Differential Equations},
    VOLUME = {258},
      YEAR = {2015},
    NUMBER = {3},
     PAGES = {736--763},
      ISSN = {0022-0396},
   MRCLASS = {35K59 (26A33 35K65 35R11 74F10 76S05)},
  MRNUMBER = {3279352},
MRREVIEWER = {Kai Diethelm},
       DOI = {10.1016/j.jde.2014.10.003},
       URL = {https://doi.org/10.1016/j.jde.2014.10.003},
}

@article {caffarelli2010asymptotic,
    AUTHOR = {Caffarelli, Luis A. and V\'{a}zquez, Juan Luis},
     TITLE = {Asymptotic behaviour of a porous medium equation with
              fractional diffusion},
   JOURNAL = {Discrete Contin. Dyn. Syst.},
  FJOURNAL = {Discrete and Continuous Dynamical Systems. Series A},
    VOLUME = {29},
      YEAR = {2011},
    NUMBER = {4},
     PAGES = {1393--1404},
      ISSN = {1078-0947},
   MRCLASS = {35R11 (35B40 35J87 35K15 35K55 35K65 76S05)},
  MRNUMBER = {2773189},
MRREVIEWER = {Philippe Lauren\c{c}ot},
       DOI = {10.3934/dcds.2011.29.1393},
       URL = {https://doi.org/10.3934/dcds.2011.29.1393},
}

@article {caffarelli2010nonlinear,
    AUTHOR = {Caffarelli, Luis A. and Vazquez, Juan Luis},
     TITLE = {Nonlinear porous medium flow with fractional potential
              pressure},
   JOURNAL = {Arch. Ration. Mech. Anal.},
  FJOURNAL = {Archive for Rational Mechanics and Analysis},
    VOLUME = {202},
      YEAR = {2011},
    NUMBER = {2},
     PAGES = {537--565},
      ISSN = {0003-9527},
   MRCLASS = {76S05 (35A01 35D30 35K59 35R11)},
  MRNUMBER = {2847534},
MRREVIEWER = {Stelian Ion},
       DOI = {10.1007/s00205-011-0420-4},
       URL = {https://doi.org/10.1007/s00205-011-0420-4},
}

@article {filip2018rational,
    AUTHOR = {Filip, Silviu-Ioan and Nakatsukasa, Yuji and Trefethen, Lloyd
              N. and Beckermann, Bernhard},
     TITLE = {Rational minimax approximation via adaptive barycentric
              representations},
   JOURNAL = {SIAM J. Sci. Comput.},
  FJOURNAL = {SIAM Journal on Scientific Computing},
    VOLUME = {40},
      YEAR = {2018},
    NUMBER = {4},
     PAGES = {A2427--A2455},
      ISSN = {1064-8275},
   MRCLASS = {41A20 (65D15)},
  MRNUMBER = {3840899},
MRREVIEWER = {J. Szabados},
       DOI = {10.1137/17M1132409},
       URL = {https://doi.org/10.1137/17M1132409},
}

@book {BreSco94,
    AUTHOR = {Brenner, Susanne C. and Scott, L. Ridgway},
     TITLE = {The mathematical theory of finite element methods},
    SERIES = {Texts in Applied Mathematics},
    VOLUME = {15},
 PUBLISHER = {Springer-Verlag},
   ADDRESS = {New York},
      YEAR = {1994},
     PAGES = {xii+294},
      ISBN = {0-387-94193-2},

}

@article {caffarelli2013regularity,
    AUTHOR = {Caffarelli, Luis A. and Soria, Fernando and V\'{a}zquez, Juan Luis},
     TITLE = {Regularity of solutions of the fractional porous medium flow},
   JOURNAL = {J. Eur. Math. Soc. (JEMS)},
  FJOURNAL = {Journal of the European Mathematical Society (JEMS)},
    VOLUME = {15},
      YEAR = {2013},
    NUMBER = {5},
     PAGES = {1701--1746},
      ISSN = {1435-9855},
   MRCLASS = {35R11 (35B65 35K57 35K65 76S05)},
  MRNUMBER = {3082241},
MRREVIEWER = {Siegfried Carl},
       DOI = {10.4171/JEMS/401},
       URL = {https://doi.org/10.4171/JEMS/401},
}

@book{Lumer1998,
  editor    = {Lumer, G{\"u}nter and Weis, Lutz},
  title     = {{Gaussian estimates for second order elliptic divergence operators on Lipschitz
and C1 domains in: Evolution Equations and Their Applications in Physical and Life Sciences}},
  publisher = {CRC Press},
  year      = {1998},
  isbn      = {978-0824790103},
  address   = {Boca Raton, FL},
  note      = {Proceedings of the 6th International Conference on Evolution Equations, held in Bad Herrenalb, Germany}
}

@book {ElMaati2005,
    AUTHOR = {Ouhabaz, El Maati},
     TITLE = {Analysis of heat equations on domains},
    SERIES = {London Mathematical Society Monographs Series},
    VOLUME = {31},
 PUBLISHER = {Princeton University Press, Princeton, NJ},
      YEAR = {2005},
     PAGES = {xiv+284},
      ISBN = {0-691-12016-1},
   MRCLASS = {35-02 (35J70 35K20 47D06 47D07 47F05)},
  MRNUMBER = {2124040},
MRREVIEWER = {Sergey\ G.\ Pyatkov},
}

@book {Martinez2001,
    AUTHOR = {Mart\'inez Carracedo, Celso and Sanz Alix, Miguel},
     TITLE = {The theory of fractional powers of operators},
    SERIES = {North-Holland Mathematics Studies},
    VOLUME = {187},
 PUBLISHER = {North-Holland Publishing Co., Amsterdam},
      YEAR = {2001},
     PAGES = {xii+365},
      ISBN = {0-444-88797-0},
   MRCLASS = {47A60 (26A33 35K90 46M35 47D06)},
  MRNUMBER = {1850825},
MRREVIEWER = {Alberto\ Venni},
}

@book {Reed1978,
    AUTHOR = {Reed, Michael and Simon, Barry},
     TITLE = {Methods of modern mathematical physics. {IV}. {A}nalysis of
              operators},
 PUBLISHER = {Academic Press [Harcourt Brace Jovanovich, Publishers], New
              York-London},
      YEAR = {1978},
     PAGES = {xv+396},
      ISBN = {0-12-585004-2},
   MRCLASS = {47-02 (81.47)},
  MRNUMBER = {493421},
MRREVIEWER = {P.\ R.\ Chernoff},
}

@book {Brezis2011,
    AUTHOR = {Brezis, Haim},
     TITLE = {Functional analysis, {S}obolev spaces and partial differential
              equations},
    SERIES = {Universitext},
 PUBLISHER = {Springer, New York},
      YEAR = {2011},
     PAGES = {xiv+599},
      ISBN = {978-0-387-70913-0},
   MRCLASS = {35-01 (46-01 46E35 46N20 47F05)},
  MRNUMBER = {2759829},
MRREVIEWER = {Vicen\c tiu\ D.\ R\u adulescu},
}

@article {Caffarelli2010,
    AUTHOR = {Caffarelli, Luis A. and Vasseur, Alexis F.},
     TITLE = {The {D}e {G}iorgi method for regularity of solutions of
              elliptic equations and its applications to fluid dynamics},
   JOURNAL = {Discrete Contin. Dyn. Syst. Ser. S},
  FJOURNAL = {Discrete and Continuous Dynamical Systems. Series S},
    VOLUME = {3},
      YEAR = {2010},
    NUMBER = {3},
     PAGES = {409--427},
      ISSN = {1937-1632,1937-1179},
   MRCLASS = {35B65 (35Q30 76D03)},
  MRNUMBER = {2660718},
MRREVIEWER = {Luca\ Capogna},
       DOI = {10.3934/dcdss.2010.3.409},
       URL = {https://doi.org/10.3934/dcdss.2010.3.409},
}

@article {Fronzoni2025,
    AUTHOR = {Carrillo, Jos\'e{} A. and Fronzoni, Stefano and S\"uli, Endre},
     TITLE = {Finite element scheme for the fractional porous medium
              equation with fractional pressure},
   JOURNAL = {Numer. Math.},
  FJOURNAL = {Numerische Mathematik},
    VOLUME = {157},
      YEAR = {2025},
    NUMBER = {5},
     PAGES = {1537--1614},
      ISSN = {0029-599X,0945-3245},
   MRCLASS = {65M60 (35K55 35R11)},
  MRNUMBER = {4972276},
       DOI = {10.1007/s00211-025-01486-3},
       URL = {https://doi.org/10.1007/s00211-025-01486-3},
}

@article{Fronzoni2024,
      title={A minimax method for the spectral fractional Laplacian and related evolution problems}, 
      author={José A. Carrillo and Stefano Fronzoni and Yuji Nakatsukasa and Endre Süli},
      year={2025},
      eprint={2505.20560},
      archivePrefix={arXiv},
journal = {arXiv},
note = {preprint arXiv:2505.20560, submitted for publication},
      primaryClass={math.NA},
      url={https://arxiv.org/abs/2505.20560}, 
}

@incollection {Vazquez2017,
    AUTHOR = {V\'{a}zquez, Juan Luis},
     TITLE = {The mathematical theories of diffusion: nonlinear and
              fractional diffusion},
 BOOKTITLE = {Nonlocal and nonlinear diffusions and interactions: new
              methods and directions},
    SERIES = {Lecture Notes in Math.},
    VOLUME = {2186},
     PAGES = {205--278},
 PUBLISHER = {Springer, Cham},
      YEAR = {2017},
   MRCLASS = {35K57 (35R11)},
  MRNUMBER = {3588125},
MRREVIEWER = {Jeffrey R. Anderson},
}

@article {Liu2025,
    AUTHOR = {Liu, Shengquan and Zheng, Jiashan},
     TITLE = {A new result for global solvability of a
              {K}eller-{S}egel-{N}avier-{S}tokes system with nonlinear
              diffusion and matrix-valued sensitivities in three dimensions},
   JOURNAL = {J. Differential Equations},
  FJOURNAL = {Journal of Differential Equations},
    VOLUME = {426},
      YEAR = {2025},
     PAGES = {721--759},
      ISSN = {0022-0396,1090-2732},
   MRCLASS = {35Q92 (35K51 35K55 35Q35 92C17)},
  MRNUMBER = {4858609},
       DOI = {10.1016/j.jde.2025.01.071},
       URL = {https://doi.org/10.1016/j.jde.2025.01.071},
}

@article {Li2018,
    AUTHOR = {Li, Xianping},
     TITLE = {Anisotropic mesh adaptation for finite element solution of
              anisotropic porous medium equation},
   JOURNAL = {Comput. Math. Appl.},
  FJOURNAL = {Computers \& Mathematics with Applications. An International
              Journal},
    VOLUME = {75},
      YEAR = {2018},
    NUMBER = {6},
     PAGES = {2086--2099},
      ISSN = {0898-1221,1873-7668},
   MRCLASS = {65M60 (65M50 76S05)},
  MRNUMBER = {3775105},
       DOI = {10.1016/j.camwa.2017.08.005},
       URL = {https://doi.org/10.1016/j.camwa.2017.08.005},
}

@article {Zhi2022,
    AUTHOR = {Zhi, Yuan and Zhan, Huashui},
     TITLE = {The well-posedness problem of an anisotropic porous medium
              equation with a convection term},
   JOURNAL = {J. Inequal. Appl.},
  FJOURNAL = {Journal of Inequalities and Applications},
      YEAR = {2022},
     PAGES = {Paper No. 108, 22},
      ISSN = {1029-242X},
   MRCLASS = {35L65 (35B05 35K55)},
  MRNUMBER = {4468564},
       DOI = {10.1186/s13660-022-02847-4},
       URL = {https://doi.org/10.1186/s13660-022-02847-4},
}

@article {Burrage2012,
    AUTHOR = {Burrage, Kevin and Hale, Nicholas and Kay, David},
     TITLE = {An efficient implicit {FEM} scheme for fractional-in-space
              reaction-diffusion equations},
   JOURNAL = {SIAM J. Sci. Comput.},
  FJOURNAL = {SIAM Journal on Scientific Computing},
    VOLUME = {34},
      YEAR = {2012},
    NUMBER = {4},
     PAGES = {A2145--A2172},
      ISSN = {1064-8275,1095-7197},
   MRCLASS = {65M60 (26A33 35K57 35R11 45K05)},
  MRNUMBER = {2970400},
MRREVIEWER = {Mohammad\ Asadzadeh},
       DOI = {10.1137/110847007},
       URL = {https://doi.org/10.1137/110847007},
}

@article {Hale2008,
    AUTHOR = {Hale, Nicholas and Higham, Nicholas J. and Trefethen, Lloyd
              N.},
     TITLE = {Computing {${\bf A}^\alpha,\ \log({\bf A})$}, and related
              matrix functions by contour integrals},
   JOURNAL = {SIAM J. Numer. Anal.},
  FJOURNAL = {SIAM Journal on Numerical Analysis},
    VOLUME = {46},
      YEAR = {2008},
    NUMBER = {5},
     PAGES = {2505--2523},
      ISSN = {0036-1429,1095-7170},
   MRCLASS = {65F30},
  MRNUMBER = {2421045},
MRREVIEWER = {Raffaella\ Pavani},
       DOI = {10.1137/070700607},
       URL = {https://doi.org/10.1137/070700607},
}

@article {Zhang2025,
    AUTHOR = {Zhang, Lei-Hong and Yang, Linyi and Yang, Wei Hong and Zhang,
              Ya-Nan},
     TITLE = {A convex dual problem for the rational minimax approximation
              and {L}awson's iteration},
   JOURNAL = {Math. Comp.},
  FJOURNAL = {Mathematics of Computation},
    VOLUME = {94},
      YEAR = {2025},
    NUMBER = {355},
     PAGES = {2457--2494},
      ISSN = {0025-5718,1088-6842},
   MRCLASS = {65D15 (41A20 41A50 90C46)},
  MRNUMBER = {4919567},
       DOI = {10.1090/mcom/4021},
       URL = {https://doi.org/10.1090/mcom/4021},
}

@article {Harizanov2018,
    AUTHOR = {Harizanov, S. and Lazarov, R. and Margenov, S. and Marinov, P.
              and Vutov, Y.},
     TITLE = {Optimal solvers for linear systems with fractional powers of
              sparse {SPD} matrices},
   JOURNAL = {Numer. Linear Algebra Appl.},
  FJOURNAL = {Numerical Linear Algebra with Applications},
    VOLUME = {25},
      YEAR = {2018},
    NUMBER = {5},
     PAGES = {e2167, 24},
      ISSN = {1070-5325,1099-1506},
   MRCLASS = {65F99},
  MRNUMBER = {3859150},
       DOI = {10.1002/nla.2167},
       URL = {https://doi.org/10.1002/nla.2167},
}
 
\renewcommand{\appendixname}{Appendix}
\addappendix

\begin{appendices}

\section{Fractional elliptic operators}

Let $\LL$ be an operator such as the one described in the Introduction, more precisely as \eqref{mytypeofoperator}, with $-\text{div}(A(x) \nabla (\cdot))$ as the principal part of the operator. Let $\LN$ be the operator paired with a Neumann boundary condition.
For a bounded open Lipschitz domain $\Omega \subset \mathbb{R}^{d}$, $s \in (0,1)$, and $f \in \mathcal{H}_{\diamond}^{-\s}(\Omega)$, consider the following fractional equation with a Neumann boundary condition:
\begin{equation} \label{general FracPoiNeu}
\LN^{\s} u = f  \quad \textrm{in } \Omega, \quad A(x) \nabla u \cdot n = 0 \quad \textrm{on }  \partial \Omega.
\end{equation}
By a \textit{weak solution} to \eqref{general FracPoiNeu} we mean a function $u \in \Hsd(\Omega)$, $s \in (0,1)$, such that 
\[ \int_\Omega \LN^{s/2}u(x) \, \LN^{s/2}v(x) \dx = \langle f , v \rangle \quad \forall v \in \mathcal{H}_\diamond^s(\Omega),\]
where $\langle \cdot, \cdot \rangle$ denotes the duality pairing between $\mathcal{H}_{\diamond}^{-s}(\Omega)$ and $\mathcal{H}_{\diamond}^s(\Omega)$, with $\mathcal{H}_{\diamond}^{s}(\Omega) \hookrightarrow \mathcal{H}^{0}(\Omega)\equiv \mathcal{H}^{0}(\Omega)' \hookrightarrow\mathcal{H}_{\diamond}^{-s}(\Omega)$ and $\mathcal{H}^{0}_{\diamond}(\Omega)= L^2_{\diamond}(\Omega)$. In particular if $f \in L^2_{\diamond}(\Omega)$, then $\langle f, v \rangle = (f,v)$ for all $v \in \mathcal{H}^s_{\diamond}(\Omega)$, $s \in [0,1)$.

\begin{theo}[Existence and uniqueness of a weak solution] \label{general FracPoiExUniq}
Suppose that $\Omega\subset \mathbb{R}^d$ is a bounded open Lipschitz domain, $s \in (0,1)$, and $f \in \mathcal{H}^{-\s}_{\diamond}(\Omega)$. Then, the boundary-value problem \eqref{general FracPoiNeu} has a unique weak solution $u \in \mathcal{H}^s_{\diamond}(\Omega)$, and 
\begin{equation} \label{general StabFracNeu}
\|u\|_{\mathcal{H}_{\diamond}^s(\Omega)} \leq \|f\|_{\mathcal{H}_{\diamond}^{-\s}(\Omega)}.
\end{equation}
\begin{proof}
By the assumptions on $\LN$, the existence of a unique weak solution $u \in \mathcal{H}^s_\diamond(\Omega)$ is a direct consequence of the Lax--Milgram lemma applied to the variational problem: find $u \in \mathcal{H}^s_\diamond(\Omega)$ such that $a(u,v) = \ell(v)$ for all $v \in \mathcal{H}_\diamond^s(\Omega)$, with the bilinear form 
\[ a(w,v):= \int_\Omega (\LL^{s/2}w(x)) \, \LL^{s/2}v(x)) \dx,  \quad w, v \in \mathcal{H}^s_\diamond(\Omega), \]
and the linear functional $\ell(v):= \langle f , v \rangle$, $v \in \mathcal{H}^s_\diamond(\Omega)$. The stated stability inequality follows by noting that 
\[ \|u\|^2_{\mathcal{H}_\diamond^s(\Omega)} = a(u,u) = \langle f , u \rangle \leq \|f\|_{\mathcal{H}_\diamond^{-s}(\Omega)} \|u \|_{\mathcal{H}_\diamond^s(\Omega)}.\]
That completes the proof of the lemma.  
\end{proof}
\end{theo}

We note in passing that the following function space interpolation inequality holds. 

\begin{lemma} \label{general InterpFracSobLemma}
Let $\Omega\subset \mathbb{R}^d$ be a bounded open Lipschitz domain, $s \in (0,1)$, and $u \in \mathcal{H}^1_{\diamond}(\Omega)$. Then,
\begin{equation*}
    \|u\|_{\Hsd(\Omega)} \leq \|u\|_{L^{2}(\Omega)}^{1-\s} \; \| u \|_{\mathcal{H}^1(\Omega)}^{\s}.
\end{equation*}
\begin{proof}
The assertion of the lemma follows by applying H\"{o}lder's inequality and the definition \eqref{general Hs space}. Indeed, without loss of generality, let the positive eigenvalues corresponding to $\LN$ be $\lambda_k$, $k=1,2, \ldots$, we have for $\theta \in (0,1)$ that
\begin{align*}
    \|u\|_{\Hsd(\Omega)}^{2} &= \sum_{k=1}^{\infty} \lambda_{k}^{\s} u_{k}^{2} = \sum_{k=1}^{\infty} |u_{k}|^{\theta} \lambda_{k}^{\s} |u_{k}|^{2 - \theta} \leq \Bigg( \sum_{k=1}^{\infty} |u_{k}|^{\theta p}\Bigg)^{\frac{1}{p}} \Bigg( \sum_{k=1}^{\infty} \lambda_{k}^{\s q} |u_{k}|^{(2-\theta)q} \Bigg)^{\frac{1}{q}},  
\end{align*}
for $p, q \in (1, \infty)$ such that $\frac{1}{p} + \frac{1}{q} = 1$. The stated inequality then follows by taking $p = \frac{2}{\theta}$, $q = \frac{2}{2 - \theta}$ and $\theta = 2(1 - \s)$, whereby $\theta p = 2$, $sq=1$, $(2-\theta)q=2$, $1/p=\theta/2=1-s$, and $1/q=s$.
\end{proof}
\end{lemma}

In passing, we notice how the following inequality holds for functions $u \in \mathcal{H}^{s}_{\diamond}(\Omega)$ and $s \in (0,1)$:  
\begin{equation} \label{FracPoincareUse}
    \|u\|_{L^2_\diamond(\Omega)} 
    \leq \lambda_1^{-s/2} \| u \|_{\mathcal{H}^{s}_{\diamond}(\Omega)}.
\end{equation}
This inequality, in conjunction with Theorem \ref{general FracPoiExUniq}, ensures the $L^{2}_\diamond(\Omega)$ norm stability of the weak solution to the fractional equation \eqref{general FracPoiNeu}, based on the spectral definition of $\LN$, for $f \in L^{2}_\diamond(\Omega)$

We now prove some auxiliary results, which concern the case when the right-hand side datum of equation \eqref{general FracPoiNeu} has $L^{\infty}(\Omega)$ regularity. Let us denote  by $L^{\infty}_{\ast}(\Omega)$  the space of functions in $L^{\infty}(\Omega)$ with zero integral average over $\Omega$

\begin{lemma}
\label{LInftyPoiStabFracOp}
Suppose that $\Omega\subset \mathbb{R}^d$ is a bounded open Lipschitz domain, $s \in (0,1)$, and $f \in L^{\infty}(\Omega) \cap L^2_\diamond(\Omega)$. Let $u \in \mathcal{H}^{s}_\diamond(\Omega) $ be the unique solution of \eqref{general FracPoiNeu}. Then $u \in L^{\infty}(\Omega)$ and
\begin{equation} \label{LInftyPoiStab}
    \| u \|_{L^{\infty}(\Omega)} \leq K_{\Omega, s, d} \| f \|_{L^{\infty}(\Omega)},
\end{equation}
for a positive constant $K_{\Omega, s, d}$, depending only on $\Omega, s$ and $d$.
\end{lemma}
\begin{proof}
We have that $u=\LN^{-s} f$ and $\LN^{-s}$ can be expressed in terms of the heat semigroup $\{ {\mathrm e}^{-t\LN} \}_{t \geq 0}$, through the formula
\begin{equation} \label{heatsemigroupformulanegative}
\LN^{-s} u(x) = \frac{1}{\Gamma(s)} \int_{0}^{\infty} {\mathrm e}^{-t\LN} u(x) \frac{1}{t^{1-s}} \, \dt,
\end{equation}
see \cite{Martinez2001}.
Thus, starting from \eqref{heatsemigroupformulanegative} we have that
\begin{equation*}
     \LN^{-s} f(x) = \frac{1}{\Gamma(s)} \int_0^\infty \int_\Omega W_t(x,y) f(y) \dy  \frac{1}{t^{1-s}} \dt,
\end{equation*}
where $W_t(x,y)$ is the heat kernel of $\LN$.

Under our assumptions on $\LN$, by the Gaussian upper bound \eqref{Gaussianupperbound} we have 
\begin{align*}
    \int_0^\infty  \int_\Omega W_t(x,y) \dy \frac{1}{t^{1-s}} \dt &\leq \int_0^\infty  \int_\Omega c_1 \frac{\ee^{-|x-y|^2/(c_2 t)}}{t^{\frac{d}{2}}} \dy \frac{1}{t^{1-s}} \dt,
\end{align*}
for positive constants $c_1, c_2 > 0$. We then have  
\begin{align*}
 \int_{0}^{\infty} \frac{1}{t^{\frac{d}{2} + 1 - s}} {\mathrm e}^{-|x-y|^{2}/t} \, \dt &=\frac{1}{|x-y|^{d-2s}}\int_{0}^{\infty} {\mathrm e}^{-r} r^{-1-s+\frac{d}{2}} \, \textrm{d}r \\
 &= \Gamma\Big(\frac{d}{2} - s\Big) \frac{1}{|x-y|^{d-2s}}, 
\end{align*}
where in the passage from the first line to the second line we used the change of variable $r = |x-y|^{2}/t$.

Therefore
\begin{align*}
    |\LN^{-s}f(x)| \leq C_{s,d} \| f \|_{L^{\infty}(\Omega)} \int_\Omega \frac{1}{|x-y|^{d-2s}} \dy, 
\end{align*}
where $C_{s,d}$ is a positive constant depending  on $s,d$, and we can bound the right-hand side using the fact that the function $1/|z|^{d - 2s}$ is integrable for $s \in (0,1)$ over any bounded ball in $\mathbb{R}^d$. 
\end{proof}

\section{Dubinski\u{\i}'s compactness theorem} \label{dubinski appendix}
Let $\mathcal {A}$ be a linear space over the field $\mathbb{R}$ of real numbers, which will be referred as \textit{ambient space}, and suppose that $\mathcal{C}$ is a subset of $\mathcal{A}$ such that, 
\begin{equation} \label{PropSemiNorm}
    \textrm{for all } \varphi \in \mathcal{C} \textrm{ and } c \in \mathbb{R}_{\geq 0}, \quad c\,\varphi \in \mathcal{C}.
\end{equation}
Suppose further that each element $\varphi$ of a set $\mathcal{C}$ with property (\ref{PropSemiNorm}) is assigned a certain real number, denoted by $[\varphi]_{\mathcal{C}}$, such that 

\begin{itemize}
\item[$1.$] $[\varphi]_{\mathcal{C}} \geq 0$, and $[\varphi]_{\mathcal{C}}=0$ if, and only if, $\varphi=0$;
\item[$2.$] $[c \varphi]_{\mathcal{C}} = c[\varphi]_{\mathcal{C}}$ for all $c \in \mathbb{R}_{\geq 0}$.
\end{itemize}
A subset $\mathcal{C}$ satisfying \eqref{PropSemiNorm}, equipped with $[\cdot]$ satisfying these two properties is referred to as a \textit{seminormed set}.

\begin{figure}[H]
    \centering
    \includegraphics[width = 0.4\textwidth]{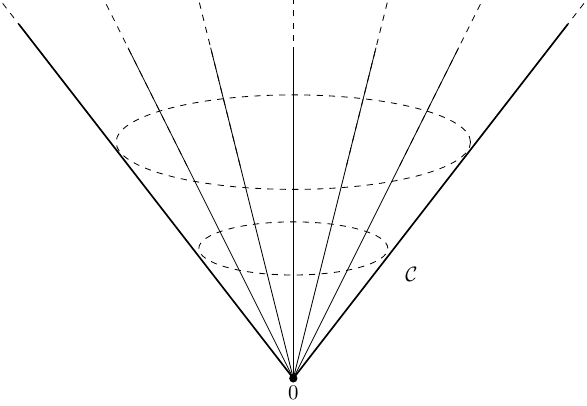}
    \caption{Seminormed set.}
    \label{fig:seminormed set}
\end{figure}

\begin{theo} [Dubinski\u{\i}'s compactness theorem; c.f. \cite{barrett2012dubinskii}] \label{Dubinsky}
    Suppose that $\mathcal{A}_{0}$ and $\mathcal{A}_{1}$ are Banach spaces, $\mathcal{A}_{0} \hookrightarrow \mathcal{A}_{1}$ (i.e., $\mathcal{A}_{0}$ is continuously embedded in $\mathcal{A}_{1}$), and $\mathcal{C}$ is a seminormed set contained in $\mathcal{A}_{0}$ such that $\mathcal{C}$ is compactly embedded in $\mathcal{A}_{0}$. Consider the set 
    \[ \mathcal{Y} := \Bigg \{ \varphi: [0, T] \to \mathcal{C}: [\varphi]_{L^{p}(0, T; \mathcal{C})} + \bigg\| \frac{\mathrm{d} \varphi}{\dt}\bigg\|_{L^{p_{1}}(0, T; \mathcal{A}_{1})} < \infty \Bigg \}, \]
    where $1\leq p \leq \infty$, $1\leq p_{1} \leq \infty$, $\| \cdot \|_{\mathcal{A}_{1}}$ is the norm of $\mathcal{A}_{1}$ and $\frac{\mathrm{d}\varphi}{\dt}$ is understood in the sense of $\mathcal{A}_{1}$-valued distributions on the open interval $(0, T)$. Then $\mathcal{Y}$, with 
    \[ [\varphi]_{\mathcal{Y}} := [\varphi]_{L^{p}(0, T; \mathcal{C})} + \bigg\| \frac{\mathrm{d} \varphi}{\dt} \bigg\|_{L^{p_{1}}(0, T; \mathcal{A}_{1})}, \]
    is a seminormed set in $L^{p}(0, T; \mathcal{A}_{0}) \cap W^{1, p_{1}}(0, T; \mathcal{A}_{1})$, and $\mathcal{Y}$ is compactly embedded in $L^{p}(0, T; \mathcal{A}_{0})$ if either $1 \leq p \leq \infty$ and $1< p_{1} < \infty$, or if $1 \leq p < \infty$ and $p_{1}=1$.
\end{theo}

The example of our application of  Dubinski\u{\i}'s theorem, as in \cite{barrett2011existence, Fronzoni2025}, is by selecting 
\[ \mathcal{A}_{0} = L^{1}(\Omega), \textrm{ with the usual Lebesgue norm } \|\varphi\|_{\mathcal{A}_{0}} := \int_{\Omega} |\varphi| \dx \]
and 
\[ \mathcal{C} = \Bigg \{ \phi \in \mathcal{A}_{0}: \varphi \geq 0 \textrm{ with } \int_{\Omega}\big| \nabla \sqrt{\varphi} \big|^{2} \dx < \infty \Bigg \};\]
and, for $\varphi \in \mathcal{C}$, we define 
\[ [\varphi]_{\mathcal{C}} := \|\varphi\|_{\mathcal{A}_{0}} + \int_{\Omega}\big| \nabla \sqrt{\varphi} \big|^{2} \dx.\]
It is a straightforward matter to confirm that (\ref{PropSemiNorm}) and properties 1. and 2. stated above hold, rendering $\mathcal{C}$ a seminormed subset of the ambient space $\mathcal{A}_{0}=L^1(\Omega)$. Finally we put 
\[ \mathcal{A}_{1} = H^{-\beta}(\Omega):= [H^{\beta}(\Omega)]'\quad \mbox{with $\beta = d + 1$}, \]
equipped with the dual norm $\|\varphi\|_{\mathcal{A}_{1}} := \|\varphi\|_{H^{-\beta}(\Omega)}$. The choice of $\mathcal{A}_{1}$ is motivated in \cite{Fronzoni2025}. This set of choices satisfies the conditions of Theorem \ref{Dubinsky}, with $\mathcal{C} \hookrightarrow \mathcal{A}_{0}$ compactly and $\mathcal{A}_{0} \hookrightarrow \mathcal{A}_{1}$
(cf. \cite{barrett2011existence}).

\end{appendices}

\end{document}